\documentclass[a4paper,11pt,reqno]{amsart}

\usepackage{fullpage}

\usepackage{amssymb}
\usepackage[all]{xy}

\usepackage{hyperref}

\numberwithin{equation}{section}

\newtheorem{theorem}{Theorem}[section]
\newtheorem{lemma}[theorem]{Lemma}
\newtheorem{proposition}[theorem]{Proposition}
\newtheorem{corollary}[theorem]{Corollary}

\theoremstyle{definition}
\newtheorem{definition}[theorem]{Definition}

\theoremstyle{remark}
\newtheorem{example}[theorem]{Example}

\theoremstyle{remark}
\newtheorem{remark}[theorem]{Remark}

\newcommand\bp{\begin{proof}}
\newcommand\ep{\end{proof}}

\newcommand\C{\mathbb{C}}
\newcommand\FF{\mathbb{F}}
\newcommand\N{\mathbb{N}}
\newcommand\R{\mathbb{R}}
\newcommand\Z{\mathbb{Z}}

\newcommand\B{\mathcal{B}}
\newcommand\F{\mathcal{F}}
\newcommand\K{\mathcal{K}}
\newcommand\LL{\mathcal{L}}
\newcommand\OO{\mathcal{O}}
\newcommand{\TT}{\mathcal{T}}

\newcommand\AN{\mathcal{A}_{\mathrm{a}}}
\newcommand{\NT}{\mathcal{NT}}
\newcommand{\NTr}{\mathcal{T}^r}
\newcommand\NOs{\widetilde{\mathcal{NO}}}

\newcommand\Ad{\operatorname{Ad}}
\newcommand\Tr{\operatorname{Tr}}
\newcommand{\supp}{\operatorname{supp}}
\newcommand\id{\operatorname{id}}
\newcommand{\lsp}{\operatorname{span}}

\newcommand\eps{\varepsilon}

\begin{document}

\title{KMS states on Nica-Toeplitz C$^*$-algebras}

\date{July 16, 2018; small changes August 9, 2018}

\author{Zahra Afsar}
\address{School of Mathematics and Statistics, The University of Sydney, NSW 2006, Australia}
\email{zahra.afsar@sydney.edu.au}

\thanks{Z.A. was supported by grants DP150101595 and DP170101821 of the Australian Research Council.}

\author{Nadia S. Larsen}\address{Department of Mathematics, University of Oslo,
P.O. Box 1053 Blindern, N-0316 Oslo, Norway.}
\email{nadiasl@math.uio.no}

\author{Sergey Neshveyev}\address{Department of Mathematics, University of Oslo,
P.O. Box 1053 Blindern, N-0316 Oslo, Norway.}
\email{sergeyn@math.uio.no}

\begin{abstract}
Given a quasi-lattice ordered group $(G,P)$ and a compactly aligned product system~$X$ of essential C$^*$-correspondences over the monoid $P$, we show that there is a bijection between the gauge-invariant KMS$_\beta$-states on the Nica-Toeplitz algebra $\NT(X)$ of $X$ with respect to a gauge-type dynamics, on one side, and the tracial states on the coefficient algebra~$A$ satisfying a system (in general infinite) of inequalities, on the other. This strengthens and generalizes a number of results in the literature in several directions: we do not make any extra assumptions on $P$ and $X$, and our result can, in principle, be used to study KMS-states at any finite inverse temperature~$\beta$. Under fairly general additional assumptions we show that there is a critical inverse temperature~$\beta_c$ such that for $\beta>\beta_c$ all KMS$_\beta$-states are of Gibbs type, hence gauge-invariant, in which case we have a complete classification of KMS$_\beta$-states in terms of tracial states on $A$, while at $\beta=\beta_c$ we have a phase transition manifesting itself in the appearance of KMS$_\beta$-states that are not of Gibbs type. In the case of right-angled Artin monoids we show also that our system of inequalities for traces on $A$ can be reduced to a much smaller system, a finite one when the monoid is finitely generated. Most of our results generalize to arbitrary quasi-free dynamics on~$\NT(X)$.
\end{abstract}

\maketitle

\section*{Introduction}

Given a  C$^*$-algebra $A$ with a time evolution $\sigma=(\sigma_t)_{t\in\R}$, the Kubo--Martin--Schwinger (KMS) condition on states of $A$ gives a rigorous mathematical characterization of equilibrium of the system $(A,\sigma)$ at inverse temperatures $\beta\in\R$. For example, in the case of a full matrix algebra and the time evolution $\sigma_t=\Ad e^{it H}$, there is a unique KMS$_\beta$-state for every $\beta\in \R$, which is precisely the Gibbs state $\Tr(\cdot\,e^{-\beta H})/\Tr (e^{-\beta H})$. Pimsner's construction~\cite{pim} of C$^*$-algebras~$\TT(X)$ and~$\OO(X)$ associated to a C$^*$-correspondence $X$ provided a new fundament for studying large classes of C$^*$-algebras in a systematic way. The new framework proved fruitful also for studying KMS-states. After a number of examples of systems $(A,\sigma)$ had been analyzed, with ${A}$  a Cuntz-Pimsner algebra $\OO(X)$ or a Toeplitz-Pimsner algebra $\TT(X)$, and $\sigma$ a quasi-free dynamics, a theory of KMS-states for such systems was developed by Pinzari--Watatani--Yonetani~\cite{PWY} for full finite rank correspondences and the gauge dynamics and then by Laca--Neshveyev~\cite{lac-nes} for essential C$^*$-correspondences and general quasi-free dynamics.

In recent years there has been a growing interest in the structure of KMS-states on C$^*$-algebras associated with product systems of C$^*$-correspondences over quasi-lattice ordered monoids, which generalize Cuntz-Pimsner and Toeplitz-Pimsner algebras. The initial motivation came from the work of Laca--Raeburn~\cite{lac-rae} on the semigroup C$^*$-algebra of the monoid $\Z_+\rtimes \Z_+^\times$. The structure of KMS-states on this C$^*$-algebra, revealed in~\cite{lac-rae}, was impossible to explain using general results available at that time. Although such an explanation was soon provided by the groupoid picture~\cite{nes}, it was still natural to try to understand it from the point of view of C$^*$-correspondences. This led Hong--Larsen--Szymanski~\cite{HLS2} and Brownlowe--an~Huef--Laca--Raeburn~\cite{BanHLR} to undertake studies of KMS-states on  C$^*$-algebras of product systems. The approach in~\cite{HLS2} was in principle guided by~\cite{lac-nes}, although in the end it employed more direct methods using special and quite strong hypotheses imposed on the correspondences, involving existence of orthonormal bases.  A relaxation of the conditions on the correspondences to existence of special Parseval frames enabled Afsar--an~Huef--Raeburn~\cite{AaHR} to study equilibrium states on C$^*$-algebras of product systems over $\Z^n_+$ constructed out of families of $*$-commuting local homeomorphisms on a compact Hausdorff space. In a related direction, Kakariadis~\cite{kak-KMS-NP-Znplus} initiated a program to study KMS-states on Pimsner-type C$^*$-algebras arising from multivariable dynamics. C$^*$-algebras associated to $k$-graphs are another large class of examples closely related to product systems whose structure of KMS-states has been extensively studied in the recent years starting with \cite{aHLRS}. Very recently, the analysis of KMS-states on the Toeplitz algebras of finite $k$-graphs has  attained a very satisfactory form in Christensen's work~\cite{C} using the groupoid approach. Interesting results on KMS-states of the Toeplitz algebras of quasi-lattice ordered monoids, which can be considered as C$^*$-algebras associated with the simplest rank one product systems, have been obtained in the work of Bruce--Laca--Ramagge--Sims~\cite{BLRS}.

This project started from the natural question whether it is possible to extend the theory developed in~\cite{lac-nes} to algebras associated with product systems. The payoff of our effort is a powerful machinery that clarifies and strengthens the existing results and, hopefully, can guide future investigations.

In order to formulate our motivation and results more precisely, let us first give some details of the theory developed in~\cite{lac-nes} in the case of the gauge dynamics. Assume we are given an essential C$^*$-correspondence $X$ over a C$^*$-algebra $A$. The induction of traces via $X$ defines a map $F_X$ on the positive traces on $A$, which is an analogue of the dual Ruelle transfer operator in dynamical systems. Then for every $\beta\in \R$, the map $\phi\mapsto\phi|_A$ defines a bijection between the KMS$_\beta$-states on $\TT(X)$ with respect to the gauge dynamics and the tracial states $\tau$ on $A$ such~that
\begin{equation}\label{eq:LN0}
e^{-\beta}F_X\tau\le\tau.\tag{$*$}
\end{equation}

As was shown in~\cite{lac-nes}, similarly to the potential theory in classical probability, any trace~$\tau$ satisfying~\eqref{eq:LN0} decomposes into a \emph{finite} part $\tau_f=\sum^\infty_{n=0}e^{-n\beta}F_X^n\tau_0$ (``potential'') and an \emph{infinite} part $\tau_\infty$ such that $e^{-\beta}F_X\tau_\infty=\tau_\infty$ (``harmonic element''). This can also be thought of as an analogue of the Wold decomposition for isometries. The KMS$_\beta$-states corresponding to traces of finite type are easy to understand. They are obtained by restriction from \emph{generalized Gibbs states} on~$\LL(\F(X))$, which are constructed similarly to the classical Gibbs states on full matrix algebras using induced traces on $\LL(\F(X))$. Therefore the main content of the above description of KMS-states is the statement that the KMS$_\beta$-states of infinite type, that is, those without a Gibbs-type component, are in a one-to-one correspondence with the tracial states such that $e^{-\beta}F_X\tau=\tau$.

Now suppose that $X=(X_p)_{p\in P}$ is a product system of C$^*$-correspondences over a quasi-lattice ordered monoid $P$. Then we obtain a collection of generalized dual transfer operators~$F_p$, one for every correspondence $X_p$. Assuming that $X$ is compactly aligned, we consider the Nica-Toeplitz C$^*$-algebra $\NT(X)$ associated with $X$ and the gauge-type time evolution $\sigma$ on it defined by a homomorphism $N\colon P\to(0,+\infty)$. Similarly to the case of a  single correspondence, it is not difficult to see then that every tracial state $\tau$ on the coefficient algebra $A$ of the form $\tau=\sum_{p\in P}N(p)^{-\beta}F_p\tau_0$ extends to a Gibbs-type state on $\NT(X)$. Therefore the starting point of our investigation can be phrased as the following questions:

\begin{itemize}
\item[(1)] Is it possible to characterize the tracial states on $A$ that are obtained from the KMS-states on $\NT(X)$ having no Gibbs-type components?
\item[(2)] Are there conditions that guarantee that all KMS-states are of Gibbs type? When can other KMS-states appear?
\end{itemize}

In this paper we will give a complete answer to question (1) and obtain quite general results answering the questions in (2).

\smallskip

We now describe in more detail the contents of the paper. In Section~\ref{sec:prelim} we collect a number of general results and definitions on quasi-lattice ordered groups, product systems and their associated C$^*$-algebras, KMS-states and induced traces, and generalized dual Ruelle transfer operators.

In Section~\ref{sec:restriction} we consider a system $(\NT(X), \sigma)$ as described above and show that any
KMS-state on $\NT(X)$, when restricted to the coefficient algebra $A$, gives a trace that satisfies a system of inequalities, see Theorem~\ref{thm:subinvariance}. We then go on to discuss the generalized Gibbs states and explain why our system of inequalities does not give interesting information for them.

In Section~\ref{sec:abelian} we consider the free abelian monoids $\Z^n_+$  for $n\geq 1$ and show that any trace on the coefficient algebra $A$ satisfying the inequalities from Section~\ref{sec:restriction} can be extended to a gauge-invariant KMS-state on $\NT(X)$. The reader interested in general quasi-lattice ordered monoids can safely skip this section, but there are a number of reasons to treat the monoids~$\Z^n_+$ separately: this is a much studied case in the literature, see~\cite{AaHR,aHLRS,kak-KMS-NP-Znplus,kak} (as well as \cite{BanHLR, HLS2}, where the monoid is $\oplus^\infty_{k=1}\Z_+$, identified with $\Z_+^\times$), and at the same time it allows for a very short proof based on a perturbation trick from~\cite{lac-nes}.

In Section~\ref{sec:graded} we prove an auxiliary result on positivity of functionals on the core subalgebra of~$\NT(X)$. The result is formulated and proved in a more abstract setting of \emph{quasi-lattice graded C$^*$-algebras}, which might be useful in other contexts.

In Section~\ref{sec:general} we prove  that a trace satisfying the inequalities from Section~\ref{sec:restriction} can be uniquely extended to a gauge-invariant KMS-state on $\NT(X)$, now without any restrictions on~$P$. The proof relies on the results of Section~\ref{sec:graded}, but otherwise follows the familiar strategy of first extending a trace to the core subalgebra and then using the gauge-invariant conditional expectation to extend the trace to a state on the entire algebra $\NT(X)$. This way we get a classification of the gauge-invariant KMS-states on $\NT(X)$ in terms of traces on $A$, which is the first main result of the paper, see Theorem~\ref{thm:general}.

In Section~\ref{sec:Wold} we show that every KMS-state on $\NT(X)$ uniquely decomposes into a finite and an infinite part, or in other words, into a Gibbs-type functional and a functional which does not dominate any functional of Gibbs type, see Proposition~\ref{prop:Wold-state} and Theorem~\ref{thm:Wold}. We show then that whether a KMS-state is of finite or of infinite type is determined completely by its restriction to the coefficient algebra $A$, see Corollary~\ref{cor:finite-infinite}. Together with our general characterization of the restrictions of KMS-states to $A$ this gives a complete answer to question~(1) above.

In Section~\ref{sec:critical-temp} we prove  that if the restriction $\tau$ of a KMS$_\beta$-state $\phi$ on $\NT(X)$ is such that the trace $\sum_{p\in P}N(p)^{-\beta}F_p\tau$ is finite, then $\phi$ is of Gibbs type, see Theorem~\ref{thm:finite-characterization}. In particular, if the series $\sum_{p\in P}N(p)^{-\beta}F_p\tau$ are convergent for all tracial states $\tau$, then all KMS$_\beta$-states are of Gibbs type and we get a complete classification of KMS$_\beta$-states in terms of traces on $A$, see Corollaries~\ref{cor:finite-class2} and~\ref{cor:finite-class3}. This gives an answer to the first question in (2) and is the second main result of the paper.

These results confirm the general principle, known to many researchers who have studied KMS-states, though perhaps not directly formalized in print, saying that the KMS$_\beta$-states for~$\beta$ in the region of convergence of a suitably defined zeta-function of the time evolution are of Gibbs type. There are many results of this sort in the literature, proved in different contexts, and usually  requiring quite strong assumptions, see, e.g., \cite[Theorem 6.1]{AaHR}, \cite[Theorem 4.10]{HLS2}, \cite[Theorem 6.1]{aHLRS}, \cite[Proposition~1.2]{lln}, \cite[Theorem 7.1]{lac-rae}. By contrast we see that in our setting of product systems over quasi-lattice ordered monoids this principle is valid in great generality.

When $N(p)\ge1$ for all $p$, there is a smallest number $\beta_c$ such that the series $\sum_{p\in P}N(p)^{-\beta}F_p\tau$ are convergent for all positive traces $\tau$ and $\beta>\beta_c$. We show that if $P$ is finitely generated and suitable additional assumptions are satisfied, there is a KMS$_{\beta_c}$-state which is not of Gibbs type, see Proposition~\ref{prop:phase-transition}. Therefore $\beta_c$ is the largest critical inverse temperature at which we have a phase transition. This gives an answer to the second question in (2).

Section~\ref{sec:gauge-invariance} is inspired by the recent work of Bruce--Laca--Ramagge--Sims~\cite{BLRS} on the Toeplitz algebras of quasi-lattice ordered monoids. Generalizing some of their results, we establish sufficient conditions for a KMS-state on~$\NT(X)$ to be gauge-invariant, which can sometimes be applied to states at the critical inverse temperature and even below it.

In Section~\ref{sec:Artin} we consider a particular class of quasi-lattice ordered monoids, the right-angled Artin monoids. We show that for them the system of inequalities from Section~\ref{sec:restriction} reduces to a much smaller system, see Theorem~\ref{thm:Artin}. We leave open the problem to what extent this result is true for other monoids. A solution to this problem would provide a better answer to question~(1).

Section~\ref{sec:products} is motivated by the recent work of Kakariadis~\cite{kak}, which appeared while we were writing up this paper. In this work Kakariadis studies questions similar to (1) and (2) for finite rank product systems over $\Z^n_+$ for $n\geq 1$. His approach is quite different from ours and relies on a more refined decomposition of KMS-states than just into finite/infinite types. Namely, a key idea in~\cite{kak}, partly inspired by~\cite{C}, is that a state can be of finite type with respect to one coordinate of $\Z_+^n$ and of infinite type with respect to another. We explain how this idea works for product-monoids in our approach and as an application show how results in~\cite{kak} and our results for $\Z^n_+$  can be quickly deduced from each other, see Example~\ref{ex:kak}.

We conclude in Section~\ref{sec:examples} by considering some examples that have appeared in the literature.

\smallskip

Finally we would like to add that most results of the paper admit relatively straightforward generalizations to arbitrary quasi-free dynamics on $\NT(X)$ (see the discussion at the end of Section~\ref{ssec:transfer} for a more precise statement), but then the description of KMS-states on $\NT(X)$ is in terms of KMS-states on $A$ instead of traces. There are several reasons why we have chosen to work mostly with the gauge-type dynamics defined by a homomorphism $N\colon P\to(0,+\infty)$ and confine ourselves to a few remarks on the general case. The main one is that the gauge-type dynamics form the main focus of the current research. The general case would also require a few more prerequisites, in particular some familiarity with modular theory and KMS-weights, which together with the length of the paper might discourage the potential reader.

\medskip

\paragraph{\bf Acknowledgement} {\small This research was initiated when Z.A. visited the Department of Mathematics at the University of Oslo. She thanks her colleagues for their hospitality.}

\bigskip

\section{Preliminaries}\label{sec:prelim}

\subsection{Product systems and associated algebras}

For a more thorough introduction into these topics see~\cite{F}.

\smallskip

Assume $G$ is a group, $P\subset G$ is a monoid generating $G$ such that $P\cap P^{-1}=\{e\}$. From time to time we will have to assume that $P$ is finitely generated. By a generating set of $P$ we mean a set $S\subset P$ such that every element of $P\setminus\{e\}$ can be written as a product of elements of $S$.

Define a partial order on $G$ by
$$
g\le h\ \ \text{iff}\ \ g^{-1}h\in P.
$$
Then $G$ is said to be \emph{quasi-lattice ordered} if any two elements $g,h\in G$ with a common upper bound have a least upper bound $g\vee h$, see~\cite{nic}. If $g,h$ do not have a common upper bound, we write $g\vee h=\infty$. Therefore
$$
gP\cap hP=\begin{cases}(g\vee h)P,&\text{if}\ \ g\vee h<\infty,\\
\emptyset,&\text{if}\ \ g\vee h=\infty. \end{cases}
$$

For a finite subset $J$ of $P$ we let
\begin{equation} \label{eq:qK}
q_J=\bigvee_{p\in J}p\in P\cup\{\infty\}.
\end{equation}

We say that a subset $F$ of $P$ is \emph{$\vee$-closed} if for any two elements $p,q\in F$ with a common upper bound in $P$ we have $p\vee q\in F$.

\smallskip

Let $A$ be a C$^*$-algebra and $X=(X_p)_{p\in P}$ be a collection of essential C$^*$-correspondences over~$A$. Therefore each $X_p$ is a right C$^*$-Hilbert $A$-module equipped with a left action of $A$ by adjointable operators such that $\overline{AX_p}=X_p$. The collection $X$ is called a \emph{product system} over~$P$~if

\begin{itemize}
\item[(a)] $\cup_{p\in P}X_p$ has a semigroup structure such that for any $\xi\in X_p$ and $\zeta\in X_q$ we have $\xi\zeta\in X_{pq}$, and the map $\xi\otimes\zeta\mapsto\xi\zeta$ extends to an isometric isomorphism $X_p\otimes_A X_q\cong X_{pq}$ for all $p,q\in P$;
\item[(b)] $X_e=A$ as a C$^*$-correspondence, and the products $X_e\times X_p\to X_p$, $X_p\times X_e\to X_p$ coincide with the structure maps of the $A$-bimodules $X_p$.
\end{itemize}

\begin{remark}
In the literature a more general notion of product systems is often used which does not require the C$^*$-correspondences $X_p$ to be essential, see, e.g., \cite{CLSV,SY}. For some of our results, however, this assumption will be important. In the end it is not very restrictive as we can always pass to the unitization $A^\sim$ of $A$ and consider every C$^*$-correspondence over $A$ as an essential correspondence over $A^\sim$.
\end{remark}

For $p\le q$ we have a natural homomorphism $\iota^q_p\colon \LL(X_p)\to\LL(X_q)$ obtained by identifying $X_q$ with $X_p\otimes_A X_{p^{-1}q}$ and mapping $S\in\LL(X_p)$ into $S\otimes1$. Then $X$ is called \emph{compactly aligned} if for all $p,q\in P$ with $p\vee q<\infty$ we have
$$
\iota^{p\vee q}_p(\K(X_p))\iota^{p\vee q}_q(\K(X_q))\subset\K(X_{p\vee q}).
$$
By~\cite[Proposition~5.8]{F}, this assumption is satisfied if the left actions of $A$ on $X_p$ are by generalized compact operators.

\smallskip

Assume now that $B$ is a C$^*$-algebra and we are given linear maps $\psi_p\colon X_p\to B$. It is often convenient to view such a collection $(\psi_p)_{p\in P}$ as one map $\psi\colon X\to B$. It is called a \emph{Toeplitz representation}~if
$$
\psi_{pq}(\xi\zeta)=\psi_p(\xi)\psi_q(\zeta)\ \ \text{for all}\ \ \xi\in X_p\ \ \text{and}\ \ \zeta\in X_q
$$
and
$$
\psi_e(\langle \xi,\zeta\rangle)=\psi_p(\xi)^*\psi_p(\zeta)\ \ \text{for all}\ \ \xi,\zeta\in X_p.
$$
This in particular implies that $\psi_e\colon A=X_e\to B$ is a $*$-homomorphism.

Given a Toeplitz representation $\psi$ of $X$ into $B$, we can define $*$-homomorphisms
\begin{equation}\label{eq:CPmap}
\psi^{(p)}\colon\K(X_p)\to B\ \ \text{by}\ \ \psi^{(p)}(\theta_{\xi,\zeta})=\psi_p(\xi)\psi_p(\zeta)^*,
\end{equation}
where $\theta_{\xi,\zeta}\in\K(X_p)$ is defined by $\theta_{\xi,\zeta}\eta=\xi\langle\zeta,\eta\rangle$.

The main example of a Toeplitz representation is constructed as follows. Consider the \emph{Fock module}
$$
\F(X)=\bigoplus_{p\in P}X_p
$$
and define a representation $\ell\colon X\to \mathcal{L}(\mathcal{F}(X))$ by
\begin{equation*} \label{eq:Toeplitzr}
\ell(\xi)\zeta=\xi \zeta\ \ \text{for}\ \ \xi\in X_p,\ \zeta\in X_q.
\end{equation*}
We define the \emph{reduced Toeplitz algebra} of $X$ as the C$^*$-algebra $\NTr(X)$ generated by the image of $\ell$ in~$\mathcal{L}(\mathcal{F}(X))$.

\begin{lemma}\label{lem:density}
For any product system $X$, the C$^*$-algebra $\NTr(X)$ is strictly dense in $\LL(\F(X))=M(\K(\F(X)))$.
\end{lemma}

\bp
This is a standard argument. We first observe that if $(u_i)_i$ is an approximate unit in~$\K(X_p)$, then the net $(\ell^{(p)}(u_i))_i\subset\NTr(X)$ converges strictly to the projection $f_p$ onto the submodule $\oplus_{q\in pP}X_q\subset\F(X)$. Next, consider the finite subsets $J\subset P\setminus\{e\}$ and partially order them by inclusion. Then the net $\big(\prod_{p\in J}(1-f_p)\big)_J$ converges strictly to the projection $Q_e$ onto $X_e\subset\F(X)$. Since for any $\xi\in X_p$ and $\zeta\in X_q$ we have $\ell(\xi)Q_e\ell(\zeta)^*=\theta_{\xi,\zeta}$, we can therefore conclude that the strict closure of $\NTr(X)$ contains $\K(\F(X))$. This gives the result.
\ep

For any Toeplitz representation $\psi\colon X\to B$ we have the following useful property: if $p\le q$, $\xi\in X_p$ and $\zeta\in X_q$, then
\begin{equation}\label{eq:Toeplitz}
\psi(\xi)^*\psi(\zeta)=\psi(\ell(\xi)^*\zeta),
\end{equation}
which is easy to check on the elementary tensors $\zeta$ in $X_q\cong X_p\otimes_AX_{p^{-1}q}$. Note that $\ell(\xi)^*\zeta\in X_{p^{-1}q}$.

Any product system admits a universal Toeplitz representation $X\to \mathcal{T}(X)$. In general, though, the \emph{Toeplitz algebra}~$\mathcal{T}(X)$ of $X$ is too large to be interesting.

\smallskip

If $X$ is compactly aligned, then a Toeplitz representation $\psi\colon X\to B$ is called \emph{Nica covariant}~if
$$
\psi^{(p)}(S)\psi^{(q)}(T)=\begin{cases}\psi^{(p\vee q)}(\iota^{p\vee q}_p(S)\iota^{p\vee q}_q(T)),& \text{if}\ p\vee q<\infty,\\
0,&\text{otherwise}. \end{cases}
$$
In practice a more useful form of Nica covariance is often the following.

\begin{lemma}
A Toeplitz representation $\psi\colon X\to B$ of a compactly aligned product system is Nica covariant if and only if the following property holds: for $\xi\in X_p$ and $\zeta\in X_q$, we have
\begin{align}
&\psi(\xi)^*\psi(\zeta)=0 \ \ \text{if}\ \ p\vee q=\infty,\label{eq:Nica0}\\
&\psi(\xi)^*\psi(\zeta)\in\overline{\psi(X_{p^{-1}(p\vee q)})\psi(X_{q^{-1}(p\vee q)})^*}\ \ \text{if}\ \ p\vee q<\infty.\label{eq:Nica}
\end{align}
\end{lemma}

\bp
The forward implication is proved in~\cite[Proposition~5.10]{F}. For the converse, assume~\eqref{eq:Nica0} and~\eqref{eq:Nica} are satisfied and take $S=\theta_{\nu,\xi}\in\K(X_p)$ and $T=\theta_{\zeta,\eta}\in\K(X_q)$. Then
$$
\psi^{(p)}(S)\psi^{(q)}(T)=\psi(\nu)\psi(\xi)^*\psi(\zeta)\psi(\eta)^*.
$$
If $p\vee q=\infty$, this expression is zero by~\eqref{eq:Nica0}.

Assume now that $p\vee q<\infty$. Since $\psi(\xi)^*\psi(\zeta)\in\overline{\psi(X_{p^{-1}(p\vee q)})\psi(X_{q^{-1}(p\vee q)})^*}$ by~\eqref{eq:Nica},
if we take an approximate unit $(u_i)_i$ in $\K(X_{p^{-1}(p\vee q)})$, we get the norm convergence
$$
\psi^{(p^{-1}(p\vee q))}(u_i)\psi(\xi)^*\psi(\zeta)\to\psi(\xi)^*\psi(\zeta).
$$
We may assume that $u_i=\sum_{\lambda\in J_i}\theta_{\lambda,\lambda}$ for some finite sets $J_i\subset X_{p^{-1}(p\vee q)}$. Then
$$
\psi(\nu)\psi^{(p^{-1}(p\vee q))}(u_i)\psi(\xi)^*=\sum_{\lambda\in J_i}\psi(\nu)\psi(\lambda)\psi(\lambda)^*\psi(\xi)^*=\psi^{(p\vee q)}(S_i),
$$
where $S_i=\sum_{\lambda\in J_i}\theta_{\nu\lambda,\xi\lambda}$. It follows that
$$
\psi^{(p)}(S)\psi^{(q)}(T)=\lim_i \psi(\nu)\psi^{(p^{-1}(p\vee q))}(u_i)\psi(\xi)^*\psi(\zeta)\psi(\eta)^*=\lim_i \psi^{(p\vee q)}(S_i)\psi(\zeta)\psi(\eta)^*.
$$
Observe also that the net $(S_i)_i$ in $\K(X_{p\vee q})$ is bounded and converges to $\iota^{p\vee q}_p(S)$ in the strict topology on $\LL(X_{p\vee q})=M(\K(X_{p\vee q}))$.

In a similar way we can take an approximate unit $(u_j)_j$ in $\K(X_{q^{-1}(p\vee q)})$, insert $\psi^{(q^{-1}(p\vee q))}(u_j)$ between $\psi(\zeta)$ and $\psi(\eta)^*$, and get a bounded net $(T_j)_j$ in $\K(X_{p\vee q})$ such that
$$
\psi^{(p)}(S)\psi^{(q)}(T)=\lim_{i,j} \psi^{(p\vee q)}(S_i)\psi^{(p\vee q)}(T_j)=\lim_{i,j}\psi^{(p\vee q)}(S_iT_j),
$$
with the convergence in norm, and $\iota^{p\vee q}_q(T)=\lim_jT_j$ in the strict topology on $\LL(X_{p\vee q})$.

To finish the proof it remains to show that
$$
\lim_{i,j}\psi^{(p\vee q)}(S_iT_j)=\psi^{(p\vee q)}(\iota^{p\vee q}_p(S)\iota^{p\vee q}_q(T)).
$$
But this is true, since $S_iT_j\to \iota^{p\vee q}_p(S)\iota^{p\vee q}_q(T)$ in the strict topology on $\K(X_{p\vee q})$ and hence $\psi^{(p\vee q)}(S_iT_j)\to \psi^{(p\vee q)}(\iota^{p\vee q}_p(S)\iota^{p\vee q}_q(T))$ in the strict topology on $\psi^{(p\vee q)}(\K(X_{p\vee q}))$.
\ep

\begin{remark}\label{rem:Nica2}
Condition~\eqref{eq:Nica} might not look like a typical relation, but in fact it dictates how $\psi(\xi)^*\psi(\zeta)$ can be approximated by elements of $\psi(X_{p^{-1}(p\vee q)})\psi(X_{q^{-1}(p\vee q)})^*$. Indeed, taking an approximate unit $u_i=\sum_{\lambda\in J_i}\theta_{\lambda,\lambda}$ in  $\K(X_{p^{-1}(p\vee q)})$ as in the previous proof, we get, using~\eqref{eq:Toeplitz},~that
$$
\psi(\xi)^*\psi(\zeta)=\lim_i\sum_{\lambda\in J_i}\psi(\lambda)\psi(\lambda)^*\psi(\xi)^*\psi(\zeta)=\lim_i\sum_{\lambda\in J_i}\psi(\lambda)\psi(\ell(\zeta)^*(\xi\lambda))^*,
$$
which gives the required approximation.
\end{remark}

For every compactly aligned product system $X$ there is a universal Nica covariant Toeplitz representation $i_X\colon X\to \NT(X)$, see~\cite{F}. The C$^*$-algebra $\NT(X)$ is called the \emph{Nica-Toeplitz algebra} of $X$; it is denoted by $\mathcal{T}_{\mathrm{cov}}(X)$ in op.~cit.

For $a\in A=X_e$, we usually identify $a$ with~$i_X(a)$, so we view $A$ as a C$^*$-subalgebra of~$\NT(X)$.
Properties~\eqref{eq:Nica0} and~\eqref{eq:Nica} for $\psi=i_X$ imply that the space spanned by the elements $i_X(\xi)i_X(\zeta)^*$ is dense in $\NT(X)$.

By~\cite[Proposition~3.5]{CLSV} the C$^*$-algebra $\NT(X)$ carries a full coaction $$\delta\colon\NT(X)\to \NT(X)\otimes C^*(G),$$ called the \emph{gauge coaction}, such that
$$
\delta(i_X(\xi))=i_X(\xi)\otimes\lambda_p\ \ \text{for}\ \ \xi\in X_p\ \ \text{and}\ \ p\in P.
$$
For $g\in G$, denote by $\NT(X)_g\subset\NT(X)$ the subspace of elements of degree $g$ with respect to the coaction $\delta$, that is, the set of elements $S\in\NT(X)$ such that $\delta(S)=S\otimes\lambda_g$. The elements $i_X(\xi)i_X(\zeta)^*$, with $\xi\in X_p$, $\zeta\in X_q$ and $pq^{-1}=g$, span a dense subspace of $\NT(X)_g$. The fixed point subalgebra $\NT(X)_e\subset\NT(X)$ is also denoted by $\F$ and called the \emph{core subalgebra} of~$\NT(X)$. We have a unique gauge-invariant conditional expectation $$E=(\iota\otimes\tau_G)\circ\delta\colon\NT(X)\to\F,$$ where~$\tau_G$ is the canonical trace on $C^*(G)$.

\begin{lemma}
For any product system $X$, there is a unique reduced coaction $\delta_r$ of $G$ on $\NTr(X)$ such that
$$
\delta_r(\ell(\xi))=\ell(\xi)\otimes\lambda_p\ \ \text{for all}\ \ \xi\in X_p\ \ \text{and}\ \ p\in P.
$$
In other words, $\delta_r$ is an injective $*$-homomorphism $\NTr(X)\to\NTr(X)\otimes C^*_r(G)$ such that $$(\iota\otimes\Delta)\circ\delta_r=(\delta_r\otimes\iota)\circ\delta_r$$ and the space $(1\otimes C^*_r(G))\delta_r(\NTr(X))$ is dense in $\NTr(X)\otimes C^*_r(G)$.
\end{lemma}

Here $\Delta$ is the standard coproduct on $C^*_r(G)$ given by $\Delta(\lambda_g)=\lambda_g\otimes\lambda_g$. We call $\delta_r$ again the gauge coaction.

\bp
Consider the right C$^*$-Hilbert $A$-module
$$
\F(X)\otimes\ell^2(G)=\bigoplus_{p\in P}X_p\otimes\ell^2(G)
$$
and the unitary operator $V=\bigoplus_{p\in P}1_{X_p}\otimes\lambda_p$ on it. For $S\in\LL(\F(X))$ define
$$
\delta_r(S)=V(S\otimes1)V^*\in\LL(\F(X)\otimes\ell^2(G)).
$$
If we view $\LL(\F(X))\otimes C^*_r(G)$ as a C$^*$-subalgebra of $\LL(\F(X)\otimes\ell^2(G))$, then by definition we immediately get that $\delta_r(\ell(\xi))=\ell(\xi)\otimes\lambda_p$ for $\xi\in X_p$. Therefore the restriction of $\delta_r$ to $\NTr(X)$ gives the required injective homomorphism $\NTr(X)\to\NTr(X)\otimes C^*_r(G)$ such that $(\iota\otimes\Delta)\circ\delta_r=(\delta_r\otimes\iota)\circ\delta_r$. The density of $(1\otimes C^*_r(G))\delta_r(\NTr(X))$ in $\NTr(X)\otimes C^*_r(G)$ is clear, since products of $\ell(\xi)$ and $\ell(\xi)^*$ for $\xi\in X_p$, $p\in P$, span a dense subspace of $\NTr(X)$ and $\delta_r$ maps every such product into the same product tensored with $\lambda_g$ for some $g\in G$. The uniqueness of $\delta_r$ is also obvious.
\ep

The fixed point subalgebra $\NTr(X)_e\subset\NTr(X)$ under the gauge coaction will also be denoted by $\F^r$ and called the core subalgebra of~$\NTr(X)$. We have a faithful conditional expectation $E_r=(\iota\otimes\tau_G)\circ\delta_r\colon \NTr(X)\to\F^r$, cf.~\cite[Proposition 5.4]{kwa-larI}.

\smallskip

If $X$ is compactly aligned, then the Toeplitz representation $\ell\colon X\to \LL(\F(X))$ is Nica covariant, hence we have a surjective $*$-homomorphism $\Lambda\colon\NT(X)\to\NTr(X)$ mapping $i_X(\xi)$ into~$\ell(\xi)$.
Amenability criteria ensuring that $\Lambda$ is injective are discussed in~\cite[Sections~7-8]{F}. In particular, by~\cite[Corollary 8.2]{F}, $\Lambda$ is injective if $(G,P)$ is the free product of quasi-lattice ordered groups $(G_i,P_i)$, $i\in I$, where every group $G_i$ is amenable.

Independently of amenability of $X$ we have the following property, which is a particular case of~\cite[Corollary~6.4]{kwa-larI}.

\begin{lemma}\label{lem:core}
For any compactly aligned product system $X$, the map $\Lambda\colon\NT(X)\to\NTr(X)$ is injective on the core subalgebra $\F$, so it defines an isomorphism $\F\cong\F^r$.
\end{lemma}

\bp
For every finite subset $F\subset P$, denote by $B_F$ the sum of the subalgebras $i_X^{(p)}(\K(X_p))\subset\NT(X)$, $p\in F$. By~\cite[Lemma~3.6]{CLSV}, if $F$ is $\vee$-closed, then $B_F$ is a C$^*$-algebra. Since $\F$ is the inductive limit of these C$^*$-algebras, it is therefore enough to show that $\Lambda|_{B_F}$ is injective for every finite $\vee$-closed subset $F\subset P$. Take such a set $F$ and a nonzero element $S\in B_F$. By removing some elements of $F$ if necessary, we may assume that there exists a minimal element~$p$ in~$F$ such that the $p$-component of $S$ is nonzero. Let $i_X^{(p)}(S')$ be this component. Now let us look at the action of $\Lambda(S)$ on $X_p\subset\F(X)$.
Since $\Lambda\circ i_X^{(q)}=\ell^{(q)}$ and we have
$$
\ell^{(p)}(S')|_{X_p}=S'\ \ \text{and}\ \ \ell^{(q)}(S'')|_{X_p}=0
$$
for any $q\not\le p$ and $S''\in\K(X_q)$, we conclude that $\Lambda(S)|_{X_p}=S'\ne0$.
\ep

\begin{corollary}\label{cor:gauge-invariance}
Any gauge-invariant state $\phi$ on $\NT(X)$ factors through $\NTr(X)$.
\end{corollary}

\bp
Let $\psi$ be the unique state on $\F^r$ such that $\psi\circ\Lambda=\phi$ on $\F$. By composing it with the gauge-invariant conditional expectation $E_r\colon \NTr(X)\to\F^r$, we can extend $\psi$ to a gauge-invariant state on $\NTr(X)$, which we continue to denote by $\psi$. Then the state $\psi\circ\Lambda$ vanishes on $\NT(X)_g$ for all $g\ne e$, hence it is gauge-invariant, and since it coincides with $\phi$ on $\F$, we conclude that $\psi\circ\Lambda=\phi$.
\ep

There are several other C$^*$-algebras naturally associated with $X$. We will only consider the following one, closely related to the Cuntz-Pimsner algebra $\OO_X$ defined in~\cite{F}. For a compactly aligned product system $X$, denote by $\NOs(X)$ the quotient of~$\NT(X)$ by the ideal generated by the elements $a-i^{(p)}_X(\varphi_p(a))$ for all $p\in P$ and $a\in\varphi_p^{-1}(\K(X_p))$, where $\varphi_p\colon A\to\LL(X_p)$ is the homomorphism defining the left $A$-module structure on $X_p$.

\begin{remark}\label{re:cuntzalg}
If $P$ is directed (that is, $p\vee q<\infty$ for all $p,q\in P$) and the maps $\varphi_p$ are injective, then the Cuntz-Nica-Pimsner algebra $\mathcal{NO}_X$ defined in~\cite{SY} is a quotient of $\NOs(X)$ (see, e.g., the computation in \cite[Proposition 5.1(1)]{SY} together with \cite[Proposition~5.4]{F}). If, moreover, $\varphi_p(A)\subset\K(X_p)$ for all $p$, then  by \cite[Proposition~5.4]{F} and \cite[Corollary 5.2]{SY} the C$^*$-algebra $\NOs(X)$ coincides with $\mathcal{NO}_X$, as well as with the Cuntz-Pimsner algebra $\OO_X$ from~\cite{F}.
\end{remark}

The gauge coaction $\delta\colon\NT(X)\to\NT(X)\otimes C^*(G)$ descends to a full coaction of~$G$ on~$\NOs(X)$, which we continue to call the gauge coaction.

\subsection{KMS-states}

If $\sigma=(\sigma_t)_{t\in\R}$ is a time evolution on a C$^*$-algebra $A$, recall that an element $a\in A$ is called analytic provided that the $A$-valued function $t\mapsto \sigma_t(a)$ extends to an analytic function on~$\mathbb{C}$. We let  $\AN$ denote the subset of analytic elements, and recall that it is a dense $*$-subalgebra of~$A$, see, for example,  \cite[\S 8.12.1]{ped}.

A state $\phi$ on $A$ is called a $\sigma$-KMS$_\beta$-state if the equality
\begin{equation}\label{eq:simple-condition-forKMS}
\phi(ab)=\phi(b\sigma_{i\beta}(a))
\end{equation}
holds for all $a$ and $b$ in a dense subset of analytic elements. It will be useful to have the following simple lemma.

\begin{lemma}\label{lem:cognoscenti} Let $\sigma$ be a time evolution on a C$^*$-algebra $A$.  Suppose that $\phi$ is a $\sigma$-invariant state on $A$ such that \eqref{eq:simple-condition-forKMS} is satisfied for all $a$ in a set of analytic elements generating $A$ as a C$^*$-algebra and all $b$ spanning a dense subspace of $A$.
Then $\phi$ is a $\sigma$-KMS$_\beta$-state.
\end{lemma}

\begin{proof}
Let $\mathcal{C}$ be the set of all elements $a\in \AN$ such that \eqref{eq:simple-condition-forKMS} holds for all $b$ spanning a dense subspace of~$A$ (depending on $a$). By linearity and continuity it follows that \eqref{eq:simple-condition-forKMS} holds for all $a\in \mathcal{C}$ and all $b\in A$. By assumption, $\mathcal{C}$ generates $A$ as a C$^*$-algebra.

Taking adjoint in \eqref{eq:simple-condition-forKMS} gives
\[
\phi(b^*a^*)=\phi(\sigma_{i\beta}(a)^*b^*)=\phi(\sigma_{-i\beta}(a^*)b^*)
\]
whenever $a\in \mathcal{C}$ and $b\in A$. Since the adjoint of an analytic element is again in $\AN$, it follows that $b^*a^*$ is analytic for every $a\in \mathcal{C}$ and $b\in \AN$. The assumed $\sigma_t$-invariance for $t\in \mathbb{R}$ implies $\sigma_z$-invariance on analytic elements for all $z\in \mathbb{C}$. Therefore
\[
\phi(\sigma_{i\beta}(b^*)\sigma_{i\beta}(a^*))=\phi(a^*\sigma_{i\beta}(b^*)).
\]
Renaming $\sigma_{i\beta}(b^*)$ as $b$, this shows that $\phi(a^*b)=\phi(b\sigma_{i\beta}(a^*))$ for all $a\in \mathcal{C}$ and $b\in \AN$. Thus, $\mathcal{C}$ is closed under taking adjoint.

The set $\mathcal{C}$ is also closed under multiplication. Indeed, let $a,c\in \mathcal{C}$. For each $b\in \AN$ we have
\[
\phi(acb)=\phi(cb\sigma_{i\beta}(a))=\phi(b\sigma_{i\beta}(a)\sigma_{i\beta}(c))=\phi(b\sigma_{i\beta}(ac)).
\]

We have established that $\mathcal{C}$ is closed under adjoint and multiplication, so it is itself dense in~$A$. Therefore \eqref{eq:simple-condition-forKMS} holds for a dense set of analytic elements $a$ and for all $b\in A$. Hence $\phi$ is a KMS-state.
\end{proof}

Assume now that $(G,P)$ is a quasi-lattice ordered group and $X$ is a compactly aligned product system over $P$. A homomorphism $N\colon P\to (0,+\infty)$, with $(0,+\infty)$ viewed as a group under multiplication, gives rise to a time evolution $\sigma$ on the Nica-Toeplitz algebra $\NT(X)$ that satisfies
\begin{equation}\label{eq:dynamics1}
\sigma_t(i_X(\xi))=N(p)^{it}i_X(\xi)\ \ \text{for}\ \ \xi\in X_p,\ p\in P.
\end{equation}
It descends to time evolutions on $\NOs(X)$ and $\NTr(X)$. This is clear for $\NOs(X)$, and in order to show this for~$\NTr(X)$ consider the (possibly unbounded) operator~$D_N$ on $\F(X)$ such that
\begin{equation}\label{eq:DN}
D_N\xi=N(p)\xi\ \ \text{for}\ \ \xi\in X_p.
\end{equation}
Then $\Ad D_N^{it}\circ\Lambda=\Lambda\circ\sigma$, where $(\Ad D_N^{it})(T)=D_N^{it}TD_N^{-it}$, so the required time evolution on~$\NTr(X)$ is given by the restriction of $\Ad D_N^{it}$.
Note that it is well-defined even when $X$ is not compactly aligned and so the algebras $\NT(X)$ and $\NOs(X)$ are not defined.

We will be interested in the KMS-states on $\NT(X)$ and $\NOs(X)$ with respect to the dynamics defined by \eqref{eq:dynamics1}. The main problem is to understand the KMS-states on the former algebra, the ones on the latter can then be obtained using the following lemma.

\begin{lemma}[cf.~{\cite[Theorem~2.5]{lac-nes}}]\label{lem:TvsP}
A KMS-state $\phi$ on $\NT(X)$ with respect to some dynamics on $\NT(X)$ factors through $\NOs(X)$ if and only if $\phi(a)=\phi(i_X^{(p)}(\varphi_p(a)))$ for all $p\in P$ and $a\in\varphi_p^{-1}(\K(X_p))$.
\end{lemma}

\bp
Clearly, the condition $\phi(a)=\phi(i_X^{(p)}(\varphi_p(a)))$ is necessary for $\phi$ to factor through $\NOs(X)$. In order to see that it is sufficient, observe that
$i_X^{(p)}(\varphi_p(a)S)=a\,i_X^{(p)}(S)$ for all $a\in A$ and $S\in\K(X_p)$. It follows that if $\varphi_p(a)\in \K(X_p)$, then
$$
(a-i_X^{(p)}(\varphi_p(a)))^*(a-i_X^{(p)}(\varphi_p(a)))=a^*a-i_X^{(p)}(\varphi_p(a^*a)).
$$
Hence $\phi$ vanishes on $(a-i_X^{(p)}(\varphi_p(a)))^*(a-i_X^{(p)}(\varphi_p(a)))$. Since $\phi$ is a KMS-state, the corresponding cyclic vector in the Hilbert space of the GNS-representation is separating, so it follows that  $a-i_X^{(p)}(\varphi_p(a))$ is contained in the kernel of the GNS-representation for all $a\in\varphi_p^{-1}(\K(X_p))$. Therefore the GNS-representation of $\NT(X)$ factors through $\NOs(X)$ and hence $\phi$ factors through $\NOs(X)$ as well.
\ep

Our main results will be about gauge-invariant KMS-states on $\NT(X)$. Note that by Corollary~\ref{cor:gauge-invariance} a study of such states on $\NT(X)$ is equivalent to that on $\NTr(X)$.

\smallskip

There is a much larger class of time evolutions on $\NT(X)$ than those defined by homomorphisms $N$, which we will now briefly consider.

Assume $A$ is equipped with a time evolution $\gamma=(\gamma_t)_{t\in\R}$ and every C$^*$-correspondence $X_p$, $p\ne e$, is equivariant, that is, we are given a strongly continuous one-parameter group $U^{(p)}$ of isometries of the Banach space $X_p$ such that
$$
\langle U^{(p)}_t\xi,U^{(p)}_t\zeta\rangle=\gamma_t(\langle\xi,\zeta\rangle)\ \ \text{and}\ \ U^{(p)}_t(a\xi)=\gamma_t(a)U^{(p)}_t\xi,
$$
which implies that we also have $U^{(p)}_t(\xi a)=(U^{(p)}\xi)\gamma_t(a)$. If we let $U^{(e)}=\gamma$ and assume that the collection $U=(U^{(p)})_{p\in P}$ is multiplicative in the sense that
$$
U^{(pq)}_t(\xi\zeta)=(U^{(p)}_t\xi)(U^{(q)}_t\zeta),
$$
then we can define a \emph{quasi-free dynamics} $\sigma^U$ on $\NT(X)$ by
$$
\sigma^U_t(i_X(\xi))=i_X(U^{(p)}_t\xi)
$$
for $\xi\in X_p$, $p\in P$ and $t\in\R$.

\subsection{Induction of traces and KMS-weights}\label{ssec:induction}

If $A$ is a C$^*$-algebra, $\tau$ is a tracial positive linear functional on $A$, and $Y$ is a right C$^*$-Hilbert $A$-module, then by \cite[Theorem 1.1]{lac-nes} we get a unique strictly lower semicontinuous, in general infinite, trace $\Tr^Y_\tau$ on $\LL(Y)$, which we often denote simply by~$\Tr_\tau$, such that
$$
\Tr_\tau(\theta_{\xi,\xi})=\tau(\langle\xi,\xi\rangle) \ \ \text{for all}\ \ \xi\in Y.
$$
Explicitly, $\Tr_\tau$ can be defined as follows. Assume $(u_i)_{i\in I}$ is an approximate unit in $\K(Y)$ such that, for every $i$, $u_i=\sum_{\xi\in J_i}\theta_{\xi,\xi}$ for some finite set $J_i\subset Y$. Then
\begin{equation}\label{eq:Tr}
\Tr_\tau(T)=\sup_i\sum_{\xi\in J_i}\tau(\langle\xi,T\xi\rangle)=\lim_i\sum_{\xi\in J_i}\tau(\langle\xi,T\xi\rangle)\ \ \text{for}\ \ T\in\LL(Y)_+.
\end{equation}
In particular, if $Y$ admits a (possibly infinite) Parseval frame $(\xi_j)_{j\in J}$, meaning that
$$
\xi=\sum_j\xi_j\langle\xi_j,\xi\rangle\ \ \text{for all}\ \ \xi\in Y,
$$
then
\begin{equation}\label{eq:Tr Parseval frame}
\Tr_\tau(T)=\sum_j\tau(\langle\xi_j,T\xi_j\rangle)\ \ \text{for}\ \ T\in\LL(Y)_+.
\end{equation}

Continuity of $\Tr_\tau$ as a function of $\tau$ is a delicate issue, since $\Tr_\tau$ is an infinite trace in general, but the following simple observation will be enough for our purposes.

\begin{lemma}[{cf.~\cite[Propostion~2.4]{lac-nes}}]\label{lem:trace-continuity}
Assume $A$ is a C$^*$-algebra, $Y$ is a right C$^*$-Hilbert $A$-module and $\tau$ is a positive trace on $A$. Then
\begin{itemize}
\item[(i)] for every $T\in\LL(Y)_+$, the function $\omega\mapsto\Tr_\omega(T)$ is weakly$^*$ lower semicontinuous; in particular, if $(\tau_i)_i$ is an increasing net of positive traces on $A$ converging weakly$^*$ to $\tau$, then $\Tr_{\tau_i}(T)\nearrow\Tr_\tau(T)$;
\item[(ii)] if $\Tr_\tau(T)<\infty$ for some $T\in\LL(Y)_+$, then the function $\omega\mapsto\Tr_\omega(T)$ is weakly$^*$ continuous on the set of positive traces~$\omega$ on~$A$ such that $\omega\le\tau$.
\end{itemize}
\end{lemma}

\bp
Part (i) is clear as the functions $\omega\mapsto\omega(\langle\xi,T\xi\rangle)$ are weakly$^*$ continuous and the function $\omega\mapsto\Tr_\omega(T)$ is the pointwise supremum of finite sums of such functions.

For part (ii), fix $\eps>0$ and take an operator $S=\sum_{\xi\in J}\theta_{\xi,\xi}\le1$, with $J\subset Y$ finite, such that $\Tr_\tau(ST)=\sum_{\xi\in J}\tau(\langle\xi,T\xi\rangle)>\Tr_\tau(T)-\eps$. Then for any positive trace $\omega\le\tau$ and $T\in\LL(Y)_+$ we have
$$
\Tr_\omega(T)=\Tr_\omega(ST)+\Tr_\omega((1-S)T)\ \ \text{and}\ \ 0\le\Tr_\omega((1-S)T)\le \Tr_\tau((1-S)T)<\eps.
$$
It follows that if $\omega'$ is another positive trace such that $\omega'\le\tau$, then
\begin{align*}
|\Tr_\omega(T)-\Tr_{\omega'}(T)|&<|\Tr_\omega(ST)-\Tr_{\omega'}(ST)|+2\eps\\
&\le \sum_{\xi\in J}|\omega(\langle\xi,T\xi\rangle)-\omega'(\langle\xi,T\xi\rangle)|+2\eps.
\end{align*}
When $\omega'$ tends to $\omega$ in the weak$^*$ topology, then the last quantity converges to $2\eps$, and as $\eps>0$ was arbitrary, we conclude that $\Tr_{\omega'}(T)\to\Tr_\omega(T)$.
\ep

The following ``induction in stages'' type property of the traces $\Tr_\tau$ is very useful, see~\cite[Proposition~1.2]{lac-nes}. Assume $Y$ is a right C$^*$-Hilbert $A$-module and $Z$ is a C$^*$-correspondence over~$A$ such that the trace $\tau^Z=\Tr_\tau^Z\circ\,\varphi$ on $A$ is finite, where $\varphi\colon A\to \LL(Z)$ defines the left action of~$A$. Then
\begin{equation}\label{eq:induction}
\Tr_{\tau^Z}^Y(T)=\Tr_\tau^{Y\otimes_AZ}(T\otimes1)\ \ \text{for}\ \ T\in \LL(Y)_+.
\end{equation}

A few times it will be useful to remember that the construction of $\Tr_\tau$ is part of a more general procedure of induction of KMS-weights~\cite{lac-nes}. Namely, assume that $A$ is equipped with a time evolution $\gamma$ and a right C$^*$-Hilbert $A$-module $Y$ is equipped with a strongly continuous one-parameter group $U$ of isometries such that
$$
\langle U_t\xi,U_t\zeta\rangle=\gamma_t(\langle\xi,\zeta\rangle).
$$
Define a time evolution $\sigma^U$ on $\K(Y)$ by $\sigma^U_t(T)=U_tTU_{-t}$. Then for every $\gamma$-KMS$_\beta$-weight $\phi$ on $\overline{\langle Y,Y\rangle}$ there is a unique $\sigma^U$-KMS$_\beta$-weight $\Phi$ on $\K(Y)$ such that
$$
\Phi(\theta_{\xi,\xi})=\phi(\langle U_{i\beta/2}\xi,U_{i\beta/2}\xi\rangle)
$$
for all $\xi$ in the domain of definition of $U_{i\beta/2}$. This weight extends uniquely to a strictly lower semicontinuous weight $\operatorname{Ind}^U_Y\phi$ on $\LL(Y)$; this extension was denoted by $\kappa_\phi$ in~\cite{lac-nes}. Although the dynamics $T\mapsto U_tTU_{-t}$ on $\LL(Y)$ is only strictly continuous, the weight $\operatorname{Ind}^U_Y\phi$ satisfies the KMS$_\beta$ condition with respect to it. When $\gamma$ and $U$ are trivial and $\phi$ is a finite trace, then $\operatorname{Ind}^U_Y\phi$ is exactly the trace $\Tr^Y_\phi$.

By~\cite[Theorem~3.2]{lac-nes}, for every $\beta\in\R$, the map $\phi\mapsto\Phi$ defines a one-to-correspondence between the $\gamma$-KMS$_\beta$-weights on $\overline{\langle Y,Y\rangle}$ and the $\sigma^U$-KMS$_\beta$-weights on $\K(Y)$. A particular (but essentially equivalent) case of this correspondence is the claim that if $B$ is a C$^*$-algebra with time evolution $\sigma$ and $p\in M(B)$ is a $\sigma$-invariant full projection, then the map $\phi\mapsto\phi|_{pBp}$ defines a one-to-one correspondence between the $\sigma$-KMS$_\beta$-weights on $B$ and those on $pBp$. We will only need injectivity of this map, which is a rather simple consequence of the KMS-condition written in the form $\phi(aa^*)=\phi(\sigma_{i\beta/2}(a)^*\sigma_{i\beta/2}(a))$.

\subsection{Generalized dual Ruelle transfer operators}\label{ssec:transfer}

Given a C$^*$-correspondence $Y$ over a C$^*$-algebra $A$, we consider the operator $F_Y$ mapping a positive trace $\tau$ on $A$ into a positive, in general infinite, trace $F_Y\tau$ defined by
\begin{equation}\label{eq:transfer}
(F_Y\tau)(a)=\Tr^Y_\tau(a)\ \ \text{for}\ \ a\in A_+.
\end{equation}
Here, and in many other places, we omit the map $\varphi\colon A\to\LL(Y)$ defining the left action of $A$. To be pedantic, we should have written $(F_Y\tau)(a)=\Tr_{\tau}^Y(\varphi(a))$.

Similarly to~\cite{lac-nes}, the maps $F_Y$ will play an important role in our study of KMS-states. As the following example shows, they can be considered as generalizations of the dual Ruelle transfer operators.

\begin{example}[\cite{E}] \label{ex:Ruelle}
Assume $Z$ is a compact Hausdorff space and $h\colon Z\to Z$ is a surjective local homeomorphism. Then the Ruelle transfer operator $\LL\colon C(Z)\to C(Z)$ (corresponding to the zero Hamiltonian) is defined by
$$
(\LL f)(z)=\sum_{w\in h^{-1}(z)}f(w).
$$

Define a C$^*$-correspondence $Y$ over $A=C(Z)$ as follows. As a space we let $Y=C(Z)$. The bimodule structure and the $A$-valued inner product are defined by
$$
(a \cdot \xi \cdot b)(z) = a(z)\xi(z)b(h(z)),\quad \langle \xi,\zeta\rangle(z)=\sum_{w\in h^{-1}(z)}\overline{\xi(w)}\zeta(w)
$$
for all $a,b,\xi,\zeta\in C(Z)$ and $z\in Z$. We claim that $F_Y=\LL^*$.

Indeed, let  $(\rho_j)^d_{j=1}$ be a  partition of unity  such that $h$ is  injective on each $\supp \rho_j$, and take $\xi_j=\sqrt{\rho_j}$. Then
$\{\xi_j\}_j$ is a  Parseval frame for $Y$ (see, for example, \cite[Lemma~5.2]{AaHR}) and
$$
\sum_j\langle\xi_j,a\cdot\xi_j\rangle=\LL(a).
$$
By~\eqref{eq:Tr Parseval frame} we then get
$$
(F_Y\tau)(a)=\sum_j\tau(\langle \xi_j,a\cdot\xi_j\rangle)=\tau(\LL(a)),
$$
which is what we claimed.
\end{example}

Given a product system $X$ over $P$, with $X_e=A$, we will write $F_p$ instead of $F_{X_p}$. In a similar way we will write $\Tr^p_\tau$ instead of $\Tr^{X_p}_\tau$. We thus have
\begin{equation}\label{eq:F}
(F_p\tau)(a)=\Tr_\tau^p(a)\ \ \text{for}\ \ a\in A_+.
\end{equation}
Property~\eqref{eq:induction} implies that
\begin{equation}\label{eq:induction2}
\text{if}\ \ F_q\tau\ \ \text{is finite, then}\ \ F_pF_q\tau=F_{pq}\tau.
\end{equation}

\smallskip

We can now make a more precise statement on the possibility of extending the results of this paper from the time evolution on $\NT(X)$ defined by a homomorphism $N$ to general quasi-free dynamics. The tracial states $\tau$ on $A$ should be replaced by the $\gamma$-KMS$_\beta$-states $\psi$ on $A$, and the expressions $N(p)^{-\beta}\Tr^p_\tau$ and $N(p)^{-\beta}F_p\tau$ should be replaced by $\operatorname{Ind}^{U^{(p)}}_{X_p}\psi$ and $(\operatorname{Ind}^{U^{(p)}}_{X_p}\psi)|_A$, respectively. Then the results involving no conditions on $N$ generalize without much effort to arbitrary quasi-free dynamics; we will only make a few remarks throughout the paper indicating the necessary minimal changes. On the other hand, with the results involving assumptions $N(p)>1$ or $N(p)\ge1$ (and in Sections~\ref{sec:restriction}-\ref{sec:products} these are only Corollary~\ref{cor:finite-class3}, Proposition~\ref{prop:phase-transition}, Corollary~\ref{cor:finite-class4} and Corollary~\ref{cor:monotonicity}) one should be more careful, and we leave it to the interested reader or future work to clarify to what extent they can be generalized.

\bigskip

\section{Restricting KMS-states to the coefficient algebra}\label{sec:restriction}

Aiming to describe KMS-states $\phi$ on $\NT(X)$ in terms of their restrictions $\tau=\phi|_A$ to $A=X_e$, our first goal is to show that such states $\tau$ satisfy a number of inequalities.

\begin{theorem} \label{thm:subinvariance}
Assume $(G,P)$ is a quasi-lattice ordered group and $X$ is a compactly aligned product system of C$^*$-correspondences over $P$, with $X_e=A$. Consider the dynamics \eqref{eq:dynamics1} on the Nica-Toeplitz algebra  $\NT(X)$ of $X$ defined by a homomorphism $N\colon P\to (0,+\infty)$, and assume $\phi$ is a $\sigma$-KMS$_\beta$-state on $\NT(X)$ for some $\beta\in\R$. Then $\tau=\phi|_{A}$ is a tracial state on $A$ satisfying the following condition:
\begin{equation}\label{eq:subinvariance}
\tau(a)+\sum_{\emptyset\neq K\subset J} (-1)^{\vert K\vert} N(q_K)^{-\beta}\Tr_{\tau}^{q_K}(a)\geq 0\ \ \text{for all finite}\ \ J\subset P\setminus\{e\}\ \ \text{and}\ \ a\in A_+,
\end{equation}
where $q_K$ is given by \eqref{eq:qK} and the convention is that the summands corresponding to $q_K=\infty$ are zero.
\end{theorem}

To be clear, our convention for the inclusion sign is that condition $K\subset J$ allows $K=J$. When we want to exclude the case $K=J$, we will write $K\subsetneq J$.

Note that for $J=\{p\}$ condition~\eqref{eq:subinvariance} gives $N(p)^{-\beta}\Tr_\tau^p(a)\le \tau(a)$. Therefore part of this condition is that the summands in~\eqref{eq:subinvariance} are all finite and therefore there is no issue in having summands of different signs.

\smallskip

Before we turn to the proof of the theorem, let us recall a construction from~\cite{F}.
Consider the characteristic functions $\chi_{pP}$ of the sets $pP\subset P$. By the quasi-lattice assumption the set of such functions is closed under product, so the norm closure of the linear span of $\chi_{pP}$, $p\in P$, in~$\ell^\infty(P)$ is a C$^*$-algebra $B_P$. Denote by $\B_P$ the algebra of subsets of $P$ generated by the sets~$pP$, $p\in P$. Then
the set of projections in $B_P$ is precisely the set of characteristic functions of the elements of $\B_P$.

Given a representation $\pi\colon\NT(X)\to B(H)$, for every $p\in P$, denote by $f_p$ the projection onto the closed linear span of the images of $\pi(i_X(\xi))$ for all $\xi\in X_p$. Clearly, $f_p\in\pi(A)'$. By~\cite[Propositions~4.1,~5.6 and~6.1]{F}, we have also
$$
f_p f_q=f_{p\vee q},
$$
with the convention $f_\infty=0$, so we get a representation $L_\pi$ of $B_P$ on $H$ such that $L_\pi(\chi_{pP})=f_p$.

\bp[Proof of Theorem~\ref{thm:subinvariance}]
Since the dynamics is trivial on $A$, the positive linear functional $\tau=\phi|_A$ is tracial. Since the C$^*$-correspondences $X_p$ are essential by our standing assumptions on product systems, an approximate unit in $A$ is an approximate unit in $\NT(X)$, which implies that $\tau$, being the restriction of a state, is a state itself.

Next, consider the GNS-triple $(H_\phi,\pi_\phi,v_\phi)$. By the above discussion we get a representation $L=L_{\pi_\phi}\colon B_P\to B(H_\phi)$ with image in $\pi_\phi(A)'$.
Hence, for every $a\in A_+$, we get a positive (finitely additive) measure $\mu_a$ on $(P,\B_P)$ defined by
\begin{equation}\label{eq:mua0}
\mu_a(\Omega)=(\pi_\phi(a)L(\chi_\Omega)v_\phi,v_\phi) \ \ \text{for}\ \ \Omega\in\B_P.
\end{equation}
We claim that $\mu_a(\Omega)$, for $\Omega=\cap_{p\in J}(P\setminus pP)$, equals the expression on the left in~\eqref{eq:subinvariance}, so condition~\eqref{eq:subinvariance} is simply a consequence of positivity of $\mu_a$. Since
$$
\chi_\Omega=\prod_{p\in J}(1-\chi_{pP})=1+\sum_{\emptyset\neq K\subset J} (-1)^{\vert K\vert}\chi_{q_KP},
$$
our claim is equivalent to
\begin{equation}\label{eq:mua}
\mu_a(pP)=N(p)^{-\beta}\Tr^p_\tau(a)\ \ \text{for all}\ \ a\in A_+\ \ \text{and}\ \ p\in P.
\end{equation}

Observe that by the definition of $f_p=L(\chi_{pP})$, we have $f_pH_\phi=\overline{\pi_\phi(i_X^{(p)}(\K(X_p)))H_\phi}$ (where~$i_X^{(p)}$ is given by \eqref{eq:CPmap} for $\psi=i_X\colon X\to\NT(X)$). It follows that if $(u_i)_i$ is an approximate unit in~$\K(X_p)$, then the operators $\pi_\phi(i_X^{(p)}(u_i))$ converge to $f_p$ in the strong operator topology. Hence
$$
\mu_a(pP)=\lim_i\phi(a\,i_X^{(p)}(u_i)).
$$
Now, by the KMS-condition, for any $\xi,\zeta\in X_p$, we have
\begin{equation} \label{eq:LN}
\phi(a\,i_X^{(p)}(\theta_{\xi,\zeta}))=\phi(a\,i_X(\xi)i_X(\zeta)^*)=N(p)^{-\beta}\phi(i_X(\zeta)^*a\,i_X(\xi))=N(p)^{-\beta}\tau(\langle\zeta,a\xi\rangle).
\end{equation}
So if we choose $(u_i)_i$ such that $u_i=\sum_{\xi\in J_i}\theta_{\xi,\xi}$ for some finite sets $J_i\subset X$, then
$$
\mu_a(pP)=\lim_iN(p)^{-\beta}\sum_{\xi\in J_i}\tau(\langle\xi,a\xi\rangle).
$$
This gives~\eqref{eq:mua} by the definition of~$\Tr^p_\tau$.
\ep

\begin{remark}\label{rem:KMScore}
In general a KMS-state $\phi$ on $\NT(X)$ is not uniquely determined by the trace $\tau=\phi|_A$. But at least the restriction of $\phi$ to the core subalgebra $\F\subset \NT(X)$ is completely determined by $\tau$. Indeed, the elements of the form $i_X(\xi)i_X(\zeta)^*$, with $\xi,\zeta\in X_p$ and $p\in P$, span a dense subspace of $\F$, and on such elements the KMS-condition gives, as we already showed in~\eqref{eq:LN}, that
$$
\phi(i_X(\xi)i_X(\zeta)^*)=N(p)^{-\beta}\tau(\langle \zeta,\xi \rangle).
$$
\end{remark}

\begin{remark}
An analogous result for general quasi-free dynamics is proved in an essentially identical way, but a computation similar to~\eqref{eq:LN} goes smoother if one observes that the projections~$f_p$ lie in the centralizer of the normal state $\bar\phi$ on $M=\pi_\phi(\NT(X))''$ defined by $v_\phi$ and therefore the measure $\mu_a$ can be written as $\mu_a(\Omega)=(\pi_\phi(a)JL(\chi_\Omega)v_\phi,v_\phi)$, where $J$ is the modular involution defined by $\bar\phi$. The reason is that for general KMS-states the bilinear form $\phi(x\sigma_{i\beta/2}(y))=(\pi_\phi(x)J\pi_\phi(y^*)v_\phi,v_\phi)$ has better positivity properties than $\phi(xy)$.
\end{remark}

Notice that formally condition~\eqref{eq:subinvariance} does not fully exploit positivity of the measures $\mu_a$ introduced in the proof of Theorem~\ref{eq:subinvariance}, since we used only particular sets in $\B_P$ instead of all of them. The reason we formulated this condition as it stands, is that the additional positivity conditions we could get are already a consequence of~\eqref{eq:subinvariance}. This is a byproduct of Theorem~\ref{thm:general} below, but let us show this directly. The following considerations will be useful in several subsequent sections.

Recall that we defined generalized dual Ruelle transfer operators $F_p$ by~\eqref{eq:F}.
Denote by $T(A)\subset A^*$ the closed subspace of bounded linear functionals $\phi$ such that $\phi(ab)=\phi(ba)$ for all $a,b\in A$. Equivalently, $T(A)$ is the linear span of tracial states. Then, if $F_p\tau$ is a finite trace for all $p$, we can define a linear operator from the space of functions spanned by the characteristic functions of the sets in $\B_P$ into $T(A)$ by mapping $\chi_{pP}$ into $N(p)^{-\beta}F_p\tau$.
(Here we use that the characteristic functions $\chi_{pP}$ are linearly independent, which can be quickly seen as follows. Assume that $h=\sum_{p\in F} c_p\chi_{pP}=0$ for some finite set $F\subset P$ and complex numbers $c_p$, but the set $K=\{p\mid c_p\ne0\}$ is nonempty. Then by evaluating $h$ at a minimal point $p$ of the set $K$ we get $c_p=0$, which is a contradiction.)

By restricting the above linear operator to the characteristic functions of the sets $\Omega\in\B_P$ we get a finitely additive $T(A)$-valued measure $\mu$ on $(P,\B_P)$ uniquely determined by the property $\mu(pP)=N(p)^{-\beta}F_p\tau$. When $\tau=\phi|_A$ for a $\sigma$-KMS$_\beta$-state $\phi$ on $\NT(X)$, this vector-valued measure is related to the measures $\mu_a$ defined in~\eqref{eq:mua0}~via the identity
\begin{equation}\label{eq:mua1}
\mu_a(\Omega)=\mu(\Omega)(a).
\end{equation}

Condition~\eqref{eq:subinvariance} means precisely that
\begin{equation*}\label{eq:OmegaJ}
\mu(\Omega_J)\ge0,\ \ \text{where}\ \ \Omega_J=\bigcap_{p\in J}(P\setminus pP),
\end{equation*}
for all finite subsets $J\subset P\setminus\{e\}$, and we want to show that this implies that $\mu$ is positive. The key point is that property~\eqref{eq:induction2} immediately gives positivity of $\mu$ on the sets $p\Omega_J$ as well. Indeed, since
$$
\chi_{p\Omega_J}=\chi_{pP}+\sum_{\emptyset\neq K\subset J} (-1)^{\vert K\vert}\chi_{pq_KP},
$$
using property \eqref{eq:induction2} we get
\begin{equation}\label{eq:mu}
\mu(p\Omega_J)=N(p)^{-\beta}F_p\tau+\sum_{\emptyset\neq K\subset J} (-1)^{\vert K\vert}N(pq_K)^{-\beta}F_{pq_K}\tau=N(p)^{-\beta}F_p\mu(\Omega_J),
\end{equation}
from which we see that $\mu(p\Omega_J)\ge0$. Therefore positivity of $\mu$ is a consequence of part (i) of the following lemma. Part (ii) will be useful later.

\begin{lemma}\label{lem:BP-decomposition}
For any quasi-lattice ordered group $(G,P)$, we have:
\begin{itemize}
\item[(i)] every set in $\B_P$ decomposes into a finite disjoint union of the sets $pP$ and $p\Omega_J$, where $p\in P$, $J\subset P\setminus\{e\}$ is a finite nonempty set and $\Omega_J=\cap_{q\in J}(P\setminus qP)$;
\item[(ii)] for every $\Omega\in\B_P$ and $p\in \Omega$, we have $p\in p\Omega_J\subset\Omega$ for some $J$.
\end{itemize}
\end{lemma}

\bp
(i) For any $\Omega\in\B_P$ there is a finite set $F\subset P$ such that $\chi_\Omega$ is contained in the linear span of the functions $\chi_{pP}$, $p\in F$. By enlarging $F$ we may assume that $F$ is $\vee$-closed. Then the linear span of the functions $\chi_{pP}$, $p\in F$, forms an algebra. The minimal projections of this algebra are the characteristic functions of sets of the form
\begin{equation}\label{eq:sets}
\Big(\bigcap_{p\in E}pP\Big)\cap\Big(\bigcap_{q\in F\setminus E}(P\setminus qP)\Big),
\end{equation}
where $E\subset F$. Since $pP\cap qP=(p\vee q)P$, such a set is either empty or it has the required form. Indeed, this is clear if $E=\emptyset$ or $E=F$. In the remaining cases we may assume that $q_E=\vee_{p\in E}\,p<\infty$, as otherwise we get the empty set. Then the set in~\eqref{eq:sets} equals $q_EP$, if $q_E\vee q=\infty$ for all $q\in F\setminus E$, or $q_E\Omega_J$,
where $J=\{q_E^{-1}(q_E\vee q)\mid q\in F\setminus E,\ q_E\vee q<\infty\}$.

\smallskip

(ii) By part (i), if $\Omega$ contains $p$, then either $p\in qP\subset\Omega$ or $p\in q\Omega_{J'}\subset\Omega$ for some $q$ and $J'$. In either case we then have $q\le p$. Then in the first case we have $p\in pP\subset\Omega$, so $p\in p\Omega_J\subset\Omega$ for any $J$. In the second case, if $ p\vee qr=\infty$ for all $r\in J'$, then $pP\cap q\Omega_{J'}=pP$, so again $p\Omega_J\subset\Omega$ for any $J$. Otherwise we have $pP\cap q\Omega_{J'}=p\Omega_J$, where $J=\{p^{-1}(p\vee qr)\mid r\in J',\ p\vee qr<\infty\}$, and therefore $p\in p\Omega_J\subset\Omega$.
\ep

We finish the section with a discussion of an important class of KMS-states for which Theorem~\ref{thm:subinvariance} does not give any nontrivial information.

Assume $X$ is a not necessarily compactly aligned product system over $P$ and  $\tau_0$ is a positive trace on $A=X_e$ such that $\sum_{p\in P}N(p)^{-\beta}\Tr^p_{\tau_0}(1)=1$. Recall that on the Fock module level the dynamics $\sigma$ is implemented by the unitaries $D_N^{it}$, where $D_N$ is given by~\eqref{eq:DN}. By our assumption we have $\Tr^{\F(X)}_{\tau_0}(D_N^{-\beta})=1$. Hence, following~\cite{lac-nes}, we can consider the state
$$
\Tr^{\F(X)}_{\tau_0}(\cdot D_N^{-\beta})
$$
on $\LL(\F(X))$, which we call a \emph{generalized Gibbs state}. The tracial property of $\Tr^{\F(X)}_{\tau_0}$ implies that it satisfies the KMS$_\beta$ condition with respect to the dynamics~$\Ad D_N^{it}$. To be more precise, there are no issues with the domain of definition of $\Tr^{\F(X)}_{\tau_0}(\cdot D_N^{-\beta})$ and the KMS-property when the operator $D_N^{-\beta}$ is bounded. In the general case it is more fruitful to think of $\Tr^{\F(X)}_{\tau_0}(\cdot D_N^{-\beta})$ as $\operatorname{Ind}^{(D_N^{it})_t}_{\F(X)}\tau_0$, that is, as the lower semicontinuous extension of the KMS$_\beta$-weight induced from~$\tau_0$ (viewed as a KMS$_\beta$-state with respect to the trivial dynamics) using the Fock module $\F(X)$ and the one-parameter unitary group $(D_N^{it})_t$. But even then we will use the more suggestive notation $\Tr^{\F(X)}_{\tau_0}(\cdot D_N^{-\beta})$.

Since $\sigma_t=\Ad D_N^{it}$ on $\NTr(X)$, by restriction we get a $\sigma$-KMS$_\beta$-state on $\NTr(X)$. If $X$ is compactly aligned, then by composing with $\Lambda\colon\NT(X)\to\NTr(X)$ we get a $\sigma$-KMS$_\beta$-state $\Tr^{\F(X)}_{\tau_0}(\Lambda(\cdot)D_N^{-\beta})$ on $\NT(X)$.
The restriction $\tau$ of either of these states to $A$ is given by
$$
\tau(a)=\sum_{p\in P}N(p)^{-\beta}(F_p\tau_0)(a)\ \ \text{for}\ \ a\in A.
$$
The conclusion is that any tracial state of the form $\tau=\sum_{p\in P}N(p)^{-\beta}F_p\tau_0$ extends to a $\sigma$-KMS$_\beta$-state on $\NTr(X)$ or $\NT(X)$. As we will see later in Corollary~\ref{cor:finite-infinite}(ii), if $X$ is compactly aligned, then these extensions are unique.

It is immediate that tracial states of the form $\tau=\sum_{p\in P}N(p)^{-\beta}F_p\tau_0$ satisfy condition~\eqref{eq:subinvariance}. Indeed, by Lemma~\ref{lem:trace-continuity}(i) and~\eqref{eq:induction2} we have
$$
N(p)^{-\beta}F_p\tau=\sum_{q\in pP}N(q)^{-\beta}F_q\tau_0,
$$
and from this it follows that for the $T(A)$-valued measure $\mu$ defined by $\mu(pP)=N(p)^{-\beta}F_p\tau$ we have
\begin{equation}\label{eq:mu2}
\mu(\Omega)=\sum_{p\in\Omega}N(p)^{-\beta}F_p\tau_0\ \ \text{for all}\ \ \Omega\in\B_P,
\end{equation}
which is obviously positive.

\bigskip

\section{KMS-states on reduced Toeplitz algebras: the case of free abelian monoids}\label{sec:abelian}

In this section we consider the monoid $P=\Z^n_+$ ($\Z_+=\{0,1,2,\dots\}$) for $n\ge1$. Denote by $e_1,\dots,e_n$ the standard generators of $\Z^n_+$.

We will write $F_i$ instead of~$F_{e_i}$ for the operators~\eqref{eq:F}. Then by~\eqref{eq:induction}, for any $i_1,\dots,i_k\in\{1,\dots,n\}$, we have
\begin{equation}\label{eq:freeF}
F_{i_1}\dots F_{i_k}\tau=F_{e_{i_1}+\dots+e_{i_k}}\tau
\end{equation}
whenever the trace $F_{i_j}F_{i_{j+1}}\dots F_{i_k}\tau$ is finite for all $j=2,\dots,k$. In particular, $F_iF_j\tau=F_jF_i\tau$ whenever both sides are well-defined.

\begin{theorem}\label{thm:free}
Assume $X$ is a product system of C$^*$-correspondences over $\Z^n_+$, $X_e=A$. Consider the dynamics \eqref{eq:dynamics1} on the reduced Toeplitz algebra  $\NTr(X)$ of $X$ defined by a homomorphism $N\colon \Z^n_+\to (0,+\infty)$. Assume $\beta\in\R$ and $\tau$ is a tracial state on $A$ satisfying the following condition:
\begin{equation}\label{eq:subinvariance2}
\prod_{i\in J}(1-N(e_i)^{-\beta}F_i)\tau\ge0\ \ \text{for all nonempty subsets}\ \ J\subset\{1,\dots,n\}.
\end{equation}
Then there exists a gauge-invariant $\sigma$-KMS$_\beta$-state $\phi$ on $\NTr(X)$ such that $\phi|_A=\tau$.
\end{theorem}

Note that if $i_1,\dots,i_k$ are all different, then
$$
e_{i_1}+\dots+ e_{i_k}=e_{i_1}\vee\dots\vee e_{i_k}.
$$
Together with \eqref{eq:freeF} this shows that condition \eqref{eq:subinvariance2} is nothing else than condition \eqref{eq:subinvariance} applied only to the subsets of $\{e_1,\dots,e_n\}$.

For $n=1$ and the homomorphism $N\colon\Z_+\to(0,+\infty)$ defined by $N(1)=e$, condition~\eqref{eq:subinvariance2} is exactly condition~\eqref{eq:LN0} discussed in the introduction.
For $n\ge 2$ and particular product systems condition~\eqref{eq:subinvariance2} has appeared in~\cite{aHLRS} and~\cite{AaHR} as a necessary condition for a trace to be extendable to a KMS-state. It is called the \emph{subinvariance relation} in these papers. This name does not seem to be suitable for the more general condition~\eqref{eq:subinvariance}.

\bp[Proof of Theorem~\ref{thm:free}]
The case $n=1$ follows from \cite[Theorem~2.1]{lac-nes}. For $n\ge2$ we can use similar arguments.

Replacing $N$ by $N^\beta$ we may assume that $\beta=1$. If we also assume that there exists a positive trace $\tau_0$ on $A$ such that
$$
\tau=\sum_{p\in\Z^n_+}N(p)^{-1}F_p\tau_0=\sum_{k_1,\dots,k_n=0}^\infty N(e_1)^{-k_1}\dots N(e_n)^{-k_n} F_1^{k_1}\dots F^{k_n}_n\tau_0,
$$
then by the discussion at the end of the previous section we get the required KMS-state by restricting the generalized Gibbs state $\Tr^{\F(X)}_{\tau_0}(\cdot D_N^{-1})$ to~$\NTr(X)$.

For an arbitrary tracial state $\tau$ satisfying~\eqref{eq:subinvariance2}, take $\eps>0$ and consider the perturbed dynamics $\sigma^\eps$ defined by the homomorphism $N_\eps\colon\Z^n_+\to(0,\infty)$ such that $N_\eps(e_i)=e^{\eps}N(e_i)$ for $i=1,\dots,n$. Consider the trace
$$
\tau_0^\eps=\prod^n_{i=1}(1-N_\eps(e_i)^{-1}F_i)\tau.
$$
This trace is positive, since by writing $1-N_\eps(e_i)^{-1}F_i$ as
$$
(1-N(e_i)^{-1}F_i)+(1-e^{-\eps})N(e_i)^{-1}F_i,
$$
we see that
$$
\tau_0^\eps=\sum_{J\subset\{1,\dots,n\}}(1-e^{-\eps})^{n-|J|}\Big(\prod_{j\in\{1,\dots,n\}\setminus J}N(e_j)^{-1}F_j\Big)\Big(\prod_{i\in J}(1-N(e_i)^{-1}F_i)\Big)\tau,
$$
which is positive by \eqref{eq:subinvariance2}.

We claim that
$$
\tau=\sum_{p\in\Z^n_+}N_\eps(p)^{-1}F_p\tau^\eps_0.
$$
To see this, take $m\ge1$ and consider the partial sum over all $p\in\Z^n_+$ with coordinates not greater than $m$, which rewrites as
$$
\prod^n_{i=1}\left(\sum^m_{k=1}N_\eps(e_i)^{-k}F^k_i\right)\tau_0^\eps=\prod^n_{i=1}(1-e^{-(m+1)\eps}N(e_i)^{-(m+1)}F^{m+1}_i)\tau.
$$
As $m\to\infty$ these partial sums converge in norm to $\tau$, since $N(e_i)^{-1}F_i\tau\le\tau$ and $e^{-(m+1)\eps}\to0$. This implies our claim.

By the first part of the proof we conclude that there exists a gauge-invariant $\sigma^\eps$-KMS$_1$-state~$\phi^\eps$ on $\NTr(X)$ such that $\phi^\eps|_A=\tau$. Then any weak$^*$ cluster point of the states $\phi^\eps$ as $\eps\to0$ can be taken as the required state $\phi$.
\ep

\begin{remark}\label{rem:criterion}
The above proof shows that, in the setting of Theorem~\ref{thm:free}, if a tracial state~$\tau$ on~$A$ satisfies~\eqref{eq:subinvariance2} and has the property
\begin{equation}\label{eq:criterion}
\lim_{p\to\infty}N(p)^{-\beta}\Tr^p_\tau(1)=0,
\end{equation}
then $\tau=\sum_{p\in \Z^n_+}N(p)^{-\beta}F_p\tau_0$ for the positive trace $\tau_0=\prod^n_{i=1}(1-N(e_i)^{-\beta}F_i)\tau$. Conversely, it is easy to see that if $\tau=\sum_{p\in \Z^n_+}N(p)^{-\beta}F_p\tau_0$ for a positive trace $\tau_0$, then~\eqref{eq:criterion} holds and $\tau_0=\prod^n_{i=1}(1-N(e_i)^{-\beta}F_i)\tau$. Therefore we have a simple criterion of presentability of $\tau$ in the form $\sum_{p\in \Z^n_+}N(p)^{-\beta}F_p\tau_0$, cf.~\cite{lac-nes}. This criterion and the abundance of homomorphisms $\Z^n_+\to(0,\infty)$, allowing us to easily modify the dynamics in order to enforce~\eqref{eq:criterion}, are what make the case $P=\Z^n_+$ special and the above proof work. We will return to the study of traces of the form $\sum_{p\in P}N(p)^{-\beta}F_p\tau_0$ for general quasi-lattice ordered groups in Section~\ref{sec:critical-temp}.
\end{remark}

If $X$ is in addition compactly aligned, then $\NT(X)$ is well-defined and we get the following result.

\begin{corollary}\label{cor:classfree}
Assume $X$ is a compactly aligned product system of C$^*$-correspondences over the monoid~$\Z^n_+$, $X_e=A$, and consider the dynamics \eqref{eq:dynamics1} on $\NT(X)$ defined by a homomorphism $N\colon \Z^n_+\to (0,+\infty)$. Then, for every $\beta\in\R$, the map $\phi\mapsto\phi|_A$ defines a one-to-one correspondence between the gauge-invariant $\sigma$-KMS$_\beta$-states on $\NT(X)$ and the tracial states on $A$ satisfying~\eqref{eq:subinvariance2}.
\end{corollary}

\bp
If $\phi$ is a gauge-invariant $\sigma$-KMS$_\beta$-state on $\NT(X)$, then by Theorem~\ref{thm:subinvariance} its restriction~$\tau$ to~$A$ is a tracial state satisfying~\eqref{eq:subinvariance}. By the discussion right after the formulation of Theorem~\ref{thm:free}, condition \eqref{eq:subinvariance2} is formally weaker than \eqref{eq:subinvariance}. By the gauge-invariance, $\phi$ is determined by its restriction to the core subalgebra~$\F$, while by Remark~\ref{rem:KMScore} this restriction is completely determined by~$\tau$. Therefore the map $\phi\mapsto\tau=\phi|_A$ is an injective map from the set of gauge-invariant $\sigma$-KMS$_\beta$-states on~$\NT(X)$ into the set of tracial states on $A$ satisfying~\eqref{eq:subinvariance2}. By Theorem~\ref{thm:free} this map is also surjective, since $\NTr(X)$ is a quotient of $\NT(X)$ (in fact, as we know, the two algebras coincide by amenability of $\Z^n$).
\ep

As a byproduct we see that for $P=\Z^n_+$ and compactly aligned product systems conditions~\eqref{eq:subinvariance} and \eqref{eq:subinvariance2} are equivalent. This can be shown directly and without using the compact alignment assumption. We will return to this conclusion and prove a more general result in Section~\ref{sec:Artin}.

In Example~\ref{ex:optimal} we will show that in general the system of inequalities~\eqref{eq:subinvariance2} cannot be replaced by yet a smaller system.

\bigskip

\section{States on quasi-lattice graded algebras}\label{sec:graded}

In order to prove an analogue of Theorem~\ref{thm:free} for general quasi-lattice ordered groups we will extend a trace $\tau$ satisfying~\eqref{eq:subinvariance} to the core subalgebra. The main problem will be to show that such an extension is positive, and in this section we prove a result which will allow us to do that. It will be convenient to work in a more abstract setting than that of Nica-Toeplitz algebras.

\smallskip

By a \emph{quasi-lattice} we mean a partially ordered set $I$ such that any two elements $p,q\in I$ either have a least upper bound, which we denote $p\vee q$, or do not have a common upper bound at all, in which case we write $p\vee q=\infty$.

\begin{definition}\label{def:qloalgebra}
A \emph{quasi-lattice graded C$^*$-algebra} is a C$^*$-algebra $\F$ together with a dense $*$-subalgebra $\B=\oplus_{p\in I}B_p$, which we call the \emph{algebraic core} of $\F$, such that $I$ is a quasi-lattice and the following conditions are satisfied:
\begin{itemize}
\item[(a)] each space $B_p$, $p\in I$, is a C$^*$-subalgebra of $\F$;

\item[(b)]  for all $p,q\in I$, we have $B_pB_q\subset B_{p\vee q}$, with the convention $B_\infty=0$;

\item[(c)]  for all $q\le p$, we have $\overline{B_qB_p}=B_p$.
\end{itemize}
\end{definition}

The following useful observation is a straightforward generalization of~\cite[Lemma~3.6]{CLSV}.

\begin{lemma}\label{lem:clsv}
If $\F$ is a quasi-lattice graded C$^*$-algebra with algebraic core $\B=\oplus_{p\in I}B_p$, then for every finite $\vee$-closed subset $J$ of $I$, the space $B_J=\oplus_{p\in J} B_p$ is a C$^*$-subalgebra of $\F$.
\end{lemma}

\bp
The proof relies only on conditions (a) and (b) and uses induction on the size of $J$ along the same lines as that of \cite[Lemma~3.6]{CLSV}. (The key point is that the sum of a C$^*$-subalgebra and a closed ideal of a C$^*$-algebra is always closed.)
\ep

Assume now that for every $p\in I$ we are given a positive linear
functional $\phi_p$ on $B_p$. Our goal is to show that an analogue of \eqref{eq:subinvariance} provides a sufficient condition for positivity of the linear functional $\phi=\oplus_p\phi_p$ on $\B$.
We need to introduce some notation to formulate the precise result.

Extend $\phi_p$ to a linear functional $\psi_p$ on $\oplus_q B_q$ as follows.
If $q\not\le p$, then we put $\psi_p|_{B_q}=0$. On the other hand,
if $q\le p$, then by assumption (b) we have a $*$-homomorphism
$B_q\to M(B_p)$. Composing the canonical extension of $\phi_p$ to a
positive linear functional on $M(B_p)$ with
this homomorphism, we get a positive linear functional on $B_q$. We
take this functional as $\psi_p|_{B_q}$.

\begin{lemma}\label{lem:positivity}
For any choice of positive linear functionals $\phi_p$ on $B_p$, the linear functionals $\psi_p$ are positive on $\B$, that is, $\psi_p(x^*x)\ge0$ for all $x\in\B$.
\end{lemma}

\bp Fix $p\in I$. First of all observe that since the closure of
$\oplus_{q\le p}B_q$ in $\F$ is a C$^*$-algebra and $B_p$ is its ideal, the functional $\phi_p$ extends in a canonical way to a positive linear
functional on this C$^*$-algebra. On the dense $*$-subalgebra
$\oplus_{q\le p}B_q$ this positive linear functional coincides with
$\psi_p$ by construction.

Next, observe that if $y\in B_q$ for some $q\not\le p$, then $q\vee
r\not\le p$ for any $r\in I$, and hence $\psi_p(xy)=0$ for all
$x\in\B$.

Now, take any $x\in\oplus_{q\le p}B_p$ and $y\in\oplus_{q\not\le
p}B_q$. Then by the above two observations we have
$$
\psi_p((x+y)^*(x+y))=\psi_p(x^*x)\ge0.
$$
Thus $\psi_p$ is indeed positive on $\B$. \ep

\begin{proposition} \label{prop:positivity}
Assume $\F$ is a quasi-lattice graded C$^*$-algebra with algebraic core $\B=\oplus_{p\in I}B_p$. Assume that we are given positive linear functionals  $\phi_p$ on $B_p$ such that, for every $p\in I$ and every finite subset $J\subset\{q|q> p\}$,
we have
\begin{equation} \label{eq:qlopos}
\phi_p(b)+\sum_{\emptyset\ne K\subset
J}(-1)^{|K|}\psi_{q_K}(b)\ge0\ \ \text{for all positive elements}\ \ b\in B_p,
\end{equation}
where the functionals $\psi_{q_K}$ are defined as explained above, with the convention $\psi_\infty=0$.
Then the linear functional $\phi=\oplus_p\phi_p$ on $\B$ is
positive. Furthermore, if $I$ has a smallest element $e$, and $\phi_e$ is a state on $B_e$, then $\phi$ extends to a state on $\F$.
\end{proposition}

For the proof we need the following lemma.

\begin{lemma}\label{lem:positivity2}
Assume $J$ is a finite quasi-lattice and $f$ is a complex-valued function on~$J$.
Define a new function~$\hat f$ on~$J$~by
$$
\hat f(p)=f(p)+\sum_{\emptyset\ne K\subset \{q\in J\mid q>
p\}}(-1)^{|K|}f(q_K),
$$
with the convention $f(\infty)=0$. Then, for every $p\in J$, we have
$$
f(p)=\sum_{q\ge p}\hat f(q).
$$
\end{lemma}

\bp
For every $p\in J$, denote by $J_p$ the set $\{q\mid q\ge p\}$. The characteristic functions $\chi_{J_p}$ of these sets, being linearly independent, form a basis of the space of functions on $J$.
It follows that there exists a unique complex measure $\mu$ on $J$ such that $\mu(J_p)=f(p)$. For every $p\in J$, we have $\{p\}=J_p\setminus\cup_{q>p}J_q$, so that
$$
\chi_{\{p\}}=\prod_{q>p}(\chi_{J_p}-\chi_{J_q})=\chi_{J_p}+\sum_{\emptyset\ne K\subset \{q\in J\mid q>
p\}}(-1)^{|K|}\chi_{J_{q_K}}.
$$
From this we see that $\mu(\{p\})=\hat f(p)$ for all $p\in J$. Then
$$
f(p)=\mu(J_p)=\sum_{q\in J_p}\mu(\{q\})=\sum_{q\in J_p}\hat f(q),
$$
which is what we need.
\ep

\bp[Proof of Proposition~\ref{prop:positivity}]
By Lemma~\ref{lem:clsv}, it suffices to show that $\phi$ is positive on the $*$-subalgebra $$B_J=\oplus_{p\in J}B_p\subset\B$$ for all finite $\vee$-closed subsets $J\subset I$. With such a $J$ fixed, for every $p\in J$ define a linear functional $\eta_p$ on $B_J$ as follows. If $q\not\le p$, then we let $\eta_p|_{B_q}=0$, and if $q\le p$, we put
$$
\eta_p(x)=\psi_p(x)+\sum_{\emptyset\ne K\subset
\{r\in J\mid r>
p\}}(-1)^{|K|}\psi_{q_K}(x)\ \ \text{for}\ \ x\in B_q.
$$

We claim that the functionals $\eta_p$ are positive. Indeed, by our assumption~\eqref{eq:qlopos}, the linear functional $\eta_p|_{B_p}$ is positive. By Lemma~\ref{lem:positivity}, applied to the set $J$ and the functional~$\eta_p|_{B_p}$ instead of $I$ and~$\phi_p$, this functional extends to a positive linear functional $\omega_p$ on~$B_J$. Specifically, we have $\omega_p|_{B_q}=0$ for $q\not\le p$, and
$$
\omega_p(x)=\lim_i\eta_p(xu_i)
$$
for $x\in B_q$ with $q\le p$, where $\{u_i\}_i$ is an approximate unit in $B_p$. By assumption (c) in Definition~\ref{def:qloalgebra}, for every $r\ge p$, the images of $u_i$ in $M(B_r)$ converge strictly to $1$. Hence $\psi_r(x)=\lim_i\psi_r(xu_i)$ for all $x\in B_q$ with $q\le p$. From this we conclude that $\omega_p=\eta_p$, hence $\eta_p$ is indeed positive.

Finally, we claim that $\phi|_{B_J}=\sum_{p\in J}\eta_p$, which clearly implies positivity of $\phi|_{B_J}$. Indeed, take $s\in J$, $x\in B_s$ and consider the function $f$ on $J$ defined by $f(p)=\psi_p(x)$. Using the notation of Lemma~\ref{lem:positivity2}, we have $\eta_p(x)=\hat f(p)$ for all $p\ge s$. Hence, by that lemma,
$$
\sum_{p\in J\colon p\ge s}\eta_p(x)=f(s)=\psi_s(x)=\phi_s(x).
$$
On the other hand, by definition, $\eta_p(x)=0$ for $p\not\ge s$. Therefore $\sum_{p\in J}\eta_p(x)=\phi_s(x)$, and our claim is proved.

\smallskip

It remains to prove the last statement of the proposition. By Lemma~\ref{lem:clsv}, for every finite $\vee$-closed subset $J$ of $I$ containing $e$, the $*$-subalgebra $B_J$ of $\F$ is closed. Since $e$ is assumed to be the smallest element of $I$, by assumption (c) an approximate unit of $B_e$ is an approximate unit of $B_J$. Since we already know that $\phi|_{B_J}$ is a positive linear functional, it follows that
$$
\|\phi|_{B_J}\|=\|\phi|_{B_e}\|=\|\phi_e\|=1,
$$
so $\phi|_{B_J}$ is a state. Since $\F$ is the closure of the union of such C$^*$-subalgebras $B_J$, we conclude that $\phi$ extends by continuity to a state on $\F$.
\ep

It might be useful also to observe the following, although in our applications this will be subsumed by stronger statements.

\begin{proposition}
Assume $\F$ is a quasi-lattice graded C$^*$-algebra with algebraic core $\B=\oplus_{p\in I}B_p$, and we are given tracial positive linear functionals  $\phi_p$ on the C$^*$-algebras $B_p$. Then the linear functional $\phi=\oplus_p\phi_p$ on $\B$ is tracial as well.

More generally, if we are given a one-parameter group $(\sigma_t)_{t\in\R}$ of automorphisms of $\F$ leaving the subalgebras $B_p$ invariant and $\sigma$-KMS$_\beta$ positive linear functionals $\phi_p$ on $B_p$, then $\phi$ is $\sigma$-KMS$_\beta$, meaning that $\phi(ab)=\phi(b\sigma_{i\beta}(a))$ for all $\sigma$-analytic elements $a,b\in\B$.
\end{proposition}

\bp
Take $x\in B_q$ and $y\in B_r$. We want to show that $\phi(xy)=\phi(yx)$. If $q\vee r=\infty$, then $xy=yx=0$ and there is nothing to prove. So assume $p=q\vee r<\infty$. As we already used in the proof of Lemma~\ref{lem:positivity}, $B_p$ is an ideal of the C$^*$-algebra $\overline{\oplus_{s\le p}B_s}$. Hence $\phi_p$ extends in a canonical way to a tracial positive linear functional $\psi_p$ on this C$^*$-algebra. Then
$$
\phi(xy)=\phi_p(xy)=\psi_p(xy)=\psi_p(yx)=\phi_p(yx)=\phi(yx).
$$

The second part of the proposition is proved in a similar way. We just want to remark that since the subspaces $B_J=\oplus_{p\in J}B_p\subset\F$ are closed for finite $\vee$-closed sets $J$ by Lemma~\ref{lem:clsv}, the topology on them coincides with the direct product topology, and this implies that an element of $\B$ is $\sigma$-analytic in~$\F$ if and only if its $B_p$-components are $\sigma$-analytic for all $p$.
\ep

\bigskip

\section{KMS-states on Nica-Toeplitz algebras: the general case}\label{sec:general}

Using the results of the previous section we can now prove an analogue of Corollary~\ref{cor:classfree} for arbitrary quasi-lattice ordered groups.

\begin{theorem} \label{thm:general}
Assume $(G,P)$ is a quasi-lattice ordered group and $X$ is a compactly aligned product system of C$^*$-correspondences over $P$, with $X_e=A$. Consider the dynamics \eqref{eq:dynamics1} on $\NT(X)$ defined by a homomorphism $N\colon P\to (0,+\infty)$. Then, for every $\beta\in\R$, the map $\phi\mapsto\phi|_A$ defines a one-to-one correspondence between the gauge-invariant $\sigma$-KMS$_\beta$-states on $\NT(X)$ and the tracial states on $A$ satisfying~\eqref{eq:subinvariance}.
\end{theorem}

\bp
We only have to prove surjectivity of the map $\phi\mapsto\phi|_A$, as injectivity follows from Remark~\ref{rem:KMScore}. Thus we have to show that every tracial state $\tau$ satisfying \eqref{eq:subinvariance} extends to a gauge-invariant $\sigma$-KMS$_\beta$-state $\phi$ on $\NT(X)$.

In order to define $\phi$ we first extend $\tau$ to the core subalgebra $\F$ and then use the gauge-invariant conditional expectation $\NT(X)\to\F$ to define $\phi$ on the whole algebra $\NT(X)$. By Remark~\ref{rem:KMScore} we know that the required extension on $\F$ must be given by
\begin{equation}\label{eq:KMScore}
\phi(i_X(\xi) i_X(\zeta)^*)=N(p)^{-\beta}\tau(\langle\zeta,\xi\rangle)\ \ \text{for}\ \ \xi,\zeta\in X_p.
\end{equation}
In order to see that such a state indeed exists we apply Proposition~\ref{prop:positivity}. For this let $B_p=i_X^{(p)}(\K(X_p))\subset\F$ for $p\in P$, where, as before, $i_X^{(p)}$ denotes the canonical injective homomorphism $\K(X_p)\to\NT(X)$, see~\eqref{eq:CPmap}. As the positive linear functionals $\phi_p$ we take
$$
\phi_p(i_X^{(p)}(T))=N(p)^{-\beta}\Tr^p_\tau(T)\ \ \text{for}\ \ T\in\K(X_p).
$$

We have to check that condition~\eqref{eq:qlopos} is satisfied for all $p\in P$ and all finite $J\subset\{q|q>p\}$. For $p=e$ this is condition~\eqref{eq:subinvariance}. For general $p$ it suffices to establish~\eqref{eq:qlopos} on all elements of the form $b=i_X(\xi) i_X(\xi)^*\in B_p$, since finite sums of the elements $\theta_{\xi,\xi}$ are dense in $\K(X_p)_+$. We have
\begin{multline*}
\phi_p(i_X(\xi) i_X(\xi)^*)+\sum_{\emptyset\ne K\subset
J}(-1)^{|K|}\psi_{q_K}(i_X(\xi) i_X(\xi)^*)\\
=N(p)^{-\beta}\Tr^p_\tau(\theta_{\xi,\xi})+\sum_{\emptyset\ne K\subset
J}(-1)^{|K|}N(q_K)^{-\beta}\Tr^{X_p\otimes_A X_{p^{-1}q_K}}_\tau(\theta_{\xi,\xi}\otimes1).
\end{multline*}
By definition and property \eqref{eq:induction} of induced traces this equals
$$
N(p)^{-\beta}\tau(\langle\xi,\xi\rangle)+\sum_{\emptyset\ne K\subset
J}(-1)^{|K|}N(q_K)^{-\beta}\Tr^{p^{-1}q_K}_\tau(\langle\xi,\xi\rangle),
$$
and this is positive by condition~\eqref{eq:subinvariance} applied to $a=\langle\xi,\xi\rangle$ and the set $p^{-1}J$. Thus a gauge-invariant state $\phi$ satisfying \eqref{eq:KMScore} indeed exists.

\smallskip

It remains to check the KMS-condition for $\phi$. For this we apply Lemma~\ref{lem:cognoscenti} to the analytic elements $a=i_X(\xi)$, $\xi\in X_p$, generating $\NT(X)$ as a C$^*$-algebra. Taking $b=i_X(\zeta) i_X(\eta)^*$, with $\zeta\in X_q$ and $\eta\in X_r$, we have to check that $\phi(ab)=N(p)^{-\beta}\phi(ba)$. Consider two cases.

\smallskip

\underline{Case 1}: $pq\neq r$. Then $ab$ and $ba$ are of degrees $pqr^{-1}$ and $qr^{-1}p$, respectively, so they are in the kernel of the gauge-invariant conditional expectation $\NT(X)\to\F$ and therefore the equality $\phi(ab)=N(p)^{-\beta}\phi(ba)$ holds for trivial reasons.

\smallskip

\underline{Case 2}: $pq=r$. Then
$$
\phi(ab)=\phi(i_X(\xi\zeta)i_X(\eta)^*)=N(pq)^{-\beta}\tau(\langle\eta,\xi\zeta\rangle).
$$
Using~\eqref{eq:Toeplitz} we also get
\begin{align*}
\phi(ba)&=\phi(i_X(\zeta)i_X(\eta)^*i_X(\xi))=\phi(i_X(\zeta)i_X(\ell(\xi)^*\eta)^*)\\
&=N(q)^{-\beta}\tau(\langle \ell(\xi)^*\eta,\zeta\rangle)=
N(q)^{-\beta}\tau(\langle \eta,\xi\zeta\rangle),
\end{align*}
so we have $\phi(ab)=N(p)^{-\beta}\phi(ba)$ in this case too.
\ep

\begin{remark}
The above theorem describes only gauge-invariant KMS-states. If one is interested in the set of all KMS-states, it is still a good idea to start with the gauge-invariant ones. The reason is that if $\phi$ is a $\sigma$-KMS$_\beta$-state on $\NT(X)$ for some $\beta\in\R$, then by Theorems~\ref{thm:subinvariance} and~\ref{thm:general} there exists a unique gauge-invariant $\sigma$-KMS$_\beta$-state $\psi$ on $\NT(X)$ such that $\psi=\phi$ on~$A$, moreover, we then have $\psi=\phi$ on $\F$. The existence of $\psi$ does not actually require any of the above results, since it is not difficult to check directly that the state $\psi=\phi\circ E$, where $E=(\iota\otimes\tau_G)\circ\delta\colon \NT(X)\to\F$ is the gauge-invariant conditional expectation, satisfies the $\sigma$-KMS$_\beta$-condition.
\end{remark}

Let us now consider the implications of Theorem~\ref{thm:general} for the quotient $\NOs(X)$ of $\NT(X)$.

\begin{corollary}\label{cor:cuntzalg}
In the setting of Theorem~\ref{thm:general}, for every $\beta\in\R$, the map $\phi\mapsto\phi|_A$ defines a one-to-one correspondence between the gauge-invariant $\sigma$-KMS$_\beta$-states on~$\NOs(X)$ and the tracial states on~$A$ such that~\eqref{eq:subinvariance} is satisfied and
$$
N(p)^{-\beta}\Tr^p_\tau(a)=\tau(a)\ \ \text{for all}\ \ a\in\varphi_p^{-1}(\K(X_p)),\ p\in P,
$$
where $\varphi_p\colon A\to\LL(X_p)$ is the map defining the left $A$-module structure on $X_p$.
\end{corollary}

\bp
If $\phi$ is a $\sigma$-KMS$_\beta$-state on $\NT(X)$ and $\tau=\phi|_A$, then it follows from \eqref{eq:LN} that
$$
N(p)^{-\beta}\Tr^p_\tau(S)=\phi(i_X^{(p)}(S))\ \ \text{for all}\ \ S\in\K(X_p).
$$
Therefore the condition $N(p)^{-\beta}\Tr^p_\tau(a)=\tau(a)$ for $a\in\varphi_p^{-1}(\K(X_p))$ means precisely that $$\phi\big(a-i_X^{(p)}(\varphi_p(a))\big)=0.$$ By Lemma~\ref{lem:TvsP} the latter condition is necessary and sufficient for $\phi$ to factor through $\NOs(X)$. The result now follows from Theorem~\ref{thm:general}.
\ep

Under extra assumptions this corollary takes a particularly simple form.

\begin{corollary}
In the setting of Theorem~\ref{thm:general}, assume in addition that~$P$ is directed and $\varphi_p(A)\subset\K(X_p)$ for all $p\in P$. Fix a generating set $S\subset P\setminus\{e\}$. Then, for every $\beta\in\R$, the map $\phi\mapsto\phi|_A$ defines a one-to-one correspondence between the gauge-invariant $\sigma$-KMS$_\beta$-states on $\NOs(X)$ and the tracial states $\tau$ on $A$ such that
\begin{equation}\label{eq:invariance}
N(p)^{-\beta}\Tr^p_\tau(a)=\tau(a)
\end{equation}
for all $a\in A$ and $p\in S$.
\end{corollary}

\bp
By Corollary~\ref{cor:cuntzalg} we just have to check that if $\tau$ is a tracial state such that \eqref{eq:invariance} holds for all $a\in A$ and $p\in S$, then \eqref{eq:invariance} holds for all $p\in P$ and~\eqref{eq:subinvariance} is satisfied as well.

The first assertion follows from~\eqref{eq:induction2}. For the second one, take a finite subset $J\subset P\setminus\{e\}$. Then, using that $q_K<\infty$ for all $\emptyset\neq K\subset J$ by assumption, for every $a\in A$ we get
$$
\tau(a)+\sum_{\emptyset\neq K\subset J} (-1)^{\vert K\vert} N(q_K)^{-\beta}\Tr_{\tau}^{q_K}(a)
=\tau(a)\sum_{K\subset J} (-1)^{\vert K\vert}=0,
$$
so~\eqref{eq:subinvariance} is satisfied.
\ep


\bigskip

\section{A Wold-type decomposition of KMS-states and traces}\label{sec:Wold}

In this section we will clarify and extend the decomposition of KMS-states of Pimsner-Toeplitz algebras into finite and infinite parts from~\cite{lac-nes} to Nica-Toeplitz algebras.

\smallskip

We start with the following surely known result.

\begin{proposition}\label{prop:Wold-rep}
Assume $(G,P)$ is a quasi-lattice ordered group, $X$ is a compactly aligned product system of C$^*$-correspondences over $P$, with $X_e=A$, and $\pi\colon\NT(X)\to B(H)$ is a representation of $\NT(X)$. Consider the von Neumann algebra $M=\pi(\NT(X))''$. Then there exists a central projection $z\in M$ such that $zH$ is the space of the largest subrepresentation of~$\pi$ induced from $A$ by the Fock module $\F(X)$.
\end{proposition}

\bp
Let $Q$ be the projection onto the space
\begin{equation}\label{eq:vacuum}
\{v\in H\colon \pi(i_X(\xi))^*v=0\ \text{for all}\ \xi\in X_p,\ p\ne e\}
\end{equation}
of ``vacuum vectors''. The space $QH$ is invariant under $M'$, hence $Q\in M$. It is also clear that the space $QH$ is invariant under $\pi(A)$.

By virtue of~\eqref{eq:Nica0} and~\eqref{eq:Nica} (for $\psi=i_X$), if $\xi\in X_p$, $\zeta\in X_q$ and $p\ne q$, then the spaces $i_X(\xi)QH$ and $i_X(\zeta)QH$ are orthogonal. It follows that the map
$$
U\colon\F(X)\otimes_A QH\to H,\ \ \xi\otimes v\mapsto \pi(i_X(\xi))v,
$$
is isometric. Its image is obviously invariant under the operators $\pi(i_X(\xi))$, but using again~\eqref{eq:Nica0}-\eqref{eq:Nica} we see that it is invariant under the operators $\pi(i_X(\xi))^*$ as well. Therefore the image of~$U$ coincides with the invariant subspace $\overline{MQH}\subset H$, so the restriction of $\pi$ to this subspace is unitarily equivalent to the representation induced by the Fock module from the representation of $A$ on $QH$.

Since $Q\in M$, the projection $z$ onto $\overline{MQH}$ is a central projection in $M$. Finally, if a subrepresentation of $\pi$ is induced by the Fock module, then its space of vacuum vectors is contained in $QH$ and hence the entire space of the subrepresentation is contained in $zH$.
\ep

The property of a representation of $\NT(X)$ to be induced by the Fock module can be reformulated as a continuity property as follows.

\begin{lemma} \label{lem:induced}
A representation $\pi\colon\NT(X)\to B(H)$ is induced from a representation of $A$ by the Fock module if and only if it factors through $\NTr(X)$ and the representation of $\NTr(X)$ we thus get is strict-strong continuous, that is, it is continuous with respect to the strict topology on $\NTr(X)\subset\LL(\F(X))$ and the strong operator topology on $B(H)$.
\end{lemma}

\bp
The ``only if'' implication is obvious. To prove the ``if'' part, assume we have a representation $\pi$ of $\NTr(X)$ which is strict-strong continuous. Since $\NTr(X)$ is strictly dense in $\LL(\F(X))$ by Lemma~\ref{lem:density}, it extends to a strict-strong continuous representation of $\LL(\F(X))$. The latter representation is completely determined by its restriction to $\K(\F(X))$. But since the right C$^*$-Hilbert $A$-module $\F(X)$ is full, any representation of $\K(\F(X))$ is induced from a representation of $A$ by the module $\F(X)$.
\ep

\begin{remark}
This lemma makes it obvious that a subrepresentation of a representation of $\NT(X)$ induced by the Fock module is itself induced by the Fock module. The proof shows more: the induction by $\F(X)$ defines an equivalence of the category of representations of $A=X_e$ onto a full subcategory of the category of representations of $\NT(X)$.
\end{remark}

\begin{definition}
A positive linear functional $\phi$ on $\NT(X)$ is said to be of \emph{finite type} if the associated GNS-representation is induced by the Fock module. It is said to be of \emph{infinite type} if the associated GNS-representation has no nonzero subrepresentations induced by the Fock module.
\end{definition}

Similarly to Lemma~\ref{lem:induced}, these notions can be formulated as continuity/discontinuity properties. To state the precise result, recall that $\Lambda$ denotes the canonical map $\NT(X)\to\NTr(X)$.

\begin{lemma}\label{lem:finite-infinite}
For any positive linear functional $\phi$ on $\NT(X)$ we have:
\begin{itemize}
\item[(i)] $\phi$ is of finite type if and only if $\phi=\psi\circ\Lambda$ for a strictly continuous positive linear functional $\psi$ on $\NTr(X)$;
\item[(ii)] $\phi$ is of infinite type if and only if there is no nonzero strictly continuous positive linear functional $\psi$ on $\NTr(X)$ such that $\phi\ge\psi\circ\Lambda$.
\end{itemize}
\end{lemma}

\bp (i) Assume $\phi$ is of finite type. Consider the associated GNS-triple $(H_\phi,\pi_\phi,v_\phi)$. Since by assumption $\pi_\phi$ is induced by the Fock module, we have $\pi_\phi=\pi\circ\Lambda$ for a strict-strong continuous representation $\pi$ of $\NTr(X)$. Then the positive linear functional $\psi=(\pi(\cdot)v_\phi,v_\phi)$ on $\NTr(X)$ is strictly continuous and $\phi=\psi\circ\Lambda$.

Conversely, assume $\phi=\psi\circ\Lambda$ for a strictly continuous positive linear functional $\psi$ on $\NTr(X)$. Since $\NTr(X)$ is strictly dense in $\LL(\F(X))$, $\psi$ extends to a strictly continuous positive linear functional on $\LL(\F(X))$, which we continue to denote by $\psi$. Let $(H,\pi,v)$ be the GNS-triple associated with $\psi|_{\K(\F(X))}$, and consider the unique extension of $\pi$ to $\LL(\F(X))$, which we also continue to denote by $\pi$. As we already used in the proof of Lemma~\ref{lem:induced}, any representation of $\K(\F(X))$ is induced by the Fock module. It follows that the representation $\pi$ of $\LL(\F(X))$, hence also $\pi\circ\Lambda$, is induced by the Fock module. At the same time we have $\psi=(\pi(\cdot)v,v)$ on $\LL(\F(X))$, since this equality holds on $\K(\F(X))$ and both sides are strictly continuous. It follows that $(H,\pi\circ\Lambda,v)$ is the GNS-triple associated with $\psi\circ\Lambda=\phi$. Hence $\phi$ is of finite type.

\smallskip

(ii) Assume $\phi$ is of infinite type. If $\phi\ge\psi\circ\Lambda$ for some nonzero strictly continuous positive linear functional $\psi$ on $\NTr(X)$, then $\psi\circ\Lambda$ is of finite type by part (i), so the associated GNS-representation is induced by the Fock module. But this representation is a subrepresentation of the GNS-representation associated with $\phi$. This contradicts the assumption that $\phi$ is of infinite type.

For the converse, assume $\phi$ is not of infinite type. Consider the associated GNS-triple $(H_\phi,\pi_\phi,v_\phi)$. By the assumption there exists a nonzero projection $f\in\pi_\phi(\NT(X))'$ such that the restriction on $\pi_\phi$ to $fH_\phi$ is induced by the Fock module. In particular, this restriction factors through $\NTr(X)$ and defines a strict-strong continuous representation~$\pi$ of~$\NTr(X)$ on~$fH_\phi$. Then $\psi=(\pi(\cdot)fv_\phi,v_\phi)$ is a nonzero strictly continuous positive linear functional on~$\NTr(X)$. As $\pi\circ\Lambda=\pi_\phi(\cdot)|_{fH_\phi}$, we have $\psi\circ\Lambda\le\phi$.
\ep

Note that Lemmas~\ref{lem:induced} and~\ref{lem:finite-infinite}, as opposed to Proposition~\ref{prop:Wold-rep}, are quite formal and do not use any properties of $\NT(X)$ apart from strict density of $\NTr(X)$ in $\LL(\F(X))$.

\begin{proposition} \label{prop:Wold-state}
Assume $(G,P)$ is a quasi-lattice ordered group and $X$ is a compactly aligned product system of C$^*$-correspondences over $P$. Then any positive linear functional $\phi$ on $\NT(X)$ has a unique decomposition $\phi=\phi_f+\phi_\infty$ where $\phi_f$ is of finite type and $\phi_\infty$ is of infinite type.
\end{proposition}

\bp
Let $(H_\phi,\pi_\phi,v_\phi)$ be the GNS-triple associated with $\phi$. Consider the von Neumann algebra $M=\pi_\phi(\NT(X))''$ and let $z\in M$ be the central projection given by Proposition~\ref{prop:Wold-rep}. Then letting $\phi_f=(\pi_\phi(\cdot)zv_\phi,v_\phi)$ and $\phi_\infty=(\pi_\phi(\cdot)(1-z)v_\phi,v_\phi)$ we get the required decomposition of~$\phi$.

In order to prove that the decomposition is unique it suffices to establish the following properties of the functionals $\phi_f$ and $\phi_\infty$ defined above: if $\psi\le\phi$ is of finite type, then $\psi\le\phi_f$, and if $\psi\le\phi$ is of infinite type, then $\psi\le\phi_\infty$.

So assume we have a positive linear functional $\psi$ such that $\psi\le\phi$. Let $x\in M'$ be the unique element such that $0\le x\le 1$ and $\psi=(\pi_\phi(\cdot)xv_\phi,v_\phi)$, and let $f\in M'$ be the support projection of $x$. Then as the GNS-triple associated with $\psi$ we can take $(fH_\phi,\pi_\phi|_{fH_\phi},x^{1/2}v_\phi)$.

Now, if $\psi$ is of finite type, then we must have $f\le z$. Hence $x\le z$ and therefore $\psi\le\phi_f$. On the other hand, if $\psi$ is of infinite type, then $zf=0$, since the representation $\pi_\phi|_{zfH_\phi}$ is induced by the Fock module and is contained in the GNS-representation of $\psi$, so it must be zero. Hence $x\le 1-z$ and therefore $\psi\le\phi_\infty$.
\ep

The construction of $\phi_f$ and $\phi_\infty$ implies the following.

\begin{corollary}\label{cor:Wold-KMS}
If in the setting of Proposition~\ref{prop:Wold-state} the positive linear functional $\phi$ satisfies the $\sigma$-KMS$_\beta$ condition for some time evolution $\sigma$ on $\NT(X)$ and $\beta\in\R$, then $\phi_f$ and $\phi_\infty$ also satisfy the $\sigma$-KMS$_\beta$ condition.
\end{corollary}

\bp
Using the same notation as in the proof of Proposition~\ref{prop:Wold-state}, consider the normal positive linear functional $\bar\phi$ on $M$ defined by the vector $v_\phi$. Then $\bar\phi$ is faithful and its modular group~$\sigma^{\bar\phi}$ satisfies $\pi_\phi\circ\sigma_{-\beta t}=\sigma^{\bar\phi}_t\circ\pi_\phi$. Since $\bar\phi$ satisfies the $\sigma^{\bar\phi}$-KMS$_{-1}$ condition, it follows that the functionals $\phi_f=\bar\phi(\pi_\phi(\cdot)z)$ and $\phi_\infty=\bar\phi(\pi_\phi(\cdot)(1-z))$ satisfy the $\sigma$-KMS$_\beta$ condition.
\ep

Our next goal is to understand what the decomposition $\phi=\phi_f+\phi_\infty$ of a $\sigma$-KMS$_\beta$-state, where $\sigma$ is given by \eqref{eq:dynamics1}, means at the level of the trace $\tau=\phi|_A$.

Recall that we denote by $F_p$ the operator mapping a trace $\tau$ on $A$ into $\Tr^p_\tau|_A$.

\begin{theorem}\label{thm:Wold}
Assume $(G,P)$ is a quasi-lattice ordered group and $X$ is a compactly aligned product system of C$^*$-correspondences over $P$, with $X_e=A$. Consider the dynamics \eqref{eq:dynamics1} on~$\NT(X)$ defined by a homomorphism $N\colon P\to (0,+\infty)$, and fix $\beta\in\R$. Assume $\phi$ is a $\sigma$-KMS$_\beta$-state on $\NT(X)$, consider the finite part $\phi_f$ of $\phi$, and let $\tilde\phi_f$ be the unique strictly continuous positive linear functional on $\LL(\F(X))$ such that $\phi_f=\tilde\phi_f\circ\Lambda$. Define a positive trace~$\tau_0$ on $M(A)\cong\LL(X_e)$ by
$$
\tau_0(a)=\tilde\phi_f(aQ_e),
$$
where $Q_e\in\LL(\F(X))$ is the projection onto $X_e\subset\F(X)$. Then
\begin{itemize}
\item[(i)] for every $a\in M(A)$,
\begin{equation}\label{eq:tau0}
\tau_0(a)=\lim_{J}\Big(\tau(a)+\sum_{\emptyset\neq K\subset J} (-1)^{\vert K\vert} N(q_K)^{-\beta}\Tr_{\tau}^{q_K}(a)\Big),
\end{equation}
where the limit is over the net of finite subsets $J$ (partially ordered by inclusion) of any fixed generating set $S\subset P\setminus\{e\}$ and we consider the canonical extension of $\tau$ to $M(A)$;
\item[(ii)] the trace $\tau_f=\phi_f|_A$ is given by
$$
\tau_f=\sum_{p\in P}N(p)^{-\beta}F_p\tau_0;
$$
\item[(iii)] the functional $\phi_f$ is given by
\begin{equation}\label{eq:Gibbs}
\phi_f=\Tr^{\F(X)}_{\tau_0}(\Lambda(\cdot)D_N^{-\beta}),
\end{equation}
where $D_N$ is the operator on $\F(X)$ defined by~\eqref{eq:DN}.
\end{itemize}
\end{theorem}

\bp (i) Using the notation from the proof of Proposition~\ref{prop:Wold-state}, we have $\phi_f=(\pi_\phi(\cdot)zv_\phi,v_\phi)$. The representation $\pi_\phi|_{zH_\phi}$ factors through $\LL(\F(X))$, and by the construction of the projection~$z$ in the proof of Proposition~\ref{prop:Wold-rep}, the strictly continuous representation of $\LL(\F(X))$ we thus get maps $Q_e$ into the projection $Q$ onto the space~\eqref{eq:vacuum} of vacuum vectors. It follows that
$$
\tau_0(a)=\tilde\phi_f(aQ_e)=(\pi_\phi(a)Qv_\phi,v_\phi)\ \ \text{for}\ \ a\in A.
$$

Next, recall from the proof of Theorem~\ref{thm:subinvariance} that we also have a representation $L\colon B_P\to B(H_\phi)$ such that $f_p=L(\chi_{pP})$ is the projection onto the closed linear span of the images of $\pi_\phi(i_X(\xi))$ for all $\xi\in X_p$. Hence
$1-Q=\vee_{p\ne e}f_p$, and therefore $Q$ is the strong operator limit of the projections
$
\prod_{p\in J}(1-f_p)
$
over the net of finite subsets $J\subset P\setminus\{e\}$. Since $f_p\ge f_q$ for $p\le q$, it suffices to consider subsets~$J$ of any given generating set $S\subset P\setminus\{e\}$. Since
$$
\prod_{p\in J}(1-f_p)=1+\sum_{\emptyset\neq K\subset J} (-1)^{\vert K\vert}f_{q_K},
$$
we therefore get, for all $a\in A$, that
$$
\tau_0(a)=\lim_{J}\Big(\tau(a)+\sum_{\emptyset\neq K\subset J} (-1)^{\vert K\vert} (\pi_\phi(a)f_{q_K}v_\phi,v_\phi)\Big).
$$
In view of equality~\eqref{eq:mua} this is exactly~\eqref{eq:tau0}.

This proves~\eqref{eq:tau0} for the elements of $A$. In order to see that the same is true for $M(A)$ we just have to observe that all the expressions we considered, like $(\pi_\phi(a)Qv_\phi,v_\phi)$, $(\pi_\phi(a)f_{q_K}v_\phi,v_\phi)$ and $\Tr^{q_K}_\tau(a)$, have obvious extensions to $M(A)$ by strict continuity, so the same proof works for~$M(A)$.

\smallskip

(ii),(iii) Consider the generalized Gibbs state $\psi=\Tr^{\F(X)}_{\tau_0}(\cdot D_N^{-\beta})$ on $\LL(\F(X))$ discussed in Section~\ref{sec:restriction}. Then the restrictions of $\psi$ and $\tilde\phi_f$ to $\K(\F(X))$ are both KMS$_\beta$ with respect to the dynamics $\Ad D_N^{it}$, and by the construction of $\psi$ they coincide on $Q_e\K(\F(X))Q_e=AQ_e$. Since~$Q_e$ is a full projection in the multiplier algebra of $\K(\F(X))$, it follows that $\psi=\tilde\phi_f$ on $\K(\F(X))$ by the discussion at the end of Section~\ref{ssec:induction}. By strict continuity we then have $\psi=\tilde\phi_f$ on $\LL(\F(X))$. This proves both (ii) and (iii).
\ep

\begin{remark}\label{rem:norm-limit}
Since the trace $\tau_0$ is the weak$^*$ limit of a decreasing net of positive traces on the unital C$^*$-algebra $M(A)$, it is actually the norm limit of these traces. The above proof makes this particularly transparent, as it hinges on the norm convergence $\prod_{p\in J}(1-f_p)v_\phi\xrightarrow[J]{} Qv_\phi$.
\end{remark}

\begin{corollary}\label{cor:finite-infinite}
In the setting of Theorem~\ref{thm:Wold}, let $\phi$ be a $\sigma$-KMS$_\beta$-state on $\NT(X)$ and $\tau=\phi|_A$. Then
\begin{itemize}
\item[(i)] $\phi$ is of infinite type if and only if the limit in~\eqref{eq:tau0} is zero for all $a\in A$, or equivalently, for $a=1$;
\item[(ii)]  $\phi$ is of finite type if and only if $\tau=\sum_{p\in P}N(p)^{-\beta}F_p\tau_0$ for some positive trace $\tau_0$ on $A$; in this case $\tau_0$ and $\phi$ are uniquely determined by $\tau$: $\tau_0$ is given by~\eqref{eq:tau0} and $\phi$ is given by~\eqref{eq:Gibbs} (with $\phi=\phi_f$).
\end{itemize}
\end{corollary}

\bp
(i) The state $\phi$ is of infinite type if and only if $\phi_f=0$, which by part (iii) of the theorem is equivalent to $\tau_0=0$ on $A$. This proves the ``if and only if'' part of the statement. The last equivalence part follows, since even if $A$ is nonunital, by the strict continuity of $\tilde\phi_f$, the functional $\tau_0=\tilde\phi_f(\cdot\,Q_e)$ is zero on $A$ if and only if $\tilde\phi_f(Q_e)=0$.

\smallskip

(ii) The ``only if'' statement follows from part (ii) of the theorem. To prove the converse, assume $\tau=\sum_{p\in P}N(p)^{-\beta}F_p\tau_0$ for some $\tau_0$. Consider the state  $\psi=\Tr^{\F(X)}_{\tau_0}(\cdot D_N^{-\beta})$ on $\LL(\F(X))$. Then, on the one hand, we have $\psi(aQ_e)=\tau_0(a)$ for $a\in A$ by construction. On the other hand, $\psi\circ\Lambda$ is a $\sigma$-KMS$_\beta$-state on $\NT(X)$ of finite type, so we can compute $\psi(aQ_e)$ using part (i) of the theorem. The conclusion is that $\tau_0$ must be given by~\eqref{eq:tau0}. But then by part (iii) of the theorem we get $\phi_f=\psi\circ\Lambda$. Since $\psi\circ\Lambda$ is a state, it follows that $\phi=\phi_f$, so $\phi$ is of finite type and it is given by~\eqref{eq:Gibbs}.
\ep

We can now give the following definition.

\begin{definition}
With $X$, $N$ and $\beta$ fixed, we say that a positive trace $\tau$ on $A$ satisfying~\eqref{eq:subinvariance} is of \emph{finite type}, if $\tau=\sum_{p\in P}N(p)^{-\beta}F_p\tau_0$ for some positive trace $\tau_0$ on $A$, which is then necessarily given by~\eqref{eq:tau0}. We say that $\tau$ is of \emph{infinite type}, if the limit in~\eqref{eq:tau0} is zero for all $a\in A$, equivalently,  for $a=1$ (with $1\in M(A)$ and $\tau(1)=\|\tau\|$ if $A$ is nonunital).
\end{definition}

Note that even with $X$ and $N$ fixed, the same trace can be of finite type for one $\beta$ and of infinite type for another.

\begin{corollary}\label{cor:Wold-trace}
In the setting of Theorem~\ref{thm:Wold}, any tracial state $\tau$ on $A$ satisfying~\eqref{eq:subinvariance} has a unique decomposition $\tau=\tau_f+\tau_\infty$ where $\tau_f$ is of finite type and $\tau_\infty$ is of infinite type. Explicitly, we first define a positive trace $\tau_0$ on~$A$ by~\eqref{eq:tau0}, then $\tau_f=\sum_{p\in P}N(p)^{-\beta}F_p\tau_0$ and $\tau_\infty=\tau-\tau_f$.
\end{corollary}

\bp
By considering the unique gauge-invariant $\sigma$-KMS$_\beta$-state $\phi$ on $\NT(X)$ extending $\tau$, this follows from the existence and uniqueness of a decomposition of $\phi$ into finite and infinite type parts and Corollary~\ref{cor:finite-infinite}.
\ep

The following is contained in part (ii) of Corollary~\ref{cor:finite-infinite}, but it is worth stating it explicitly.

\begin{corollary}\label{cor:finite-class}
In the setting of Theorem~\ref{thm:Wold}, we have an affine bijection between the $\sigma$-KMS$_\beta$-states $\phi$ on $\NT(X)$ of finite type and the positive traces $\tau_0$ on $A$ such that $$\sum_{p\in P}N(p)^{-\beta}\Tr^p_{\tau_0}(1)=1.$$ Namely, going from $\phi$ to $\tau_0$ is by~\eqref{eq:tau0}, and back by~\eqref{eq:Gibbs} (with $\phi=\phi_f$). Furthermore, if $\phi$ is such a state and $\psi$ is another $\sigma$-KMS$_\beta$-state (a priori not necessarily of finite type) such that $\psi|_A=\phi|_A$, then $\psi=\phi$.
\end{corollary}

\bigskip

\section{Traces of finite type and critical temperature}\label{sec:critical-temp}

In practice, checking whether a tracial state is of finite type by definition can be a difficult problem. The following result provides a powerful sufficient condition.

\begin{theorem}\label{thm:finite-characterization}
Assume $(G,P)$ is a quasi-lattice ordered group and $X$ is a compactly aligned product system of C$^*$-correspondences over $P$, with $X_e=A$. Consider the dynamics \eqref{eq:dynamics1} on~$\NT(X)$ defined by a homomorphism $N\colon P\to (0,+\infty)$. Fix $\beta\in\R$ and assume that~$\tau$ is a tracial state on $A$ satisfying~\eqref{eq:subinvariance} such that
$$
\sum_{p\in P}N(p)^{-\beta}\Tr_\tau^p(1)<\infty.
$$
Then $\tau$ is of finite type.
\end{theorem}

For $P=\Z^n_+$ this theorem is an immediate consequence of the finiteness criterion~\eqref{eq:criterion}. But this case is quite special. For general $P$ the proof is based on the following lemma of independent interest.

\begin{lemma}\label{lem:atoms}
Assume $(G,P)$ is a quasi-lattice ordered group and $\mu$ is a finitely additive probability measure on $(P,\B_P)$ such that
$$
\sum_{p\in P}\mu(pP)<\infty.
$$
For every $p\in P$, define $w_p=\inf\{\mu(\Omega)\mid p\in \Omega\in\B_P\}$. Then, for every $\Omega\in\B_P$, we have
$$
\mu(\Omega)=\sum_{p\in\Omega}w_p.
$$
In particular, $\mu$ extends to a $\sigma$-additive probability measure defined on all subsets of $P$.
\end{lemma}


\bp The key part is to establish that
\begin{equation}\label{eq:sum-one}
\sum_{p\in P}w_p\ge 1.
\end{equation}

In order to prove this inequality, take $\eps>0$ and choose a finite set $F\subset P$ such that $\sum_{p\notin F}\mu(pP)<\eps$. We claim that if for every $p\in F$ we take a set $\Omega_p\in \B_P$ containing $p$, then
\begin{equation}\label{eq:sum-one2}
\sum_{p\in F}\mu(\Omega_p)>1-\eps.
\end{equation}
This implies that $\sum_{p\in F}w_p\ge 1-\eps$, and since $\eps$ could be taken arbitrarily small, proves~\eqref{eq:sum-one}.

To prove~\eqref{eq:sum-one2}, consider the set $\cap_{p\in F}(P\setminus\Omega_p)$. By Lemma~\ref{lem:BP-decomposition}(i) it can be written as the disjoint union of finitely many sets of the form $qP$ or $\cap_{r\in J}q(P\setminus rP)$, with $q\in P$ and $J\subset P\setminus\{e\}$. It is clear that for every such set appearing in the disjoint union the element $q$ does not lie in any of the sets $\Omega_p$, hence $q\notin F$. Therefore if $E$ is the set of such $q$'s, then $E$ is finite, $E\cap F=\emptyset$ and $P$ is covered by the sets $\Omega_p$ ($p\in F$) and $qP$ ($q\in E$). It follows that
$$
1\le \sum_{p\in F}\mu(\Omega_p)+\sum_{q\in E}\mu(qP)<\sum_{p\in F}\mu(\Omega_p)+\eps.
$$
This proves~\eqref{eq:sum-one2} and hence~\eqref{eq:sum-one}.

\smallskip

Next, we claim that
\begin{equation}\label{eq:sum-less}
\sum_{p\in\Omega}w_p\le\mu(\Omega)\ \ \text{for any}\ \ \Omega\in\B_P.
\end{equation}
Indeed, take any finite subset $F\subset\Omega$ and disjoint sets $\Omega_p$ ($p\in F$) such that $p\in\Omega_p\subset\Omega$ for all $p\in F$. This is possible as $\B_P$ separates points. Then
$$
\sum_{p\in F}w_p\le\sum_{p\in F}\mu(\Omega_p)\le\mu(\Omega),
$$
which is what we need.

\smallskip

Now, from~\eqref{eq:sum-one} and~\eqref{eq:sum-less}, for every $\Omega\in\B_P$, we get
$$
1\le\sum_{p\in P}w_p=\sum_{p\in\Omega}w_p+\sum_{p\in P\setminus\Omega}w_p\le\mu(\Omega)+\mu(P\setminus \Omega)=1,
$$
hence $\sum_{p\in\Omega}w_p=\mu(\Omega)$.
\ep

\bp[Proof of Theorem~\ref{thm:finite-characterization}]
Let $\tau_0$ be the trace on $A$ given by~\eqref{eq:tau0}. We have to show that
\begin{equation}\label{eq:finite-tau}
\tau(a)=\sum_{p\in P}N(p)^{-\beta}\Tr^p_{\tau_0}(a)
\end{equation}
for all $a\in A$.

Fix $a\in A_+$ and consider the measure $\mu_a$ on $(P,\B_P)$ defined by~\eqref{eq:mua1}. Since by definition we have $\mu_a(pP)=N(p)^{-\beta}\Tr^p_{\tau}(a)$, we can apply Lemma~\ref{lem:atoms} and conclude that, with $w_p=\inf\{\mu_a(\Omega)\mid p\in \Omega\in\B_P\}$, we have $\sum_{p\in P}w_p=\tau(a)$. We claim that this is exactly identity~\eqref{eq:finite-tau}, that is,
\begin{equation}\label{eq:finite-tau2}
w_p=N(p)^{-\beta}\Tr^p_{\tau_0}(a)\ \ \text{for all}\ \ p\in P.
\end{equation}

In order to prove~\eqref{eq:finite-tau2}, for every finite subset $J\subset P\setminus\{e\}$ consider the positive trace $\tau_J$ on~$A$ defined by
$$
\tau_J=\tau+\sum_{\emptyset\neq K\subset J} (-1)^{\vert K\vert} N(q_K)^{-\beta}F_{q_K}\tau.
$$
With the sets $J$ partially ordered by inclusion, the traces $\tau_J$ converge in the weak$^*$ topology (and even in norm) to~$\tau_0$. By~\eqref{eq:mu} we also have
$$
\mu_a\Big(\bigcap_{q\in J}p(P\setminus qP)\Big)=N(p)^{-\beta}(F_p\tau_J)(a).
$$
On the other hand, by Lemma~\ref{lem:BP-decomposition}(ii) we know that every set in $\B_P$ containing $p$ contains a set of the form $\cap_{q\in J}\,p(P\setminus qP)$, so in computing $w_p$ it suffices to take the limit over such sets. In other words, we have
$$
w_p=\lim_JN(p)^{-\beta}(F_p\tau_J)(a),
$$
Therefore in order to establish~\eqref{eq:finite-tau2} we have to show that $F_p\tau_J\to F_p\tau_0$ in the weak$^*$ topology for all $p\in P$. But this is true by Lemma~\ref{lem:trace-continuity}(ii), since $\tau_J\le\tau$ for all finite $J\subset P\setminus\{e\}$.
\ep

Combining this theorem with Corollary~\ref{cor:finite-class} we get the following.

\begin{corollary}\label{cor:finite-class2}
In the setting of Theorem~\ref{thm:finite-characterization}, assume that for some $\beta\in\R$ we have
\begin{equation}\label{eq:all-finite}
\sum_{p\in P}N(p)^{-\beta}\Tr^p_{\tau}(1)<\infty\ \ \text{for all tracial states}\ \tau\ \text{on} \ A.
\end{equation}
Then all $\sigma$-KMS$_\beta$-states on $\NT(X)$ are of finite type, so we get an affine bijection between the $\sigma$-KMS$_\beta$-states on $\NT(X)$ and the positive traces $\tau_0$ on $A$ such that $\sum_{p\in P}N(p)^{-\beta}\Tr^p_{\tau_0}(1)=1$.
\end{corollary}

Note that since by rescaling the traces $\tau_0$ as above we get all tracial states on $A$, we can also say that we have a bijective correspondence between the $\sigma$-KMS$_\beta$-states on $\NT(X)$ and the tracial states on $A$. This correspondence, however, is not affine in general.

\begin{definition}\label{def:critical}
Assuming that $N(p)\ge1$ for all $p\in P$, we denote by $\beta_c$ the infimum of $\beta\in\R$ such that condition~\eqref{eq:all-finite} holds, and call it the \emph{critical inverse temperature}.
\end{definition}

The critical inverse temperature equals the maximum of the abscissas of convergence of the series $\sum_{p\in P}N(p)^{-\beta}\Tr^p_{\tau}(1)$, since if we have a sequence of tracial states~$\tau_n$ and an increasing sequence $(\beta_n)_n$ of inverse temperatures converging to $\beta_c$ such that $\sum_{p\in P}N(p)^{-\beta_n}\Tr^p_{\tau_n}(1)=\infty$ for all $n$, then for the tracial state $\tau=\sum^\infty_{n=1}2^{-n}\tau_n$ we have $\sum_{p\in P}N(p)^{-\beta}\Tr^p_{\tau}(1)=\infty$ for all $\beta<\beta_c$.

In this terminology, Corollary~\ref{cor:finite-class2} contains the following statement.

\begin{corollary}\label{cor:finite-class3}
In the setting of Theorem~\ref{thm:finite-characterization}, assume that $N(p)\ge 1$ for all $p\in P$ and that the critical inverse temperature $\beta_c$ satisfies $\beta_c<+\infty$.
Then, for every $\beta>\beta_c$, all $\sigma$-KMS$_\beta$-states on $\NT(X)$ are of finite type, so we get an affine bijection between the $\sigma$-KMS$_\beta$-states on $\NT(X)$ and the positive traces $\tau_0$ on $A$ such that $\sum_{p\in P}N(p)^{-\beta}\Tr^p_{\tau_0}(1)=1$.
\end{corollary}

Let us now discuss assumption~\eqref{eq:all-finite} in more detail.  Recall from Section~\ref{sec:restriction} that we denote by $T(A)\subset A^*$ the linear span of tracial states. If the trace $F_p\tau$ is finite for any tracial state~$\tau$, then~$F_p$ extends by linearity to an operator on~$T(A)$. Furthermore, since the operator~$F_p$ is positive, standard arguments show that it must then be bounded.

The only realistic general assumption that guarantees boundedness of $F_p$ seems to be that each right C$^*$-Hilbert $A$-module $X_p$ is isomorphic to a direct summand of $A^{m(p)}$ for some $m(p)\in\N$. Then $\|F_p\|\le m(p)$.

\begin{remark}\label{rem:Parseval}
If $A$ is unital, a right C$^*$-Hilbert $A$-module $Y$ is isomorphic to a direct summand of $A^m$ if and only if $Y$ is finitely generated, and the smallest such $m$ is the minimal number of generators, equivalently, the minimal cardinality of a Parseval frame in $Y$. Briefly, this can be seen as follows. Assume $Y$ is generated by $\xi_1,\dots,\xi_m$, so that we have a surjective $A$-module map $S\colon A^m\to Y$, $Se_i=\xi_i$, where $(e_i)^m_{i=1}$ is the standard frame in $A^m$. Then $S$ is adjointable, so by \cite[Theorem~3.2]{lan} the submodule $\ker S\subset A^m$ is complemented. It follows that $Y$ is isomorphic, hence isometrically isomorphic, to a complemented submodule of~$A^m$, and the projection of~$(e_i)^m_{i=1}$ onto this submodule is a Parseval frame of cardinality $m$.
\end{remark}

Therefore, although formally Corollary~\ref{cor:finite-class2} does not require anything special about the product system $X$, in practice it mostly applies to the ``finite rank'' case. For such systems the following result shows that, under additional assumptions, we see a change in the behavior of KMS-states at $\beta_c$, justifying the name ``critical inverse temperature''. This result is partly motivated by~\cite[Theorem~2.5]{PWY} and generalizes~\cite[Proposition~4.5(1)]{BLRS}, and its proof uses ideas from both papers.

\begin{proposition}\label{prop:phase-transition}
Assume $(G,P)$ is a quasi-lattice ordered group and $X$ is a compactly aligned product system of C$^*$-correspondences over $P$, with $X_e=A$. Consider the dynamics \eqref{eq:dynamics1} on~$\NT(X)$ defined by a homomorphism $N\colon P\to (0,+\infty)$. Assume in addition that
\begin{itemize}
\item[(a)] the monoid $P$ has a finite generating set $S\subset P\setminus\{e\}$;
\item[(b)] $N(p)>1$ for all $p\in P\setminus\{e\}$;
\item[(c)] $A$ is unital;
\item[(d)] the right C$^*$-Hilbert $A$-modules $X_p$ are finitely generated.
\end{itemize}
Then
\begin{itemize}
\item[(i)] the critical inverse temperature $\beta_c$ is the smallest number, possibly $-\infty$, such that the operator
$$
1+\sum_{\emptyset\neq K\subset S} (-1)^{\vert K\vert} N(q_K)^{-\beta}F_{q_K}
$$
on $T(A)$ is invertible for all $\beta>\beta_c$;
\item[(ii)] if $\beta_c\ne-\infty$, then there exists a tracial state $\tau$ on $A$ which satisfies~\eqref{eq:subinvariance} for $\beta=\beta_c$ and is of infinite type; in particular, $\sum_{p\in P}N(p)^{-\beta_c}\Tr_\tau^p(1)=\infty$.
\end{itemize}
\end{proposition}

\bp
Consider the operators $S_\beta$ and $T_\beta$ on $T(A)$ defined by
$$
S_\beta=\sum_{p\in P}N(p)^{-\beta}F_p, \ \ T_\beta=1+\sum_{\emptyset\neq K\subset S} (-1)^{\vert K\vert} N(q_K)^{-\beta}F_{q_K}.
$$
The operator $T_\beta$ is well-defined for all real $\beta$, as well as for complex ones, in which case we write $T_z$ to reserve $\beta$ for the real numbers. The operator $S_\beta$ is well-defined for $\beta>\beta_c$, if the convergence of the series above is understood pointwise. Note that the series does converge absolutely for $\beta$ large enough, since $P$ is assumed to be finitely generated and so an upper bound for $\sum_{p\in P}N(p)^{-\beta}\|F_p\|$ can be obtained by looking at the free monoid $\FF^+_S$ with generators $g_s$, $s\in S$, and the series
$$
\sum_{g=g_{s_1}\dots g_{s_n}\in\FF_S^+}N(s_1)^{-\beta}\dots N(s_n)^{-\beta}\|F_{s_1}\|\dots \|F_{s_n}\|.
$$

By Corollary~\ref{cor:finite-infinite}(ii), for $\beta>\beta_c$, we have $T_\beta S_\beta\tau_0=\tau_0$ for any positive trace $\tau_0$ on $A$ (this is also easy to see directly), hence for all $\tau_0\in T(A)$. Since by assumption (b) we have $N(q_K)>1$ for all $\emptyset\neq K\subset S$ with $q_K<\infty$, for large~$\beta$ the operator $T_\beta$ is close to $1$ and hence is invertible. It follows that $S_\beta= T_\beta^{-1}$ for all $\beta$ large enough.

\smallskip

Let now $\beta_1$ be the smallest number, possibly $-\infty$, such that $T_\beta$ is invertible for all $\beta>\beta_1$. Let $\beta_2\ge\beta_1$ be the smallest number such that $S_\beta$ is well-defined and $S_\beta=T_\beta^{-1}$ for all $\beta>\beta_2$. We claim that $\beta_2=\beta_1$.

Assume this is not true. Then there is $\eps>0$ such that $\beta_2-\eps>\beta_1$ and $T_z$ is invertible for~$z$ in the disc $D=\{z:|z-\beta_2|<\eps\}$. Since $T_z$ is analytic in $z$, the operator-valued map $D\ni z\mapsto T_z^{-1}$ is analytic. Take a positive trace~$\tau$ on~$A$ and $a\in A_+$. Then the Dirichlet series $\sum_{p\in P}N(p)^{-\beta}(F_p\tau)(a)$ converges for all $\beta>\beta_2$ and extends by analyticity to the function $D\ni z\mapsto  (T_z^{-1}\tau)(a)$. By the proof of Landau's theorem, see~\cite[Theorem~10]{HR}, it follows that the Dirichlet series converges to this function for all $\beta\in D\cap \R$. Therefore the series $\sum_p N(p)^{-\beta}F_p\tau$ converges in the weak$^*$ topology, hence in norm, to $T_\beta^{-1}\tau$ for all $\beta>\beta_2-\eps$. Since this is true for all $\tau$, this contradicts the definition of $\beta_2$. Thus $\beta_2=\beta_1$. In particular, $\beta_c\le\beta_1$.

\smallskip

If $\beta_1=-\infty$, there is nothing left to prove. So assume $\beta_1$ is finite. Then $T_{\beta_1}$ is not invertible, as otherwise $T_\beta$ would also be invertible for $\beta$ in a neighbourhood of $\beta_1$, which is not possible. Take a sequence of real numbers $t_n$ such that $t_n\downarrow\beta_1$. We claim that $\|T_{t_n}^{-1}\|\to\infty$.

Assume this is not true. Then by passing to a subsequence we may assume that the sequence $(T_{t_n}^{-1})_n$ is bounded. But then the identity
$$
T_{t_n}^{-1}-T_{t_m}^{-1}=T_{t_n}^{-1}(T_{t_m}-T_{t_n})T_{t_m}^{-1}
$$
shows that $(T_{t_n}^{-1})_n$ is a Cauchy sequence converging to an operator $S$, and by passing to the limit in the identities $T_{t_n}T_{t_n}^{-1}=T_{t_n}^{-1}T_{t_n}=1$ we get $T_{\beta_1}S=ST_{\beta_1}=1$, which is a contradiction.

By the uniform boundedness principle it follows then that there exists $\tau_0\in T(A)$ such that the sequence $(\|T_{t_n}^{-1}\tau_0\|)_n$ is unbounded. Since $T(A)$ is spanned by positive traces, we may assume that $\tau_0$ is a tracial state. By passing to a subsequence we may also assume that $\|T_{t_n}^{-1}\tau_0\|\to\infty$.

Consider the tracial states
$$
\tau_n=\frac{T_{t_n}^{-1}\tau_0}{\|T_{t_n}^{-1}\tau_0\|}.
$$
Let $\tau$ be a weak$^*$ cluster point of $(\tau_n)_n$. Since $A$ is unital, $\tau$ is a tracial state. Next, since $T_{t_n}^{-1}=S_{t_n}=\sum_{p\in P}N(p)^{-t_n}F_p$, the trace $\tau_n$ satisfies condition~\eqref{eq:subinvariance} for $\beta=t_n$. Since the right C$^*$-Hilbert $A$-modules are finitely generated and therefore admit finite Parseval frames, the operators $F_p$ are obviously weakly$^*$ continuous. From this we may conclude that $\tau$ satisfies condition~\eqref{eq:subinvariance} for $\beta=\beta_1$.

Since $T_{t_n}\tau_n=\|T_{t_n}^{-1}\tau_0\|^{-1}\tau_0$,
by passing to the limit we see also that $T_{\beta_1}\tau=0$, that is, the trace $\tau$ is of infinite type. By Theorem~\ref{thm:finite-characterization} it follows that $\sum_{p\in P}N(p)^{-\beta_1}\Tr_\tau^p(1)=\infty$. Since we already know that $\beta_c\le\beta_1$, this shows that $\beta_c=\beta_1$ and completes the proof of the proposition.
\ep

From the proof of the proposition we get the following result complementing Corollary~\ref{cor:finite-class3}.

\begin{corollary}\label{cor:finite-class4}
In the setting of Proposition~\ref{prop:phase-transition}, for every $\beta>\beta_c$, the map $\phi\mapsto \phi|_A$ defines a one-to-one correspondence between the $\sigma$-KMS$_\beta$-states on $\NT(X)$ and the tracial states~$\tau$ on~$A$ such that
\begin{equation}\label{eq:subinvariance2.5}
\tau+\sum_{\emptyset\neq K\subset S} (-1)^{\vert K\vert} N(q_K)^{-\beta}F_{q_K}\tau\ge0.
\end{equation}
\end{corollary}

In particular, we see that for $\beta>\beta_c$ condition~\eqref{eq:subinvariance} reduces to one inequality corresponding to $J=S$.

\bp
In the notation from the proof of Proposition~\ref{prop:phase-transition}, Corollary~\ref{cor:finite-class3} states that the map $\phi\mapsto\phi|_A$ defines a one-to-one correspondence between the $\sigma$-KMS$_\beta$-states on $\NT(X)$ and the tracial states on $A$ of the form $S_\beta\tau_0$, with $\tau_0\ge0$.
On the other hand, condition~\eqref{eq:subinvariance2.5} means that $T_\beta\tau\ge0$. Therefore we have to prove that a tracial state $\tau$ equals $S_\beta\tau_0$ for a positive trace~$\tau_0$ if and only if the trace $T_\beta\tau$ is positive. But this is clear, since $S_\beta=T_\beta^{-1}$.
\ep

\begin{example}\label{ex:free-abelian-critical}
In the setting of Proposition~\ref{prop:phase-transition} consider $P=\Z^n_+$. In this case the operator from part (i) of the proposition is
$$
\prod^n_{i=1}(1-N(e_i)^{-\beta}F_i),
$$
where, as before, $F_i=F_{e_i}$. Therefore $\beta_c$ is the largest number such that $N(e_i)^{\beta_c}$ belongs to the spectrum of $F_i$ for some $i$. By~\cite[Theorem~2.5]{PWY} we know that the largest positive eigenvalue of~$F_i$ coincides with the spectral radius of $F_i$. (To be more precise, the result in~\cite{PWY} is formulated for a more restricted class of operators, but its proof works for any weakly$^*$ continuous positive operator on $T(A)$.) We conclude that
\begin{equation}\label{eq:critical-beta}
\beta_c=\max_{1\le i\le n}\frac{\log r(F_i)}{\log N(e_i)},
\end{equation}
where $r(F_i)$ denotes the spectral radius of $F_i$.

This equality shows that whenever $r(F_i)>1$ for all $i$, the most natural dynamics on $\NT(X)$ is given by $N(e_i)=r(F_i)$. This type of dynamics was first singled out in~\cite{Y} and is now called the \emph{preferred dynamics}~\cite{aHLRS}.
\end{example}

Returning to the more general setting of Corollary~\ref{cor:finite-class2}, observe that the uniform boundedness principle implies that, under the assumptions of that corollary, the partial sums $\sum_{p\in F}N(p)^{-\beta}F_p$ are uniformly bounded when $F$ runs over finite subsets of $P$. We do not know whether this implies that the series
\begin{equation}\label{eq:abszeta}
\sum_{p\in P}N(p)^{-\beta}\|F_p\|
\end{equation}
is convergent. Similarly, we do not know whether the critical inverse temperature (defined if $N(p)\ge1$ for all $p$) coincides with the abscissa of convergence of the series~\eqref{eq:abszeta} in general. Equality~\eqref{eq:critical-beta} implies that these two numbers do coincide for the monoids $P=\Z^n_+$ under the assumptions of Proposition~\ref{prop:phase-transition}.

\begin{remark}\label{rem:finite-intervals}
Series~\eqref{eq:abszeta} can be convergent for nontrivial reasons only under some additional assumptions on $P$. For instance, if we assume that $F_p\ne0$ for all $p$ (and this is the case, e.g., if every right C$^*$-Hilbert module $X_p$ is full and $T(A)\ne0$), then finiteness of \eqref{eq:abszeta} for some $\beta\in\R$ implies that $P$ has finite initial intervals $\{q\mid q\le p\}$. Indeed, otherwise there exist $p$, $q_n$ and $r_n$, such that $p=q_nr_n$ and the elements $q_n$ are all different. Then
$$
N(p)^{-\beta}\|F_p\|\le N(q_n)^{-\beta}\|F_{q_n}\|N(r_n)^{-\beta}\|F_{r_n}\|,
$$
and since both $N(q_n)^{-\beta}\|F_{q_n}\|$ and $N(r_n)^{-\beta}\|F_{r_n}\|$ must converge to zero, we get a contradiction.
\end{remark}

\bigskip

\section{Gauge-invariance}\label{sec:gauge-invariance}

One of the obvious shortcomings of Theorem~\ref{thm:general} is that it deals only with gauge-invariant KMS-states. By the results of the previous section we know, however, that for large $\beta$ all KMS$_\beta$-states can happen to be of finite type and, in particular, gauge-invariant. The following generalization of~\cite[Proposition~2.2]{BLRS} can sometimes be used to establish gauge-invariance of KMS-states of infinite type.

\begin{proposition}\label{prop:automatic-gauge}
Assume $(G,P)$ is a quasi-lattice ordered group and $X$ is a compactly aligned product system of C$^*$-correspondences over $P$, with $X_e=A$. Consider the dynamics \eqref{eq:dynamics1} on~$\NT(X)$ defined by a homomorphism $N\colon P\to (0,+\infty)$. Given a $\sigma$-KMS$_\beta$-state $\phi$ for some $\beta\in\R$, consider the trace $\tau=\phi|_A$ and assume that it satisfies the following property: for any strictly increasing sequence $\{p_n\}^\infty_{n=1}$ in $P$, we have
\begin{equation} \label{eq:sup}
\lim_{n\to\infty}N(p_n)^{-\beta}\Tr_\tau^{p_n}(1)=0.
\end{equation}
Then $\phi$ is gauge-invariant.
\end{proposition}

\bp In order to make the argument more transparent, let us assume first that $A$ is unital and the right C$^*$-Hilbert $A$-modules $X_p$ admit finite Parseval frames. For every $p\in P$, we fix such a frame $(\lambda^{(p)}_i)_i$ and define a completely positive map
$$
S_p\colon\NT(X)\to\NT(X)\ \ \text{by}\ \ S_p(x)=\sum_i i_X(\lambda^{(p)}_i)^*x\, i_X(\lambda^{(p)}_i).
$$
Then the KMS-state $\phi$ has the following property:
\begin{align}
\text{if}\ p\vee q<\infty\ \text{and}\ x\in{i_X(X_p)^*i_X(X_q)},&\ \text{then}\ \ S_{p^{-1}(p\vee q)}(x)\in{i_X(X_{q^{-1}(p\vee q)})^*i_X(X_{p^{-1}(p\vee q)})}\nonumber\\
&\ \text{and}\ \ \phi(x)=N(p^{-1}(p\vee q))^{-\beta}\phi(S_{p^{-1}(p\vee q)}(x)).\label{eq:scaling}
\end{align}

Indeed, it suffices to consider $x=i_X(\xi)^*i_X(\zeta)$, with $\xi\in X_p$ and $\zeta\in X_q$. Then the first statement is clear as, writing~$\lambda_i$ instead of~$\lambda^{(p^{-1}(p\vee q))}_i$ for simplicity, we have, using~\eqref{eq:Toeplitz}, that
$$
i_X(\lambda_i)^*i_X(\xi)^*i_X(\zeta)=i_X(\ell(\zeta)^*(\xi\lambda_i))^*\ \ \text{and}\ \ \ell(\zeta)^*(\xi\lambda_i)\in X_{q^{-1}(p\vee q)}.
$$
For the second statement note that since by~\eqref{eq:Nica} the element $i_X(\xi)^*i_X(\zeta)$ can be written as the sum of elements $i_X({\eta})i_X({\mu})^*$, with $\eta\in X_{p^{-1}(p\vee q)}$ and $\mu\in X_{q^{-1}(p\vee q)}$, and
$$
\sum_ii_X(\lambda_i)i_X(\lambda_i)^*i_X({\eta})=i_X({\eta}),
$$
we have $x=\sum_ii_X(\lambda_i)i_X(\lambda_i)^*x$. Therefore by applying $\phi$ and the KMS-condition we get
$$
\phi(x)=\sum_i\phi(i_X(\lambda_i)i_X(\lambda_i)^*x)=N(p^{-1}(p\vee q))^{-\beta}\sum_i\phi(i_X(\lambda_i)^*x\,i_X(\lambda_i)),
$$
which proves~\eqref{eq:scaling}.

Now, in order to prove that $\phi$ is gauge-invariant, we have to show that given $p\ne q$ and $\xi\in X_p$, $\zeta\in X_q$, we have $\phi(i_X(\zeta)i_X(\xi)^*)=0$. Applying the KMS-condition we get
$$
\phi(i_X(\zeta)i_X(\xi)^*)=N(q)^{-\beta}\phi(i_X(\xi)^*i_X(\zeta)).
$$
If $p\vee q=\infty$, we are done, since then $i_X(\xi)^*i_X(\zeta)=0$ by Nica covariance. Otherwise we let $x=i_X(\xi)^*i_X(\zeta)$ and proceed as follows.

Put $p_1=q^{-1}(p\vee q)$ and $q_1=p^{-1}(p\vee q)$. If $p\ne p\vee q$, we define $r_1=q_1$ and by~\eqref{eq:scaling} get
\begin{equation}\label{eq:scaling2}
S_{r_1}(x)\in i_X(X_{p_1})^*i_X(X_{q_1})\ \ \text{and}\ \ \phi(x)=N(r_1)^{-\beta}\phi(S_{r_1}(x)).
\end{equation}
If $p=p\vee q$, then we must have $q\ne p\vee q$. In this case we define $r_1=p_1$, and applying~\eqref{eq:scaling} to $x^*=i_X(\zeta)^*i_X(\xi)$ instead of $x$ we again get~\eqref{eq:scaling2}. Then we apply the same procedure to $S_{r_1}(x)$ instead of $x$, and so on.

There are two possibilities. One is that this process continues indefinitely. Then after $n$ steps we get elements $p_1,\dots,p_n,q_1,\dots,q_n,r_1,\dots,r_n\in P$, with $p_k\ne q_k$, $r_k\in\{p_k,q_k\}\setminus\{e\}$ and $p_k\vee q_k<\infty$ for all $k$, such that $(S_{r_n}\circ\dots\circ S_{r_1})(x)$ lies in $i_X(X_{p_n})^*i_X(X_{q_n})$ and
\begin{equation} \label{eq:scaling3}
\phi(x)=N(r_1)^{-\beta}\dots N(r_n)^{-\beta}\phi((S_{r_n}\circ\dots\circ S_{r_1})(x)).
\end{equation}
Since the positive linear functional $\phi\circ S_{r_n}\circ\dots\circ S_{r_1}$ has norm
$$
\phi((S_{r_n}\circ\dots\circ S_{r_1})(1))=\Tr^{r_1\dots r_n}_\tau(1),
$$
we then get
$$
|\phi(x)|\le N(r_1\dots r_n)^{-\beta}\Tr^{r_1\dots r_n}_\tau(1)\|x\|.
$$
Since the sequence $\{r_1\dots r_n\}^\infty_{n=1}$ is strictly increasing, letting $n\to\infty$, by the assumption of the proposition we obtain $\phi(i_X(\xi)^*i_X(\zeta))=\phi(x)=0$, hence also $\phi(i_X(\zeta)i_X(\xi)^*)=0$.

The other possibility is that the process stops after $n$ steps for some $n\ge1$, that is, we get $p_n\vee q_n=\infty$. Then $i_X(X_{p_n})^*i_X(X_{q_n})=0$ by Nica covariance, hence $(S_{r_n}\circ\dots\circ S_{r_1})(x)=0$, so we again get $\phi(x)=0$. This finishes the proof of the proposition under the additional assumption that $A$ is unital and the right $A$-modules $X_p$ are finitely generated.

\smallskip

In the general case the proof is basically the same, but at every step instead of exact identities we get approximate ones. Namely, for every $p\in P$, choose an approximate unit $(u^{(p)}_i)_{i\in I_p}$ in $\K(X_p)$ of the form $u^{(p)}_i=\sum_{\xi\in J^{(p)}_i}\theta_{\xi,\xi}$ for some finite sets $J^{(p)}_i\subset X_p$. Then define
$$
S^{(i)}_p\colon\NT(X)\to\NT(X)\ \ \text{by}\ \ S^{(i)}_p(x)=\sum_{\xi\in J^{(p)}_i} i_X(\xi)^*x\, i_X(\xi).
$$
Following the proof of~\eqref{eq:scaling} we then get, for $x\in\overline{i_X(X_p)^*i_X(X_q)}$, that $$S^{(i)}_{p^{-1}(p\vee q)}(x)\in\overline{i_X(X_{q^{-1}(p\vee q)})^*i_X(X_{p^{-1}(p\vee q)})},$$
$$
\phi(i_X^{(p^{-1}(p\vee q))}(u^{(p^{-1}(p\vee q))}_i)x)=N(p^{-1}(p\vee q))^{-\beta}\phi(S^{(i)}_{p^{-1}(p\vee q)}(x)),
$$
and $i_X^{(p^{-1}(p\vee q))}(u^{(p^{-1}(p\vee q))}_i)x\xrightarrow[i]{} x$. Correspondingly, instead of~\eqref{eq:scaling2} we get
$$
S^{(i_1)}_{r_1}(x)\in\overline{i_X(X_{p_1})^*i_X(X_{q_1})}\ \ \text{and}\ \ |\phi(x)-N(r_1)^{-\beta}\phi(S^{(i_1)}_{r_1}(x))|<\frac{\eps}{2},
$$
where $\eps>0$ is any fixed number and $i_1$ is a sufficiently large index in $I_{r_1}$. At the next step we replace~$\eps$ by~$\eps/2$, and so on. Then after $n$ steps instead of~\eqref{eq:scaling3} we get in addition to $p_k,q_k,r_k$ indices~$i_k$ such that
$$
|\phi(x)-N(r_1)^{-\beta}\dots N(r_n)^{-\beta}\phi((S^{(i_n)}_{r_n}\circ\dots\circ S^{(i_1)}_{r_1})(x))|<\sum^{n}_{k=1}\frac{\eps}{2^k}.
$$
Since $\|\phi\circ S^{(i_n)}_{r_n}\circ\dots\circ S^{(i)}_{r_1}||\le\Tr^{r_1\dots r_n}_\tau(1)$, this is still enough to conclude that $|\phi(x)|\le\eps$, and as $\eps>0$ was arbitrary, we get $\phi(x)=0$.
\ep

\begin{remark}
Since any strictly increasing sequence in $P$ tends to infinity (in the usual sense, that is, it eventually leaves every finite subset of $P$), the condition
\begin{equation} \label{eq:sup2}
\lim_{p\to\infty}N(p)^{-\beta}\Tr^p_\tau(1)=0
\end{equation}
is stronger than~\eqref{eq:sup}. Observe next that if $P$ has finite initial intervals (recall that by Remark~\eqref{rem:finite-intervals} this condition is automatically satisfied if $F_p\ne0$ for all $p$ and the series~\eqref{eq:abszeta} is convergent for some $\beta$), then~\eqref{eq:sup2} is satisfied for any trace $\tau$ of finite type. Indeed, when $p\to\infty$, the sets $pP$ eventually do not intersect any given finite subset of $P$, so if
$\tau=\sum_{q\in P}N(q)^{-\beta}F_q\tau_0$, then
$$
N(p)^{-\beta}\Tr^p_\tau(1)=\sum_{q\in pP}N(q)^{-\beta}\Tr^q_{\tau_0}(1)\to0.
$$
Therefore, under a very mild additional assumption, Proposition~\ref{prop:automatic-gauge} gives an alternative proof of the fact that if $\tau$ is of finite type, then it can only be extended to a gauge-invariant KMS-state on $\NT(X)$.
\end{remark}

\begin{corollary}\label{cor:automatic-gauge}
Assume $(G,P)$ is a quasi-lattice ordered group and $X$ is a compactly aligned product system of C$^*$-correspondences over $P$, with $X_e=A$. Suppose in addition that $A$ is unital and $X_p$ is finitely generated as a right $A$-module for all $p\in P$. Consider the dynamics~\eqref{eq:dynamics1} on  $\NT(X)$ defined by a homomorphism $N\colon P\to (0,+\infty)$, and assume that for some $\beta\ne0$ and $\delta>0$ we have
\begin{equation}\label{eq:largebeta}
N(p)^{\beta}\ge(1+\delta)m(p)
\end{equation}
for all $p\in P\setminus\{e\}$, where $m(p)$ is the minimal number of generators of $X_p$ as a right $A$-module. Then every $\sigma$-KMS$_\beta$-state on $\NT(X)$ is gauge-invariant.
\end{corollary}

Recall that by Remark~\ref{rem:Parseval} the number $m(p)$ equals the minimal cardinality of a Parseval frame in $X_p$.

\bp
We claim that condition~\eqref{eq:sup} is satisfied for any tracial state~$\tau$ on~$A$. It is clear from the definition of induced traces that if a right C$^*$-Hilbert $A$-module $Y$ has a Parseval frame consisting of $m$ elements, then $\Tr^Y_\tau(1)\le m$. Hence, for any $r_1,\dots,r_n\in P\setminus\{e\}$, we have
$$
\Tr_\tau^{r_1\dots r_n}(1)\le m(r_1)\dots m(r_n)\le(1+\delta)^{-n}N(r_1\dots r_n)^\beta,
$$
from which we see that~\eqref{eq:sup} is satisfied.
\ep

\begin{remark}
The function $m$ is submultiplicative on $P$, that is, $m(pq)\le m(p)m(q)$ for $p,q\in P$. It follows that \eqref{eq:largebeta} is satisfied for all $p\in P\setminus\{e\}$ once it is satisfied for all $p$ in a generating set of $P$. In particular, if $P$ is finitely generated and $N(p)>1$ for all $p$ in a finite generating set, then \eqref{eq:largebeta} is satisfied for all $\beta$ large enough. Such $\beta$'s can still be smaller than $\beta_c$, see~\cite{BLRS} and Example~\ref{ex:blrs} below.
\end{remark}

\begin{remark}
Although Corollary~\ref{cor:automatic-gauge} involves assumptions on $N$, it still has a rather straightforward generalization to arbitrary quasi-free dynamics. Namely, the analogues of the maps $N(p)^{-\beta}S^{(i)}_p$ from the proof of Proposition~\ref{prop:automatic-gauge} are
$$
\tilde S^{(i)}_p(x)=\sum_{\xi\in J^{(p)}_i} i_X(U^{(p)}_{i\beta/2}\xi)^*x\, i_X(U^{(p)}_{i\beta/2}\xi),
$$
where $J^{(p)}_i\subset X_p$ is a finite set of vectors analytic with respect to the one-parameter group $(U^{(p)}_t)_t$ (and such that $\sum_{\xi\in J^{(p)}_i}\theta_{\xi,\xi}\le1$). There are two ways of getting an estimate for $\phi(\tilde S^{(i)}_p(x))$. One is to use that by definition we have $(\phi\circ \tilde S^{(i)}_p)(1)\le\big(\operatorname{Ind}^{U^{(p)}}_{X_p}(\phi|_A)\big)(1)$. This leads to an analogue of Proposition~\ref{prop:automatic-gauge}. Another is to use that $\|\tilde S^{(i)}_p\|\le|J^{(p)}_i|\,\|U^{(p)}_{i\beta/2}\|^2$. This leads to an analogue of Corollary~\ref{cor:automatic-gauge} where~\eqref{eq:largebeta} gets replaced by the condition
$\|U^{(p)}_{i\beta/2}\|^{-2}>(1+\delta)m(p)$.
\end{remark}

\bigskip

\section{Right-angled Artin monoids} \label{sec:Artin}

In this section we will show that for a class of groups interpolating between free groups and free abelian groups condition~\eqref{eq:subinvariance} reduces to a much smaller system of inequalities.

\smallskip

Let $\Gamma$ be a simple graph, finite or infinite. The \emph{right-angled Artin group} $G_\Gamma$ associated to $\Gamma$ is a group with generators $s_i$ indexed by the vertices of $\Gamma$ and with relations
$$
s_is_j=s_js_i,\ \ \text{if}\ \ i\ \ \text{and}\ \ j\ \ \text{are connected by an edge}.
$$
We denote by $S_\Gamma$ the set of the generators $s_i$ and call it the standard generating set of $G_\Gamma$. Denote by $P_\Gamma\subset G_\Gamma$ the monoid generated by $S_\Gamma$. It is shown in~\cite{CL} that $(G_\Gamma,P_\Gamma)$ is a quasi-lattice ordered group.

\begin{theorem}\label{thm:Artin}
Consider the quasi-lattice ordered right-angled Artin group $(G_\Gamma,P_\Gamma)$ associated with a simple graph $\Gamma$. Assume we are given a product system $X$ of C$^*$-correspondences over~$P_\Gamma$, with $X_e=A$, a homomorphism $N\colon P_\Gamma\to(0,+\infty)$ and $\beta\in\R$. Then a tracial state $\tau$ on $A$ satisfies condition~\eqref{eq:subinvariance} if and only if
\begin{equation}\label{eq:subinvariance3}
\tau(a)+\sum_{\genfrac{}{}{0pt}{}{K\in C(S_\Gamma):}{\emptyset \ne K\subset J}}(-1)^{|K|}N(s_K)^{-\beta}\Tr^{s_K}_\tau(a)\ge0\ \ \text{for all finite}\ \ J\subset S_\Gamma\ \ \text{and}\ \ a\in A_+,
\end{equation}
where $C(S_\Gamma)$ is the collection of finite sets of pairwise commuting standard generators, or in other words, the collection of sets of generators corresponding to complete finite subgraphs of $\Gamma$, and $s_K=\prod_{s\in K}s$.
\end{theorem}

For the proof we need the following partial description of the operation $\vee$ on $P_\Gamma$, which is a simple consequence of~\cite[Proposition~13]{CL}.

\begin{lemma}\label{lem:veeArtin}
For any $p=s_{i_1}\dots s_{i_n}\in P_\Gamma$ and $s_i\in S_\Gamma$, we have
\begin{itemize}
\item[(a)] if $i\in\{i_1,\dots,i_n\}$ and for the lowest index $k$ such that $i=i_k$ the vertex $i$ and each vertex~$i_j$ with $j<k$ are connected by an edge, then $p\vee s_i=p$;
\item[(b)] if $i\notin \{i_1,\dots,i_n\} $, but the vertex $i$ and each vertex $i_j$ are connected by an edge, then $p\vee s_i=ps_i=s_ip$;
\item[(c)] if neither {\rm (a)} nor {\rm (b)} applies, then $p\vee s_i=\infty$.
\end{itemize}
\end{lemma}

An immediate consequence of this lemma is that if $K\subset S_\Gamma$ is a finite set of generators, then
\begin{equation}\label{eq:veeArtin}
q_K=\bigvee_{s\in K}s=\begin{cases}s_K,&\text{if}\ K\in C(S_\Gamma),\\ \infty,&\text{otherwise}.\end{cases}
\end{equation}

\begin{lemma}\label{lem:Artin-decomposition}
Let $(G,P)=(G_\Gamma,P_\Gamma)$. Then every set $\Omega\in\B_P$ can be written as a finite disjoint union of the sets $pP$ and $p(\cap_{s\in J}(P\setminus sP))$, where $p\in P$ and $J\subset S_\Gamma$ is a finite nonempty set.
\end{lemma}

\bp
By Lemma~\ref{lem:BP-decomposition}(i) we know that $\Omega$ can be written as a finite disjoint union of the sets $pP$ and $p(\cap_{q\in J}(P\setminus qP))$, where $J\subset P\setminus\{e\}$ is finite. In the first case there is nothing to prove. In the second case each set $P\setminus qP$ can be further decomposed as follows. Write $q$ as $s_{k_1}\dots s_{k_m}$. Then $P\setminus qP$ is the disjoint union of the sets $s_{k_1}\dots s_{k_{i-1}}P\setminus s_{k_1}\dots s_{k_i}P$ for $i=1,\dots,m$. It follows that every set $p(\cap_{q\in J}(P\setminus qP))$, where $J\subset P\setminus\{e\}$ is finite, can be written as the disjoint union of sets of the form
$$
p(q_1P\setminus q_1s_{i_1}P)\cap\dots\cap p(q_nP\setminus q_ns_{i_n}P).
$$

With sets of the latter form we proceed as follows. We may assume that $q=q_1\vee\dots\vee q_n$ is finite, as otherwise we get the empty set. Then the above set is contained in $pqP$, so we can write it as
\begin{equation}\label{eq:decomposeArtin}
p(qP\setminus q_1s_{i_1}P)\cap\dots\cap p(qP\setminus q_ns_{i_n}P).
\end{equation}

Now, for every $j=1,\dots,n$, put $r_j=q_j^{-1}q$. Then
$$
p(qP\setminus q_js_{i_j}P)=pq_j(r_jP\setminus s_{i_j}P)=pq_j(r_jP\setminus (r_j\vee s_{i_j})P),
$$
with the convention $\infty P=\emptyset$.
But by Lemma~\ref{lem:veeArtin} we have only three options for the set $(r_j\vee s_{i_j})P$:
$\ r_jP,\ r_js_{i_j} P,\ \emptyset$.
Correspondingly, the only options for the set
$p(qP\setminus q_js_{i_j}P)$ are
$$
\emptyset,\ \ pq(P\setminus s_{i_j}P),\ \ pqP.
$$
Hence the set~\eqref{eq:decomposeArtin} is either empty, $pq(\cap_{s\in J}(P\setminus sP))$, where $J\subset\{s_{i_1},\dots,s_{i_n}\}$, or $pqP$.
\ep

\bp[Proof of Theorem~\ref{thm:Artin}]
Denote $P_\Gamma$ by $P$.
From~\eqref{eq:veeArtin} we see that condition~\eqref{eq:subinvariance3} is nothing other than condition~\eqref{eq:subinvariance} applied to the sets $J\subset S_\Gamma$. Thus~\eqref{eq:subinvariance} implies~\eqref{eq:subinvariance3}.

In order to prove the converse, observe first of all that if~\eqref{eq:subinvariance3} is satisfied, then $F_s\tau$ is finite and dominated by a scalar multiple of $\tau$ for all $s\in S_\Gamma$, hence $\Tr^p_\tau$ is finite for all $p\in P$ by~\eqref{eq:induction2}.

Consider now, as in Section~\ref{sec:restriction}, the $T(A)$-valued finitely additive measure $\mu$ on $(P,\B_P)$ defined by $\mu(pP)=N(p)^{-\beta}F_p\tau$. By assumption we have
$\mu(\Omega_J)\ge0$ for all finite $J\subset S_\Gamma$, where $\Omega_J=\cap_{s\in J}(P\setminus sP)$. But then, by~\eqref{eq:mu}, we have
$$
\mu(p\Omega_J)=N(p)^{-\beta}F_p\mu(\Omega_J)\ge0
$$
for all $p\in P$ and finite $J\subset S_\Gamma$. By Lemma~\ref{lem:Artin-decomposition} we conclude that $\mu$ is positive, that is, condition~\eqref{eq:subinvariance} is satisfied.
\ep

\begin{remark}
When $\Gamma$ is a finite complete graph, so that $P_\Gamma=\Z^n_+$, Theorem~\ref{thm:Artin} gives another proof of equivalence of conditions~\eqref{eq:subinvariance} and~\eqref{eq:subinvariance2}, which was promised in Section~\ref{sec:abelian}.

When~$\Gamma$ is a finite graph with no edges, so that $P_\Gamma$ is a free monoid $\FF^+_n$, then Theorem~\ref{thm:Artin} implies that condition~\eqref{eq:subinvariance} is equivalent to
$$
\sum^n_{i=1}N(s_i)^{-\beta}F_i\tau\le\tau,
$$
where $F_i=F_{s_i}$. However, a product system over a free monoid is determined by one graded correspondence $\oplus_i X_{s_i}$, so Theorem~\ref{thm:general} and the equivalence of the above inequality to~\eqref{eq:subinvariance} in this case follows already from~\cite{lac-nes}.
\end{remark}

The following corollary shows that for the right-angled Artin monoids the set of possible temperatures of KMS-states on $\NT(X)$ is often a half-line.

\begin{corollary}\label{cor:monotonicity}
In the setting of Theorem~\ref{thm:Artin}, assume that $N(p)\ge1$ for all $p\in P_\Gamma$ and $\tau$ is a tracial state on $A$ satisfying~\eqref{eq:subinvariance} for some $\beta=\beta_0$. Then $\tau$ satisfies \eqref{eq:subinvariance} for all $\beta>\beta_0$.
\end{corollary}

\bp
Denote $P_\Gamma$ by $P$.
We will prove the following statement: if $N$ and $N'$ are two homomorphisms $P\to (0,+\infty)$ such that $N'(p)\le N(p)$ for all $p\in P$, and $\tau$ satisfies~\eqref{eq:subinvariance3} for $N$ and $\beta=-1$, then it also satisfies~\eqref{eq:subinvariance3} for $N'$ and $\beta=-1$. The corollary then follows by applying this to the homomorphisms $N^{-\beta_0}$ and $N^{-\beta}$ and using that conditions~\eqref{eq:subinvariance3} and~\eqref{eq:subinvariance} are equivalent.

The argument is analogous to the one used in the proof of Theorem~\ref{thm:free} for a similar statement, but becomes slightly more complicated in the general case of right-angled Artin monoids.

Define a map $\eta$ from the linear span of the characteristic functions $\chi_{pP}$, $p\in P$, into $T(A)$ by letting $\eta(\chi_{pP})=F_p\tau$. In view of~\eqref{eq:veeArtin}, condition~\eqref{eq:subinvariance3} for $N$ and $\beta=-1$ can be written as
$$
\eta\Big(\prod_{s\in J}(1-N(s)\chi_{sP})\Big)\ge0
$$
for all finite $J\subset S_\Gamma$, and we have to show that the same is true with $N$ replaced by $N'$.

Using that
$$
\prod_{s\in J}(1-N'(s)\chi_{sP})=\prod_{s\in J}((1-N(s)\chi_{sP})+(N(s)-N'(s))\chi_{sP}),
$$
we see that we can write the function $\prod_{s\in J}(1-N'(s)\chi_{sP})$ as a linear combination with positive coefficients of functions of the form
$$
\Big(\prod_{s\in E}\chi_{sP}\Big)\Big(\prod_{t\in F}(1-N(t)\chi_{tP})\Big),
$$
where $E$ and $F$ are finite disjoint sets of generators. If $E$ or $F$ are empty, then $\eta$ is positive on such a function by our assumptions.
Otherwise, by~\eqref{eq:veeArtin}, we may assume that the elements of~$E$ commute with each other. Using~\eqref{eq:veeArtin} again, we see also that if an element of $F$ does not commute with an element of $E$, then the factor $1-N(t)\chi_{tP}$ can be replaced by $1$. Therefore we may assume that the elements of $E$ commute with each other and with the elements of $F$. Then the above function equals
$$
\chi_{s_EP}+\sum_{\genfrac{}{}{0pt}{}{K\in C(S_\Gamma):}{\emptyset \ne K\subset F}}(-1)^{|K|}N(s_K)\chi_{s_Es_KP},
$$
from which we see that the value of $\eta$ on this function is
$$
F_{s_E}\tau+\sum_{\genfrac{}{}{0pt}{}{K\in C(S_\Gamma):}{\emptyset \ne K\subset F}}(-1)^{|K|}N(s_K)^{-\beta}F_{s_Es_K}\tau=F_{s_E}\eta\Big(\prod_{t\in F}(1-N(t)\chi_{tP})\Big)\ge0.
$$
This completes the proof of the corollary.
\ep

As an immediate application let us consider the following example.

\begin{example}\label{ex:blrs}
Let $(G,P)$ be a quasi-lattice ordered group such that $P$ has a finite generating set $S\subset P\setminus\{e\}$. Assume we are given a homomorphism $N\colon P\to(0,+\infty)$ such that $N(p)>1$ for all $p\in P\setminus\{e\}$ and the abscissa of convergence of the series $\sum_{p\in P} N(p)^{-\beta}$ satisfies $\beta_c>0$. Consider the full semigroup C$^*$-algebra of $P$, that is, the algebra $\NT(X)$ for the trivial product system $X_p=\C$, and the dynamics $\sigma$ on it defined by $N$. Then, using our terminology, Theorem~3.5 and Proposition~3.5 in~\cite{BLRS} can be formulated as follows:
\begin{itemize}
\item[(i)] for every $\beta>\beta_c$ there is a unique $\sigma$-KMS$_\beta$-state, and it is of finite type;
\item[(ii)] there is a unique $\sigma$-KMS$_{\beta_c}$-state, and it is of infinite type;
\item[(iii)] if for some $\beta\in[0,\beta_c)$ there is a $\sigma$-KMS$_\beta$-state, then it is of infinite type.
\end{itemize}
We remark that parts (i) and (iii) follow also from our results in Section~\ref{sec:Wold}. Indeed, by the definition of $\beta_c$, the unique tracial state on $A=\C$ is of finite type for $\beta>\beta_c$, so we get (i) by Corollary~\ref{cor:finite-class}, and there are no traces of finite type for $\beta<\beta_c$, so we get (iii). Part (ii) also follows from our Proposition~\ref{prop:phase-transition} and Corollary~\ref{cor:automatic-gauge}, but the proofs of those results use ideas from~\cite{BLRS} and do not provide any new perspective here.

By Corollary~\ref{cor:finite-infinite}(i), part (iii) above means that $\beta$ must satisfy the equation
$$
1+\sum_{\emptyset\neq K\subset S} (-1)^{\vert K\vert} N(q_K)^{-\beta}=0,
$$
an observation which was already made in~\cite{BLRS}. This implies that the set of possible $\beta$'s in (iii) is finite.

Consider now the case when $(G,P)=(G_\Gamma,P_\Gamma)$ for a finite simple graph $\Gamma$. The assumption $\beta_c>0$ means precisely that $G$ is not abelian, that is, $\Gamma$ is not a complete graph. By Corollary~\ref{cor:monotonicity} we know that the set of possible inverse temperatures is a half-line. We can therefore conclude that there are no $\sigma$-KMS$_\beta$-states on the full semigroup C$^*$-algebra of $P_\Gamma$ for $\beta<\beta_c$, so for such systems the classification of KMS-states reduces completely to the computation of $\beta_c$.
\end{example}

\bigskip

\section{Direct products of quasi-lattice ordered monoids}\label{sec:products}

Throughout this section we fix a quasi-lattice ordered group $(G,P)$ of the form $(G_1,P_1)\times\dots\times(G_n,P_n)$, a compactly aligned product system $X$ of C$^*$-correspondences over $P$, with $X_e=A$, a homomorphism $N\colon P\to(0,+\infty)$, and $\beta\in\R$. In this setting, developing the ideas in~\cite{kak,C}, we can refine the decomposition of states on $\NT(X)$ into positive functionals of finite and infinite types by considering states that are finite with respect to some factors and infinite with respect to other.

\smallskip

We need to introduce some notation in order to formulate the precise result.
Whenever convenient we view $G_i$ as a subgroup of $G$. For any set $F\subset\{1,\dots,n\}$ we denote by $P_F$ the submonoid of~$P$ generated by the monoids $P_i$ for $i\in F$, with the convention $P_\emptyset=\{e\}$. Denote by~$X^{(F)}$ the product system $(X_p)_{p\in P_F}$. When $F=\{i\}$ we write $X^{(i)}$ instead of $X^{(\{i\})}$, and we use the same convention for various constructions below. The embeddings $X^{(F)}\hookrightarrow X$ induce $*$-homomorphisms $\NT(X^{(F)})\to\NT(X)$.

\smallskip

Assume now that we are given a representation $\pi\colon \NT(X)\to B(H)$. Consider the von Neumann algebra $M=\pi(\NT(X))''$. By restriction we get representations $\pi_F\colon\NT(X^{(F)})\to B(H)$. Put $M_F=\pi_F(\NT(X^{(F)}))''\subset M$. Clearly, $M_F=\vee_{i\in F}M_i$. Following Proposition~\ref{prop:Wold-rep} and its proof, for every $F$ we consider the projection $Q_F\in M_F$ onto the space of vacuum vectors for $\NT(X^{(F)})$ and its central support $z_F\in M_F$.

\begin{proposition}
In the above setting we have:
\begin{itemize}
\item[(i)] for every $i=1,\dots,n$, the projection $Q_i$ commutes with $M_j$ for $j\ne i$ and the projection~$z_i$ is central in $M$;
\item[(ii)] $Q_F=\prod_{i\in F}Q_i$ and $z_F=\prod_{i\in F}z_i$ for every set $F\subset\{1,\dots,n\}$.
\end{itemize}
\end{proposition}

\bp (i) The first statement says that the space $Q_iH$ is invariant under $M_j$. In order to prove this we have to show that if $v\in Q_iH$, $\xi\in X_p$ for some $p\in P_i$ and $\zeta\in X_q$ for some $q\in P_j$, then
$$
\pi(i_X(\xi))^*\pi(i_X(\zeta))v=0\ \ \text{and}\ \ \pi(i_X(\xi))^*\pi(i_X(\zeta))^*v=0.
$$
The first equality follows from \eqref{eq:Nica}, since $p\vee q=pq=qp$. The second equality holds since $\zeta\xi\in X_{qp}=X_{pq}$ can be approximated by finite sums of elements of the form $\xi'\zeta'$, with $\xi'\in X_p$ and $\zeta'\in X_q$.

For the second statement we have to show that the space $z_iH$ is invariant under $M_j$ for all $j\ne i$. Since this space is the closed linear span of the vectors $\pi(i_X(\xi))v$, where $v\in Q_iH$ and $\xi\in X_p$ ($p\in P_i$), this can be proved using arguments similar to the ones above.

\smallskip

(ii) The first statement is an immediate consequence of (i), since by definition we have $Q_FH=\cap_{i\in F}\,Q_iH$. For the second statement observe that for each $i\in F$ we have
$$
\overline{M_iQ_FH}=\overline{\Big(\prod_{j\in F\setminus\{i\}}Q_j\Big)M_iQ_iH}=\Big(\prod_{j\in F\setminus\{i\}}Q_j\Big)z_iH.
$$
So by applying $M_i$ ($i\in F$) one by one to the space $Q_FH$ we see that we can generate the entire space $\big(\prod_{i\in F}z_i\big)H$. Since this space is invariant under $M_F$, we conclude that the central support of $Q_F$ in $M_F$ equals $\prod_{i\in F}z_i$.
\ep

\begin{corollary}
Every positive linear functional $\phi$ on $\NT(X)$ uniquely decomposes as
$$
\phi=\sum_{F\subset\{1,\dots,n\}}\phi_F,
$$
where $\phi_F$ is a positive linear functional on $\NT(X)$ defining by restriction a functional on $\NT(X^{(i)})$ which is of finite type for every $i\in F$ and of infinite type for every $i\in F^c$. Furthermore, $\phi_F$ defines a functional of finite type on $\NT(X^{(F)})$.
\end{corollary}

\bp
The proof is similar to that of Proposition~\ref{prop:Wold-state}. Consider the GNS-triple $(H_\phi,\pi_\phi,v_\phi)$ defined by $\phi$. Let $z_i$ be the central projections in $M=\pi(\NT(X))''$ as in the above proposition for $\pi=\pi_\phi$. Define central projections
$$
w_F=\Big(\prod_{i\in F}z_i\Big)\Big(\prod_{j\in F^c}(1-z_j)\Big)\in M.
$$
Then the functionals $\phi_F=(\pi(\cdot)w_Fv_\phi,v_\phi)$ give the desired decomposition. Note that since $w_F\le z_F$, the restriction of $\phi_F$ to $\NT(X^{(F)})$ gives a  functional of finite type.

In order to prove the uniqueness, it suffices to show that if $\psi\le\phi$ and $\psi$ defines a functional of finite type on $\NT(X^{(i)})$ for $i\in F$ and of infinite type for $i\in F^c$, then $\psi\le\phi_F$.

Let $x\in M'$, $0\le x\le1$, be the unique element such that $\psi=(\pi(\cdot)xv_\phi,v_\phi)$. For every $i$, put $K_i=\overline{M_ix^{1/2}v_\phi}$. Then as the GNS-triple associated with $\psi|_{\NT(X^{(i)})}$ we can take $(K_i,\pi_i|_{K_i},x^{1/2}v_\phi)$. It follows that if $i\in F$, then $x^{1/2}v_\phi\in K_i\subset z_iH_\phi$. Applying $M$ and using that $z_i$ is a central projection in $M$, we then conclude that $x^{1/2}H_\phi\subset z_iH_\phi$, hence $x\le z_i$. On the other hand, if $i\in F^c$, then the vector $z_ix^{1/2}v_\phi$ defines a functional of finite type on $\NT(X^{(i)})$ dominated by $\psi$, so it must be zero. Hence $x^{1/2}v_\phi\in (1-z_i)H_\phi$ and then $x\le 1-z_i$. It follows that $x\le w_F$ and therefore $\psi\le\phi_F$.
\ep

Note that the finite/infinite decomposition of $\phi$ is then given by
$$
\phi_f=\phi_{\{1,\dots,n\}},\ \ \phi_\infty=\sum_{F\subsetneq\{1,\dots,n\}}\phi_F.
$$

Similarly to Corollary~\ref{cor:Wold-KMS}, from the construction of the functionals $\phi_F$ we see also that if $\phi$ satisfies the $\sigma$-KMS$_\beta$ condition for some dynamics $\sigma$ on $\NT(X)$, then $\phi_F$ satisfy this condition as well.

For the dynamics we are interested in, this leads to the following refinement of Corollary~\ref{cor:Wold-trace}.

\begin{corollary}\label{cor:Wold-trace2}
With $N\colon P\to(0,+\infty)$ and $\beta\in\R$ fixed, every positive trace $\tau$ on $A=X_e$ satisfying~\eqref{eq:subinvariance} decomposes uniquely as
$$
\tau=\sum_{F\subset\{1,\dots,n\}}\tau_F,
$$
where $\tau_F$ is a positive trace satisfying~\eqref{eq:subinvariance} for $P$ such that $\tau_F$ is of finite type with respect to~$P_i$ for every $i\in F$ and of infinite type for every $i\in F^c$. Furthermore, $\tau_F$ is of finite type with respect to $P_F$.
\end{corollary}

Let us also observe the following.

\medskip

\begin{lemma}~\label{lem:product}
In order to verify condition~\eqref{eq:subinvariance} for a positive trace $\tau$ on $A$ it suffices to consider subsets $J\subset P\setminus\{e\}$ of the form $J=\prod_i J_i$, $J_i\subset P_i$.
\end{lemma}

\bp As $pP=\prod_ip_iP_i$ for $p=(p_1,\dots,p_n)$, every set in $\B_P$ decomposes into the disjoint union of sets of the form $\prod_i\Omega_i$, with $\Omega_i\in\B_{P_i}$. By Lemma~\ref{lem:BP-decomposition}(i), each set $\Omega_i$ decomposes into a disjoint union of the sets $p_iP_i$ and $p_i(\cap_{q\in J_i}(P_i\setminus qP_i))$. The result now follows by the same argument as in the proof of Theorem~\ref{thm:Artin}.
\ep

Together with Corollary~\ref{cor:Wold-trace2} this lemma can in principle simplify analysis of traces satisfying~\eqref{eq:subinvariance} for direct product monoids. Namely, we have the following result.

\begin{proposition}\label{prop:Wold-trace2}
In the above setting, with $N\colon P\to(0,+\infty)$ and $\beta\in\R$ fixed, assume that for every subset $F\subset\{1,\dots,n\}$ we are given a positive trace $\tau_{F,0}$ on $A=X_e$ such that
\begin{itemize}
\item[(a)] $\tau_{F,0}$ satisfies condition~\eqref{eq:subinvariance} for $P_{F^c}$;
\item[(b)] $\tau_{F,0}$ is of infinite type with respect to $P_i$ for every $i\in F^c$;
\item[(c)] $\sum_{p\in P_F}N(p)^{-\beta}\Tr^p_{\tau_{F,0}}(1)<\infty$.
\end{itemize}
Then the trace
$$
\tau=\sum_{F\subset\{1,\dots,n\}}\sum_{p\in P_F}N(p)^{-\beta}F_p\tau_{F,0}
$$
satisfies condition~\eqref{eq:subinvariance} for $P$, and every positive trace satisfying~\eqref{eq:subinvariance} is obtained this way for uniquely defined positive traces $\tau_{F,0}$ satisfying conditions (a)-(c).
\end{proposition}

\bp
Given a collection of traces $\tau_{F,0}$ as in the formulation, it is easy to see that, for every $F$, the trace $\sum_{p\in P_F}N(p)^{-\beta}F_p\tau_{F,0}$ is still of infinite type with respect to $P_i$ for every $i\in F^c$ and, using Lemma~\ref{lem:product} and identity~\eqref{eq:mu2}, that it satisfies condition~\eqref{eq:subinvariance} for $P$. Hence the trace
$$
\tau=\sum_{F\subset\{1,\dots,n\}}\sum_{p\in P_F}N(p)^{-\beta}F_p\tau_{F,0}
$$
satisfies condition~\eqref{eq:subinvariance} for $P$.

Conversely, starting with a trace satisfying condition~\eqref{eq:subinvariance} for $P$, we first apply Corollary~\ref{cor:Wold-trace2} and get a decomposition $\tau=\sum_{F\subset\{1,\dots,n\}}\tau_F$. Then, by Corollary~\ref{cor:Wold-trace} applied to the monoids~$P_F$, we get uniquely defined positive traces $\tau_{F,0}$ such that
$$
\tau_F=\sum_{p\in P_F}N(p)^{-\beta}F_p\tau_{F,0}.
$$
From formula~\eqref{eq:tau0} for $\tau_{F,0}$ it is easy to see that every trace $F_p\tau_{F,0}$ ($p\in P_F$) satisfies condition~\eqref{eq:subinvariance} for $P_{F^c}$. It follows then that since $\tau_F$ is infinite with respect to $P_i$ for every $i\in F^c$, the traces $F_p\tau_{F,0}$ are infinite with respect to $P_i$ as well. Therefore conditions (a)-(c) are satisfied.

The uniqueness statement follows from Corollaries~\ref{cor:Wold-trace2} and~\ref{cor:Wold-trace}.
\ep

The key point of this proposition is that condition (b) might give strong restrictions on possible traces and be easier to understand since it involves only individual factors of $P$.

\begin{example}\label{ex:kak}
Consider $P=\Z^n_+$. Denote as usual the standard generators of $\Z^n_+$ by $e_1,\dots,e_n$ and write $F_i$ instead of $F_{e_i}$. Condition (b) in Proposition~\ref{prop:Wold-trace2} for $\tau_{F,0}$ means that $$N(e_i)^{-\beta}F_i\tau_{F,0}=\tau_{F,0}$$ for all $i\in F^c$. For every such trace condition~\eqref{eq:subinvariance} is satisfied for $\Z^{F^c}_+$, namely, the left hand side of~\eqref{eq:subinvariance} side is always zero (this is particularly transparent for the equivalent condition~\eqref{eq:subinvariance2}). We therefore conclude that if for every $F\subset\{1,\dots,n\}$ we are given a positive trace~$\tau_{F,0}$ on $A$ such that
\begin{itemize}
\item[(a$'$)] $N(e_i)^{-\beta}F_i\tau_{F,0}=\tau_{F,0}$ for every $i\in F^c$,
\item[(b$'$)] $\sum_{p\in \Z^F_+}N(p)^{-\beta}\Tr^p_{\tau_{F,0}}(1)<\infty$,
\end{itemize}
then the trace
\begin{equation}\label{eq:kak}
\tau=\sum_{F\subset\{1,\dots,n\}}\sum_{p\in \Z^F_+}N(p)^{-\beta}F_p\tau_{F,0}
\end{equation}
satisfies condition~\eqref{eq:subinvariance} (or, equivalently,~\eqref{eq:subinvariance2}) for $\Z^n_+$, and any positive trace satisfying this condition is obtained this way for uniquely defined positive traces $\tau_{F,0}$ satisfying conditions (a$'$) and (b$'$).

For finite rank product systems over $\Z^n_+$ this leads to an alternative proof of~\cite[Theorem~4.4]{kak}, which asserts that the tracial states $\tau$ as in~\eqref{eq:kak} are exactly the restrictions of the gauge-invariant KMS$_\beta$-states on $\NT(X)$. Moreover, we see that the result extends to compactly aligned product systems.
\end{example}

\bigskip

\section{Some examples and applications}\label{sec:examples}

In this section we will consider a few more examples that have been recently studied in the literature and show how our results allow us to recover them, and sometimes strengthen, in a quick unified way.

\subsection{Product systems arising from surjective local homeomorphisms}

Let $(G,P)$ be a quasi-lattice ordered group. Assume $P^{\mathrm{op}}$ acts on a compact Hausdorff space $Z$ by surjective local homeomorphisms $h_p\colon Z\to Z$, so that we have $h_{pq}=h_q\circ h_p$. For every $p\in P$ we can consider a C$^*$-correspondence $X_p$ over $A=C(Z)$ as in Example~\ref{ex:Ruelle}; in other words, $X_p$ is the correspondence associated with the topological graph $(Z,Z,\id,h_p)$. These correspondences form a product system, with the product defined by
$$
(\xi\zeta)(z)=\xi(z)\zeta(h_p(z))\ \ \text{for}\ \ \xi\in X_p,\ \ \zeta\in X_q.
$$
Since the right $C(Z)$-modules $X_p$ are finitely generated, the left actions of $C(Z)$ on them are by generalized compact operators and the product system $X$ is compactly aligned.

By Example~\ref{ex:Ruelle}, the corresponding operators $F_p\colon C(Z)^*\to C(Z)^*$ coincide with the dual Ruelle transfer operators $\LL^*_p$. We thus see that in order to understand the KMS$_\beta$-states of~$\NT(X)$ with respect to the dynamics given by a homomorphism $N\colon P\to (0,+\infty)$, and to get a complete classification of such gauge-invariant states, we need to study the representation of $P$ on $C(Z)^*$ given by $p\mapsto N(p)^{-\beta}\LL^*_p$. Spectral analysis of just one dual transfer operator is already a difficult problem in dynamical systems theory, but a few things can be said on a general basis.

\smallskip

Consider the case $P=\Z^n_+$, so $X$ is described by a family of $n$ commuting surjective local homeomorphisms on $Z$. In this case we get the following result.

\begin{proposition}\label{prop:local homeomorphisms}
Let $h_1,\dots,h_n$ be  commuting surjective local homeomorphisms of a compact Hausdorff space  $Z$ and let
 $X$ be the corresponding product system over $\Z^n_+$ as  above.  Let $\sigma$  be the dynamics \eqref{eq:dynamics1} on $\NT(X)$  determined by a homomorphism $N\colon \Z^n_+\to (0,+\infty)$. For each state $\phi$ on $\NT(X)$, let $\mu_\phi$ be  the probability measure on $Z$ such that $\phi(a)=\int_Z a\,d\mu_\phi$ for all $a\in C(Z)$. Then
\begin{itemize}
\item[(i)] for every $\beta\in \R$,  the map $\phi\mapsto \mu_\phi$ defines a one-to-one correspondence between the gauge-invariant $\sigma$-KMS$_\beta$-states on $\NT(X)$ and the probability measures $\mu$ on $Z$  such that
\begin{equation}\label{subinvariance measure1}
\prod_{i\in J}\big(1-N(e_i)^{-\beta }\LL^*_i\big)\mu\geq 0
\end{equation}
for every nonempty subset $ J\subset \{1,\dots, n\}$, where $\LL^*_i=\LL^*_{e_i}$; in particular, there are no $\sigma$-KMS$_\beta$-states unless $N(e_i)^{\beta}\ge1$ for all $i$.
\end{itemize}

Assume in addition that $N(e_i)>1$ for all $i$ and put
$$
r_i=\lim_{j\to \infty}\big( \max_{z\in Z}|h^{-j}_i(z)|\big)^{1/j},\ \ i=1,\dots,n,
$$
and
$$
\beta_c=\max_{1\le i\le n}\frac{\log r_i}{\log N(e_i)}.
$$
Then
\begin{itemize}
\item[(ii)] the number $\beta_c$ is the  abscissa of convergence of the series $\sum_{p\in \Z_+^n}N(p)^{-\beta}\|\LL^*_p\|$, and for
every $\beta>\beta_c$ all $\sigma$-KMS$_\beta$-states on $\NT(X)$ are of finite type, so there is an affine bijection between such states and the set
\[\Delta_\beta=\big\{\nu: \nu \text{ is a positive measure on } Z\text{ and } \sum_{p\in \Z^n_+}N(p)^{-\beta}\LL^*_p\nu \text{ is a probability measure}\big\};\]
\item[(iii)] at $\beta=\beta_c$ we have a phase transition, namely, there is a $\sigma$-KMS$_{\beta_c}$-state on $\NT(X)$ of infinite type.
\end{itemize}
\end{proposition}

\begin{proof}
(i) The first statement follows from  Corollary~\ref{cor:classfree}. For the second one observe that $\LL_i(1)\ge 1$, so if $N(e_i)^{-\beta}\LL^*_i\mu\le\mu$ for some positive measure $\mu$, then we must have $N(e_i)^{-\beta}\le1$.

Turning to (ii) and (iii), note that
$$
\|\LL^j_i\|=\|\LL^j_i(1)\|=\max_{z\in Z}|h_i^{-j}(z)|.
$$
This implies that $r_i$ is the spectral radius of $\LL_i$, or equivalently, of $\LL_i^*$. Hence, by \eqref{eq:critical-beta}, the number $\beta_c$ coincides with the critical inverse temperature defined in Section~\ref{sec:critical-temp}. As was already remarked there, the critical inverse temperature coincides with the abscissa of convergence of the series $\sum_{p\in \Z_+^n}N(p)^{-\beta}\|\LL^*_p\|$. The rest of (ii) and (iii) follows from Corollary~\ref{cor:finite-class3} and Proposition~\ref{prop:phase-transition}.
\end{proof}

\begin{remark}
The number $\log r_i$ often coincides with the topological entropy of $h_i$, see~\cite{FFN}. It is also interesting to note that by the proof of \cite[Proposition~2.3]{FFN},  without any extra assumptions on $h_i$, there is a point $z_i\in Z$ such that $r_i=\lim_{j\to\infty}|h^{-j}_i(z_i)|^{1/j}$.
\end{remark}

The description of KMS$_\beta$-states in part (ii) of the above proposition has been obtained in~\cite[Theorem~6.1]{AaHR} under the additional assumptions that the homeomorphisms $h_1,\dots,h_n$ $*$-commute and the numbers $\log N(e_1)$, $\dots$, $\log N(e_n)$ are rationally independent. Note that the operators~$\LL_i^*$ were denoted by $R^{e_i}$ in op.~cit.

\smallskip

Since  $X$ satisfies all the conditions of Remark~\ref{re:cuntzalg},  the algebras $\mathcal{NO}_X$, $\OO_X$ and $\NOs(X)$  coincide and therefore the next proposition follows immediately from Corollary~\ref{cor:cuntzalg}.

\begin{proposition}\label{prop:local homeomorphisms2}
In the setting of the previous proposition, for every $\beta\in\R$, we have an affine one-to-one correspondence between the gauge-invariant $\sigma$-KMS$_\beta$-states on $\mathcal{NO}_X$ and the probability measures $\mu$ on $Z$  such that
\begin{equation*}\label{subinvariance measure2}
N(e_i)^{-\beta}\LL^*_i\mu=\mu\ \ \text{for}\ \ i=1,\dots,n.
\end{equation*}
\end{proposition}

It is worth stressing that for $n\ge 2$ it is not enough for a $\sigma$-KMS$_\beta$-state $\phi$ on $\NT(X)$ to be of infinite type to factor through $\mathcal{NO}_X$, since infiniteness only means that $\mu=\phi|_{C(Z)}$ is in the kernel of the operator $\prod^n_{i=1}(1-N(e_i)^{-\beta}\LL^*_i)$. The difference is particularly illuminating in terms of the decomposition $\phi=\sum_{F\subset\{1,\dots,n\}}\phi_F$ from Section~\ref{sec:products}: $\phi$ is infinite if and only if $\phi_{\{1,\dots,n\}}=0$, while $\phi$ factors through $\mathcal{NO}_X$ if and only if $\phi_F=0$ for all $F\ne\emptyset$. See also~\cite{kak} for a relation between the components $\phi_F$ and intermediate quotients of $\NT(X)$.

\smallskip

Let us now give an example showing that the system of inequalities in~\eqref{eq:subinvariance2} cannot be replaced by a proper subsystem, which was promised in Section~\ref{sec:abelian}.

\begin{example}\label{ex:optimal}
Consider the space $Y=\{0,1\}^\N$. Let $h\colon Y\to Y$ be the shift map and $\nu$ be the Bernoulli measure on $Y$ with weights $(1/2,1/2)$.  Then for the dual Ruelle transfer operator $\LL^*$ we have $\LL^*\nu=2\nu$.

Fix $n\ge2$ and a nonempty subset $I\subset\{1,\dots,n\}$. Denote by $Z$ the disjoint union of two copies of $Y$ and define commuting surjective local homeomorphisms $h_i\colon Z\to Z$, $1\le i\le n$, as follows. If $i\in I^c$, then $h_i$ preserves each copy of $Y$ and coincides with $h$ on each of them. If $i\in I$, then $h_i$ is the composition of the local homeomorphism as in the case $i\in I^c$ with the flip homeomorphism exchanging the two copies of $Y$. The corresponding dual Ruelle transfer operators $\LL^*_i$ leave the two-dimensional space $V$ spanned by the measures $(\nu,0)$ and $(0,\nu)$ invariant, and in the basis given by these two measures the restrictions of $\LL^*_i$ to this subspace are represented by the matrices
$$
\LL^*_i|_V=\begin{pmatrix}2 & 0\\ 0 & 2\end{pmatrix},\ \ \text{if}\ \ i\in I^c,\ \ \text{and}\ \
\LL^*_i|_V=S=\begin{pmatrix}0 & 2\\ 2 & 0\end{pmatrix},\ \ \text{if}\ \ i\in I.
$$

Now, fix $\alpha>2$ and  define a homomorphism $N\colon\Z^n_+\to(0,+\infty)$ by letting $N(e_i)=2$ for $i\in I^c$ and $N(e_i)=\alpha$ for $i\in I$. Put $\beta=1$; note that this is precisely the critical inverse temperature~$\beta_c$ if $I\ne\{1,\dots, n\}$, and $\beta>\beta_c=\log2/\log\alpha$ if $I=\{1,\dots,n\}$. Consider the measure
$$
\mu=(1-\alpha^{-1}S)^{1-|I|}(0,\nu).
$$
In the basis $\{(\nu,0),(0,\nu)\}$ of $V$ it is represented by the vector
$$
(1-4\alpha^{-2})^{1-|I|}\begin{pmatrix}1 & 2\alpha^{-1}\\ 2\alpha^{-1} & 1\end{pmatrix}^{|I|-1}\begin{pmatrix}0\\ 1\end{pmatrix},
$$
so it is positive. We claim that $\mu$ satisfies~\eqref{subinvariance measure1} for all nonempty subsets $J\subset\{1,\dots,n\}$ different from $I$ and it does not satisfy~\eqref{subinvariance measure1} for $J=I$.

Let us start with the case $J\ne I$. Consider two subcases. Assume first that $J\not\subset I$. Then, for any $j\in J\setminus I$, the operator $1-N(e_j)^{-1}\LL^*_j=1-2^{-1}\LL^*_j$ is zero on $V$, so
$$
\prod_{i\in J}(1-N(e_i)^{-1}\LL^*_i)\mu=0.
$$
Assume now that $J\subsetneq I$. Then the measure
$$
\prod_{i\in J}(1-N(e_i)^{-1}\LL^*_i)\mu=(1-\alpha^{-1}S)^{|J|}\mu=(1-\alpha^{-1}S)^{1+|J|-|I|}(0,\nu)
$$
is represented by the vector
$$
(1-4\alpha^{-2})^{1+|J|-|I|}\begin{pmatrix}1 & 2\alpha^{-1}\\ 2\alpha^{-1} & 1\end{pmatrix}^{|I|-|J|-1}\begin{pmatrix}0\\ 1\end{pmatrix},
$$
so it is positive.

On the other hand, for $J=I$ the measure
$$
\prod_{i\in I}(1-N(e_i)^{-1}\LL^*_i)\mu=(1-\alpha^{-1}S)^{|I|}\mu=(1-\alpha^{-1}S)(0,\nu)
$$
is represented by the vector
$$
\begin{pmatrix}1 & -2\alpha^{-1}\\ -2\alpha^{-1} & 1\end{pmatrix}\begin{pmatrix}0\\ 1\end{pmatrix}=
\begin{pmatrix}-2\alpha^{-1}\\ 1\end{pmatrix},
$$
so it is not positive, proving our claim.

\smallskip

If $I\ne\{1,\dots,n\}$, then, for every $\beta\le\beta_c=1$, it is actually not difficult to describe explicitly all probability measures $\mu$ on $Z$ satisfying~\eqref{subinvariance measure1} for all nonempty $J$.  The key point is that Ruelle's Perron-Frobenius theorem (see~\cite[Theorem~1.7]{Bow}) implies that if a probability measure~$\eta$ on~$Y$ satisfies $\LL^*\eta\le c\,\eta$ for some $c\le 2$, then $c=2$ and $\eta=\nu$. Therefore if a probability measure $\mu$ on $Z$ satisfies~\eqref{subinvariance measure1} for some $\beta\le1$ and $J=\{j\}$, with $j\in I^c$, then $\beta=1$ and $\mu=(\lambda\nu,(1-\lambda)\nu)$ for some $\lambda\in[0,1]$. By the above arguments, such a measure satisfies~\eqref{subinvariance measure1} for $\beta=1$ and all nonempty $J$ if and only if $(1-\alpha^{-1}S)^{|I|}\mu$ is a positive measure, that is, the vector
$$
\begin{pmatrix}1 & -2\alpha^{-1}\\ -2\alpha^{-1} & 1\end{pmatrix}^{|I|}\begin{pmatrix}\lambda\\ 1-\lambda\end{pmatrix}
$$
has nonnegative entries. A simple computation shows that this means that
$$
\Big|\lambda-\frac{1}{2}\Big|\le\frac{1}{2}\left(\frac{\alpha-2}{\alpha+2}\right)^{|I|}.
$$

Note also that Proposition~\ref{prop:local homeomorphisms}(ii) gives a description of measures satisfying~\eqref{subinvariance measure1} for all nonempty~$J$ and any fixed $\beta>\beta_c$, but it is not \emph{very} explicit.

\smallskip

This example can be modified in several ways to satisfy additional properties. For instance, we can replace $h_i$ by $h_i^2$ for all $i\in I^c$, let $N(e_i)=4$ for such $i$'s and take $\alpha=4$, thus getting $N(e_i)=4$ for all $i=1,\dots,n$. We can also take the disjoint union of $Z$ with another compact Hausdorff space $W$ equipped with $n$ commuting surjective local homeomorphisms and consider a nontrivial convex combination of the measure $\mu$ defined above with a suitable measure on~$W$ in order to get a probability measure $\eta$ on $W\sqcup Z$ such that the measures $\prod_{i\in J}(1-N(e_i)^{-1}\LL^*_i)\eta$ are positive and \emph{nonzero} for all nonempty subsets $J\subset\{1,\dots,n\}$ different from $I$, while the measure $\prod_{i\in I}(1-N(e_i)^{-1}\LL^*_i)\eta$ is not positive.
\end{example}

\subsection{Higher rank graph \texorpdfstring{C$^*$}{C*}-algebras}

Let $k\geq 1$. Suppose that  $(\Lambda,d)$ is a $k$-graph with vertex set $\Lambda^0$ in the sense of  \cite{KP}.

For  $n\in \Z_+^k$, we write $\Lambda^n=\{\lambda\in \Lambda^*: d(\lambda)=n\}$.
A  $k$-graph is \emph{finite} if $\Lambda^n$ is finite for all $n\in \Z_+^k$.
Given $v,w\in \Lambda^0$,   $v\Lambda^n w$ denotes  $\{\lambda\in \Lambda^n: r(\lambda)=v \text { and } s(\lambda)=w\}$. We say that $\Lambda$ is \textit{row finite}  if $v\Lambda^n $ is finite all $n\in \Z_+^k$ and $v \in \Lambda^0$.
The $k$-graph $\Lambda$ has   \textit{no sources} if $v\Lambda^m\neq\emptyset$ for every $v\in \Lambda^0$ and  $m\in \Z_+^k$.
For $\mu,\nu\in \Lambda$,  we write
\[\Lambda^{\min}(\mu,\nu)=\{(\xi,\eta)\in \Lambda\times \Lambda: \mu\xi=\nu \eta \text{ and } d(\mu\xi)=d(\mu)\vee d(\nu)\}.\]
We say  that $\Lambda$ is \textit{finitely aligned} if $\Lambda^{\min}(\mu,\nu)$ is finite for all $\mu,\nu\in \Lambda$.

Following \cite{RS}, for a finitely aligned $k$-graph $\Lambda$, a \textit{Toeplitz-Cuntz-Krieger $\Lambda$-family} in a C$^*$-algebra $B$ is a set  of partial isometries $\{T_\lambda: \lambda\in \Lambda\}$ such that
    \begin{enumerate}
        \item $\{T_v: v\in \Lambda^0\}$ is a collection  of mutually orthogonal projections,
        \item $T_\lambda T_\mu=T_{\lambda \mu}$ if $s(\lambda)=r(\mu)$,
        \item $T_\mu^*T_\nu=\sum_{(\xi,\eta)\in\Lambda^{\min}(\mu,\nu)} T_\xi T_\eta ^*$ for all $\mu,\nu\in \Lambda$.
    \end{enumerate}
The Toeplitz algebra $\TT C^*(\Lambda)$ is  generated by a universal Toeplitz-Cuntz-Krieger $\Lambda$-family $\{t_\lambda: \lambda\in \Lambda\}$.

Given a finitely aligned $k$-graph,  \cite[Corrolary~7.5]{RS} shows that  $\TT C^*(\Lambda)$ is the Nica-Toeplitz algebra of a compactly aligned product system $X(\Lambda)$ of C$^*$-correspondences over $\Z^k_+$ constructed as follows.
The coefficient algebra is $C_0(\Lambda^0)$. For each $n\in \Z_+^k$, the set
 $C_c(\Lambda^n)=\lsp\{\delta_\lambda:\lambda\in \Lambda^n\}$ is a right $C_0(\Lambda^0)$-module  with the module structure
$ (\xi\cdot a) (\lambda) = \xi(\lambda)a(s(\lambda))$. Then $X(\Lambda)_n$ is the completion of $ C_c(\Lambda^n)$ in the norm arising from the $C_0(\Lambda^0)$-valued inner product  $\langle \xi,\zeta\rangle(v)=\sum_{\mu\in \Lambda^n v}\overline{\xi(\mu)}\zeta(\mu)$.
Now $X(\Lambda)_n$ becomes a C$^*$-correspondence over $ C_0(\Lambda^0)$, with the left action given by $(a\cdot\xi)(\lambda) =a(r(\lambda))\xi(\lambda)$.  The multiplication in $X(\Lambda)$ is given by $\delta_\lambda \delta_\mu=\delta_{\lambda \mu}$ if $r(\mu)=s(\lambda)$ and zero otherwise
(see \cite[Proposition~3.2]{RS}). The isomorphism  $\omega:\TT C^*(\Lambda)\rightarrow \NT(X(\Lambda))$ is given by $\omega(t_\lambda)=i_{X(\Lambda)}(\delta_\lambda)$.

\begin{proposition}\label{prop:k-graph}
Let  $\Lambda$ be a finitely aligned  $k$-graph. For each $1 \leq i \leq k$, let  $A_i\in M_{\Lambda^0}(\Z_+)$ be the matrix with entries $A_i(v,w)=|v\Lambda^{e_i}w|$. Consider a homomorphism $N\colon\Z^k_+\to(0,+\infty)$ and the dynamics $\sigma$ on   $\TT C^*(\Lambda)$ defined by
$\sigma_t(t_\lambda)=N(e_j)^{it}t_\lambda$ for $\lambda\in\Lambda^{e_j}$. For every state $\phi$ on $\TT C^*(\Lambda)$ define a vector $V_{\phi}=(\phi(t_v))_{v\in\Lambda^0}$. Then
\begin{itemize}
        \item[(i)]  for every $\beta\in\R$, the map $\phi\mapsto V_\phi$ defines a one-to-one correspondence between the gauge-invariant $\sigma$-KMS$_\beta$-states on $\TT C^*(\Lambda)$ and the vectors $V\in [0,1]^{\Lambda^0}$  such that  $\|V\|_1=1$ and
\begin{equation*}\label{subinvariance vector1}
\prod_{i\in J}(1-N(e_i)^{-\beta }A_i)V\geq 0
\end{equation*}
for every nonempty subset $ J\subset \{1,\dots, k\}$.
\end{itemize}

Assume in addition that $\Lambda$ is finite, $N(e_i)>1$ for all $i$, and put
$$
\beta_c=\max_{1\le i\le n}\frac{\log r(A_i)}{\log N(e_i)},
$$
where $r(A_i)$ is the spectral radius of the matrix $A_i$. Then
\begin{itemize}
\item[(ii)]  for every $\beta>\beta_c$, there is an affine bijection between the $\sigma$-KMS$_\beta$-states on $\TT C^*(\Lambda)$ and the set
\[\Sigma_\beta=\big\{W: W\in [0,1]^{\Lambda^0} \text{  and } \big\| \sum_{n\in \Z^k_+}N(p)^{-\beta}A^n W\big\|_1=1\big\},\]
where $A^n=A^{n_1}_1\dots A^{n_k}_k$, or equivalently, $A^n(v,w)=|v\Lambda^nw|$.
\end{itemize}
\end{proposition}

\begin{proof}
Working with $\NT(X(\Lambda))$ instead of $\TT C^*(\Lambda)$, let us first of all compute the operators~$F_n$ on $C_0(\Lambda^0)^*=\ell^1(\Lambda^0)$.

A straightforward computation using the formulas for the right action and the inner product for the module $X(\Lambda)_n$ shows that the set $\{\delta_\lambda:\lambda\in \Lambda^n\}$ forms a Parseval frame for $X(\Lambda)_n$.  Hence, for each $v\in\Lambda^0$, applying \eqref{eq:Tr Parseval frame} we get
\begin{align*}
\Tr_{\tau}^{n}(\delta_v)
= \sum_{\lambda\in \Lambda^n}\tau(\langle \delta_\lambda,\delta_v\cdot \delta_\lambda\rangle).
\end{align*}
Since $\delta_v\cdot\delta_\lambda=\delta_{v,r(\lambda)}\delta_\lambda$ and $\langle \delta_\lambda,\delta_\lambda\rangle=\delta_{s(\lambda)}$, we see that the above expression equals
\begin{align*}
\sum_{\lambda\in v \Lambda^n}\tau(\langle \delta_\lambda,\delta_\lambda\rangle)= \sum_{\lambda\in v \Lambda^n}\tau(\delta_{s(\lambda)})=\sum_{w\in \Lambda^0}|v\Lambda^nw|\tau(\delta_w).
\end{align*}
In other words, if $V=(\tau(\delta_v))_{v\in\Lambda^0}$ is the vector defined by $\tau$, then the vector defined by~$F_n\tau$ is~$A^n V$.

Similarly to Proposition~\ref{prop:local homeomorphisms}, the result now follows from Corollary~\ref{cor:classfree}, identity~\eqref{eq:critical-beta} and Corollary~\ref{cor:finite-class3}.
\end{proof}

For finite $k$-graphs part (i) of the above proposition recovers \cite[Proposition~4.3]{C} that was proved using the groupoid picture. Together with the theory developed in Section~\ref{sec:products}, leading to Example~\ref{ex:kak}, this provides an alternative route to the classification of gauge-invariant KMS-states in~\cite[Theorem~5.9]{C}. Part (ii) recovers \cite[Theorem~6.1(a)]{aHLRS}.

\smallskip

Following \cite{KP}, for a row finite $k$-graph  $\Lambda$ with no sources, we say that a Toeplitz-Cuntz-Krieger $\Lambda$-family in a C$^*$-algebra $B$  is a \textit{Cuntz-Krieger $\Lambda$-family} if it satisfies
    \begin{itemize}
     \item[(4)] $T_v = \sum_{\lambda\in v \Lambda^n}T_\lambda T_\lambda^*$  for all $v\in\Lambda^0$ and $n\in \Z^k_+$.
    \end{itemize}
The Cuntz-Krieger algebra  $C^*(\Lambda)$ is the  quotient of $\TT C^*(\Lambda)$ by the ideal $\big\langle t_{v}- \sum_{\lambda\in v\Lambda^n}t_\lambda t_\lambda^*:v\in \Lambda^0\big\rangle.$ By \cite[ Corollary~4.4]{RS}, $C^*(\Lambda)$ coincides with $\OO_{X(\Lambda)}$.
Since $\Lambda$ is row finite and has no  sources, the left action  $C_0(\Lambda)$  on each fiber is by compact operators and is injective~\cite{FR}. Thus, by Remark~\ref{re:cuntzalg}, $C^*(\Lambda)$ is the same as $\NOs(X(\Lambda))$.
By Corollary~\ref{cor:cuntzalg} we now get the following result.

\begin{proposition}\label{prop:k-graph2}
Let  $\Lambda$ be a row finite  $k$-graph with no sources, $N\colon\Z^k_+\to(0,+\infty)$ be a homomorphism and $\sigma$ be the dynamics on   $C^*(\Lambda)$ defined by $\sigma_t(t_\lambda)=N(e_j)^{it}t_\lambda$ for $\lambda\in\Lambda^{e_j}$. Then, for every $\beta\in\R$, we have an affine one-to-one correspondence between the gauge-invariant  $\sigma$-KMS$_\beta$-states on $C^*(\Lambda)$ and the vectors $V\in [0,1]^{\Lambda^0}$  such that $\|V\|_1=1$ and
\begin{equation*}
N(e_i)^{-\beta } A _i V=V\ \ \text{for}\ \  i=1,\dots,k.
\end{equation*}
\end{proposition}

For finite $k$-graphs with no sources this recovers~\cite[Theorem~7.4]{C0}.

\bigskip\bigskip

\end{document}